\documentclass[13pt]{amsart}
\usepackage{amsmath}
\usepackage{}
\usepackage{amsfonts}
\usepackage{amsfonts,latexsym,rawfonts,amsmath,amssymb,amsthm}
\usepackage[plainpages=false]{hyperref}
\usepackage{amsfonts}
\usepackage{bbm}
\usepackage{amscd}
\usepackage{mathrsfs}
\usepackage{latexsym,amssymb,amsmath,amscd,amscd,amsthm,amsxtra,xypic}

\usepackage{graphicx}
\usepackage{color}
\usepackage{tikz}

\usepackage{pdfsync}

\usepackage[a4paper]{geometry}
\usepackage{esint}
\usepackage{graphics,graphicx}
\usepackage{tikz}

\numberwithin{equation}{section}

\newcommand{\beq}{\begin{equation}}
\newcommand{\eeq}{\end{equation}}
\newcommand{\beqs}{\begin{eqnarray*}}
\newcommand{\eeqs}{\end{eqnarray*}}
\newcommand{\beqn}{\begin{eqnarray}}
\newcommand{\eeqn}{\end{eqnarray}}
\newcommand{\beqa}{\begin{array}}
\newcommand{\eeqa}{\end{array}}

\def\lra{\longrightarrow}

\def\bc{\begin{center}}
\def\ec{\end{center}}

\def\begeq{\begin{equation}}
\def\endeq{\end{equation}}
\def\and{\quad{\rm and}\quad}

\let\lra=\longrightarrow

\def\mapright\#1{\,\smash{\mathop{\lra}\limits^{\#1}}\,}

\newtheorem{prop}{Proposition}[section]
\newtheorem{theo}[prop]{Theorem}
\newtheorem{lem}[prop]{Lemma}

\newtheorem{cor}[prop]{Corollary}
\newtheorem{rem}[prop]{Remark}

\newtheorem{defi}[prop]{Definition}

\begin{document}

\title{Uniform K-stability of $G$-varieties of complexity 1}


\author[Yan Li]{Yan Li$^{*1}$}
\author[Zhenye Li]{Zhenye Li$^{*2}$}

\address{$^{*1}$School of Mathematics and Statistics, Beijing Institute of Technology, Beijing, 100081, China.}
\address{$^{*2}$College of Mathematics and Physics, Beijing University of Chemical Technology, Beijing, 100029, China.}
\email{liyan.kitai@yandex.ru,\ \ \ lizhenye@pku.edu.cn}

\thanks {$^{*1}$Partially supported by NSFC Grant 12101043 and the Beijing Institute of Technology Research Fund Program for Young Scholars.}
\thanks {$^{*2}$Partially supported by NSFC Grant 12001032.}



\begin{abstract}
Let ${\rm k}$ be an algebraically closed field of characteristic 0 and $G$ a connected, reductive, linear algebraic group of simply connected type over ${\rm k}$. Let $X$ be a projective $G$-variety of complexity 1. We classify $G$-equivariant normal test configurations of $X$ with integral central fibre via the combinatorial data. We also give a formula of anti-canonical divisors on $X$. Based on this formula, when $X$ is $\mathbb Q$-Fano, we give an expression of the Futaki invariant, and derive a criterion of uniform K-stability in terms of the combinatorial data.
\end{abstract}
\maketitle

\section{Introduction}

The famous Yau-Tian-Donaldson conjecture asserts that the existence of K\"ahler-Einstein metrics on a Fano manifold is equivalent to the K-stability. The conjecture has been proved by Tian \cite{Tian-15} and Chen-Donaldson-Sun \cite{CDS} alternatively. In recent years, the conjecture has been widely extended to other canonical metrics (such as solitons, cscK metrics on polarized vareities), log pairs, and a uniform version for singular $\mathbb Q$-Fano varieties, see \cite{Boucksom-Hisamoto-Jonsson, Li19, Xu-EMS-survey-2021}, etc.

The notion of K-stability was first introduced by Tian \cite{Tian-97} in terms of special degenerations and then reformulated by Donaldson \cite{Do} in an algebro-geometric way via test-configurations. In general, to test K-stability one has to study infinitely many possible degenerations. A natural question is how to reduce it to a finitely dimensional process. The reduction would be possible when considering varieties with large group actions and the respective equivariant K-stability.

Let ${\rm k}$ be an algebraically closed field of characteristic 0 and $G$ a connected, reductive, linear algebraic group over ${\rm k}$. Let $X$ be a normal variety with a regular $G$-action. Fix any Borel subgroup $B$ of $G$. The \emph{complexity} of the $G$-action on $X$ is the codimension of a $B$-orbit in general position (cf. \cite{Vinberg-86}). The varieties of complexity 0 are called \emph{spherical varieties}, which include the well-known toric varieties. The geometry of spherical varieties has been comprehensively studied by various of authors, we refer to \cite{Luna-Vust, Brion-Luna-Vust, Knop91, Timashev-book}, and references therein for more knowledge. The spherical variety has a nice combinatorial description which brings great convenience to the study of its geometric structure. Roughly speaking, given a spherical variety $X$, one may begin with the lattice $\Gamma$ of $B$-semiinvariant rational functions\footnote{More precisely, the lattice of $B$-weights of semiinvariant functions. Note that on a spherical variety, up to multiplication by a non-zero constant, a $B$-semiinvariant function is uniquely determined by its $B$-weight.}, which is contained in the lattice $\mathfrak X(B)$ of $B$-weights, and its ($\mathbb Q$-)dual $\Gamma_\mathbb Q^*$ in which the set of $G$-valuations (called its valuation cone, which is a solid cosimplicial cone) is embedded. The variety $X$ itself is fully characterised by a fan of convex cones (more precisely, coloured cones, see \cite{Luna-Vust,Knop91,Timashev-book} for details), while line bundles are described by piecewise linear functions on this fan (cf. \cite{Brion89,Timashev-2000}). In particular, when $X$ is polarized by some ample line bundle $L$, there is a nice moment polytope in $\Gamma_\mathbb R:=\Gamma\otimes_\mathbb Q\mathbb R$ that encodes the structure of $G$-modules on the spaces of sections ${\rm H}^0(X,L^k)$ for $k\in\mathbb N$. In terms of these data, in recent years several combinatorial criterions of K-stability of spherical varieties were established, which give a practical way to test K-stability and provide rich examples. For instance, the K-stability criterion for toric varieties was studied by Donaldson \cite{Do}, for spherical varieties established by Delcroix \cite{Del-2021,Del-2020-09}, etc. It turns out that in these cases K-stability is characterised by certain barycenter of the moment polytope.

It is then natural to consider $G$-varieties of complexity 1. In fact, many examples in the literatures belong to this class, see for instance \cite{Mukai-Umemura-1983,Moser-Jauslin-87}. The classification theory of $G$-varieties of complexity 1 was established by Timash\"ev \cite{Timashev-1997}. In particular \cite{Timashev-1997} gives a combinatorial description of these varieties analogous to the spherical cases. Unlike the spherical case, a $B$-semiinvariant rational function can not be completely determined only by its weight. In this case we need to consider the hyperspace instead of $\Gamma_\mathbb Q^*$ in which the $G$-valuations are embedded. Roughly speaking, the hyperspace is a family of half-spaces $\{\mathscr Q_{x,+}|x\in C\}$ parameterized by points of a smooth projective curve $C$, each isomorphic to $\Gamma_\mathbb Q^*\times \mathbb Q_+$, that glued together along common boundary $\mathscr Q=\Gamma_\mathbb Q^*\times\{0\}$. The valuation cone is obtained by gluing a collection of solid convex cones in $\mathscr Q_{x,+}$ along common face in $\mathscr Q$. Varieties with given field of rational functions are classified by fans of coloured cones and hypercones (a hypercone is again a family of cones that glued along common boundary). The combinatorial theory of line bundles on $G$-varieties of complexity 1 was then studied in \cite{Timashev-2000}. Line bundles are described by collection of linear functionals defined on $\mathscr Q_{x,+}$ with common restriction on $\mathscr Q$. See Section 2.2 for details. For $T$-varieties of complexity 1 with $T$ an algebraic torus, Ilten-S\"u\ss\,\cite{Ilten-Suss-2011} developed the classification theory of polarized $T$-varieties of complexity 1. They introduced the ``divisorial polytope" which turns out to be a very useful tool in the study of polarized $T$-varieties of complexity 1. Based on this theory, they obtained a combinatorial criterion of equivariantly K-(poly)stability of Fano $T$-varieties of complexity 1 in \cite{Ilten-Suss-Duke}. Recently Rogers \cite{Rogers-2022} studied K-(poly)stability of smooth Fano ${\rm SL_2}$-varieties of complexity 1.

Let $G$ be an arbitrary connected, reductive, linear algebraic group of simply connected type, and $X$ a projective $G$-variety of complexity 1. Motivated by the works cited above, in this paper we give a criterion of $G$-equivariantly uniform K-stability in terms of its combinatorial data. To this end, we will first classify $G$-equivariant normal test configuration of $X$ with integral  central fibre. This class in particular includes the special test configurations, which are enough for testing K-stability of a $\mathbb Q$-Fano variety (cf. \cite{Li-Xu-Annl,Li19}). We call $X$ a $\mathbb Q$-Fano variety if $X$ is $\mathbb Q$-Gorenstein and $K_X^{-1}$ is ample. When $X$ is $\mathbb Q$-Fano, we then consider a family of polytopes $\{\Delta_x^O(K_X^{-1})\subset\mathfrak X_\mathbb R(B)\times\mathbb R|x\in C\}$ that encodes the information of K-stability of $X$, where $C\cong\mathbb P^1$ is birational to the rational quotient for the $B$-action. These polytopes were first introduced in \cite{Ilten-Suss-Duke} to characterize K-stability of $T$-varieties of complexity 1. Based on a formula of $B$-stable anti-canonical divisors of $X$ derived in Section 3, it turns out that the Futaki invariant can be expressed in a very simple way on this family of polytopes. Finally we prove a combinatorial criterion of K-stability if $X$ in addition has at most klt singularities. Let $\mathscr V$ be the valuation cone of $X$ and $\mathscr V_x$ its intersection with the hyperspace over $x\in C$, $\mathscr A_x\subset\mathscr V_x$ the space parameterizing one-parameter subgroups in the $G$-equivariant automorphism group ${\rm Aut}_G(X)$ of $X$ (see Section 4.2.1 below). Denote by $\kappa_P$ the sum of unipotent roots of the associated parabolic subgroup $P$ of $X$  (see Section 2.3.2 below). Then each $\Delta_{x}^O(K_X^{-1})$ is in fact a solid convex polytope in $(\Gamma_\mathbb R+\kappa_P)\times\mathbb R$. Also set
$${\mathbf b}(\Delta_{x}^O(K_X^{-1})):=\frac1V\int_{\Delta_{x}^O(K_X^{-1})}(\lambda,t)\pi(\lambda)d\lambda\wedge dt$$
the barycenter of a polytope $\Delta_{x}^O(K_X^{-1})(\subset(\Gamma_\mathbb R+\kappa_P)\times\mathbb R)$, where $(\lambda,t)$ are coordinates on $\Delta_{x}^O(K_X^{-1})$, $\pi(\lambda)$ is a non-negative polynomial function depends on $X$, $d\lambda$ the standard Lebesgue measure on $\Gamma_\mathbb R+\kappa_P$ which is normalized by the lattice $\Gamma$, and $$V=\int_{\Delta_{x}^O(K_X^{-1})}\pi(\lambda)d\lambda\wedge dt$$ is the volume of $\Delta_{x}^O(K_X^{-1})$ which is in fact independent of $x$ (see Section 5.2 below for details). We have the following
\begin{theo}\label{K-st-thm}
Let $X$ be a projective, $\mathbb Q$-Fano $G$-varieties of complexity 1 with klt singularities, which satisfies ${\rm k}(X)^B\cong{\rm k}(\mathbb P^1)$. Then $X$ is $G$-equivariantly K-semistable if and only if
\begin{align}\label{K-ss-eq}
\kappa_P-\mathbf{b}(\Delta_{x}^O(K_X^{-1}))\in(\mathscr V_x)^\vee,~\forall x\in C,
\end{align}
where $(\mathscr V_x)^\vee$ is the dual cone of $\mathscr V_x$. Moreover, the following conditions are equivalent:
\begin{itemize}
\item[(1)] It holds \eqref{K-ss-eq} and
\begin{align}\label{K-ps-eq}
(\kappa_P-\mathbf{b}(\Delta_{x}^O(K_X^{-1})))^\perp\cap \mathscr V_x=\mathscr A_x,~\forall x\in C;
\end{align}
\item[(2)] $X$ is $G$-equivariantly K-polystable;
\item[(3)] $X$ is ${\mathbf G}$-uniformly K-stable, provided 
$X$ is not a $G\times{\rm k}^\times$-spherical variety. 
Here ${\mathbf G}$ is the subgroup of the automorphism group ${\rm Aut}(X)$ generated by $G$ and a torus ${\mathbf T}$ contained in the central automorphism group of $X$ whose Lie algebra is the linear part of $\mathscr V$ (see Section 2.2.3 below).
\end{itemize}
\end{theo}
Precise definitions and useful properties of the terms involved here can be found in the preliminary section and indicated above places. By the famous Yau-Tian-Donaldson conjecture \cite{Tian-15,Li19}, when ${\rm k}=\mathbb C$ the above K-stability criterion then provides an existence criterion of K\"ahler-Einstein metrics.

\begin{rem}
Note that when $X$ is a $\mathbb Q$-Fano $G\times{\rm k}^\times$-spherical variety, the K-stability can be tested by the combinatorial criterion of Delcroix \cite[Theorem A]{Del-2021}.
\end{rem}

As mentioned above, we need to derive a formula of $B$-stable anti-canonical divisors which helps to simplify the Futaki invariant. The following is a counterpart of a  result of Brion \cite{Brion97} for spherical varieties.
\begin{theo}\label{anti-can-div-thm-in-all}
Let $X$ be a $\mathbb Q$-Gorenstein, projective $G$-variety of complexity 1, and $\mathscr B(X)$ the set of its $B$-stable prime divisors. Denote by $P\subset G$ its associated parabolic subgroup. Then there is a $B$-stable anti-canonical $\mathbb Q$-divisor of $X$,
\begin{align}\label{anti-can-div-in-all}
\mathfrak d=\sum_{D\in\mathscr B(X), h_D\not=0}(1-h_D+h_Da_{x_D})D+\sum_{D\in\mathscr B(X)^G, h_D=0}D+\sum_{D\in\mathscr D^B, h_D=0}\bar m_DD,
\end{align}
where $v_D=h_Dq_{x_D}+\ell_D\in\mathscr Q_{x_D,+}$ is the valuation defined by $D$, $\mathfrak a:=\sum_{x\in\mathbb P^1}a_x[x]$ is an anti-canonical $\mathbb Q$-divisor of $C$. The coefficient $\bar m_D$ of each colour $D$ with $h_D=0$ is explicitly defined in \eqref{coe-colour-one-para} and \eqref{bar-m-D-quasi}. Moreover, if $m\mathfrak d$ is a Cartier divisor for some $m\in\mathbb N_+$, then $m\mathfrak d$ is the divisor of a $B$-semiinvariant rational section of $K_X^{-m}$ with weight $m\kappa_P$.
\end{theo}
As will be seen below, there are two classes of $G$-varieties of complexity one, the one-parameter and quasihomogeneous ones. For this reason Theorem \ref{anti-can-div-thm-in-all} will be proved separately in Sections 3.1 and 3.2 (under slightly weaker assumptions, respectively). In fact the proof in the quasihomogeneous case relies on that in the one-parameter case. We would also like to briefly explain here the relation between Theorem \ref{anti-can-div-thm-in-all} with results derived by different authors in the literatures. The $T$-varieties, or more general, horospherical varieties of complexity 1 are of one-parameter type. There such a formula has been proved by Petersen-S\"uss \cite{Petersen-Suss-2011} and Langlois-Terpereau \cite{Langlois-Terpereau-2016}, respectively. Later Langlois \cite{Langlois} proved a general formula, which in particular include the one-parameter case here, in a different framework. However, up to the authors' knowledge, there is not a unified formula like \eqref{anti-can-div-in-all} for the quasihomogeneous cases in the literatures.

Concerning applications of the above results, we give examples in Section 6. In Section 6.1 we illustrate our frame on the well-known Mukai-Umemura threefold, which is a ${\rm SL}_2$-variety of complexity 1. Then in Section 6.2 we apply the above results to the following example
$$X = \{(p_1, p_2, p_3, l_1, l_2, l_3)|p_j\in \mathbb P^2,~l_i \in \mathbb P^{2^*},~p_j \in l_i~\text{whenever}~i\not=j\},$$
which can be realized as a ${\rm SL}_3$-variety of complexity 1 (see Section 6.2 for details on group actions). We will prove
\begin{prop}
The  ${\rm SL}_3$-variety $X$ of complexity 1 defined above is a $\mathbb Q$-Fano variety with klt singularities, and is ${\mathbf G}$-uniformly K-stable.
\end{prop}

The paper is organized as following: In Section 2 we give preliminaries. In Section 3 we derive a formula of anti-canonical divisors on $X$. This formula will be crucial when simplifying the expression of Futaki invariants later. In Section 4 we classify $G$-equivariant normal test configuration of $X$ with integral central fibre. We also study the central fibre and give combinatorial parameterizations of ${\rm Aut}_G(X)$ via studying the associated filtration in Section 4.2. In Section 5 we give a formula of Futaki invariant in terms of the combinatorial data, based on the formula of total weights. In particular we simplify the formula on $\mathbb Q$-Fano varieties. In Section 5.4 we prove Theorem \ref{K-st-thm}. In Section 6 we give two examples where Theorem \ref{K-st-thm} is applied. In the Appendix we collect useful Lemmas, and give an alternative computation of the Futaki invariant for $\mathbb Q$-Fano varieties via a formula of intersection numbers. Also some discussions on $\Delta_{x}^O(K_X^{-1})$'s are included by the end of the Appendix.

\subsection*{Notations and conventions}
In this paper we use the following conventions:
\begin{itemize}
\item ${\rm k}$-an algebraically closed field of characteristic 0;
\item $G$-a connected, reductive, linear algebraic group over ${\rm k}$. $G$ is assumed to be of simply connected type from Section 2.3;
\item $\Phi_G$-root system of $G$;
\item $\Pi_G$-the set of positive roots in $\Phi_G$ of some group $G$ with respect to certain Borel subgroup;
\item $\rho:=\frac12\sum_{\alpha\in\Pi_G}\alpha$-half the sum of positive roots of $G$;
\item $P=L_P\cdot P_u$-Levi decomposition of a parabolic subgroup $P\subset G$, where $L_P$ is a Levi subgroup of $P$ and $P_u$ the unipotent radical of $P$;
\item $\Pi_{L_P}\subset\Pi_G$-positive roots of $L_P$ in $\Pi_G$;
\item $\Pi_{P_u}\subset\Pi_G$-the unipotent roots of a parabolic subgroup $P$;
\item $\mathfrak X(\cdot)$-characters of a group;
\item $V^H$-the subset of all $H$-invariants in a linear $H$-representation $V$;
\item $V^{(H)}_\chi$-the subset of all $H$-semiinvariants with weight $\chi$ in a linear $H$-representation $V$;
\item $\langle S\rangle$-linear span of a set $S$ in a linear space;
\item $\langle v,w^*\rangle$-linear pairing of $v\in V$ and $w^*\in V^*$, where $V$ is a linear space;
\item ${\rm Conv}(S)$-convex hull of a set $S$;
\item $S^\vee$-dual cone of a set $S$ in $V^*$, where $V$ is a linear space containing $S$;
\item $S^\perp$-elements in $V^*$ that is orthogonal to $S$, where $V$ is a linear space containing $S$;
\item ${\rm Stab}_G(S)$-the group of normalizer in $G$ of a set $S$. That is, the set of elements in $g$ that keep the set $S$ invariant. If $S=\{D\}$ for a single divisor $D$, we write ${\rm Stab}_G(D)$ in short of ${\rm Stab}_G(\{D\})$;
\item $K$-a function field which admits a $G$-action. We assume in this paper that $K^B\cong {\rm k}(\mathbb P^1)$ unless otherwise stated;
\item $X$-a $G$-model of $K$, see Section 2.2; 
\item $\kappa_P$-sum of unipotent roots of the associated parabolic subgroup of $X$ (see Section 2.3 for this parabolic subgroup). It is also the $B$-weight of a canonically chosen rational section of $K_X^{-1}$, see Theorems \ref{anti-can-div-thm} and \ref{anti-can-div-quasi-homo} below;
\item $L^n$-the $n$-th tensor power of a line bundle $L$;
\item $L^{\cdot k}$-the $k$-th self-intersection of a line bundle $L$;
\item If $\mathfrak M$ is a lattice, then we denote $\mathfrak M_\mathbb Q=\mathfrak M\otimes_\mathbb Z\mathbb Q$ and $\mathfrak M_\mathbb R=\mathfrak M\otimes_\mathbb Q\mathbb R$;
\item $d\lambda$-the standard Lebesgue measure on $\Gamma_\mathbb R$ or its translation $\Gamma_\mathbb R+\lambda_0$ for some $\lambda_0\in\mathfrak X(B)$, which is normalized by $\Gamma$.
\end{itemize}

\section{Preliminaries}

\subsection{Test configurations and K-stability}
Let $X$ be a projective variety. Then the pair $(X,L)$ with $L$ an ample line bundle on $X$ is called a \emph{polarized variety} (polarized by the line bundle $L$). The notation of K-stability of a polarized variety are usually stated in terms of test configurations.
\begin{defi}
Let $(X,L)$ be a polarized variety. A test configuration of $X$ consists of the following data:
\begin{itemize}
\item A scheme $\mathcal X$ with a ${\rm k}^\times$-action;
\item A ${\rm k}^\times$-equivariant flat and proper morphism ${\rm pr}:\mathcal X\to{\rm k}$ so that $\mathcal X\setminus{\rm pr}^{-1}(0)$ is ${\rm k}^\times$-equivariantly isomorphic to $X\times{\rm k}^\times$.
\end{itemize}
A (semi-)test configuration of $(X,L)$ consists of a test configuration $\mathcal X$ and a ${\rm k}^\times$-linearized $\mathbb Q$-line bundle $\mathcal L$ on $\mathcal X$ so that:
\begin{itemize}
\item $\mathcal L$ is relative (semi-)ample on $\mathcal X$ with respect to the morphism ${\rm pr}$;
\item With respect to the ${\rm k}^\times$-action on $\mathcal X$, $(\mathcal X,\mathcal L)|_{{\rm pr}^{-1}({\rm k}^\times)}$ is ${\rm k}^\times$-equivariantly isomorphic to $(X,L^{r_0})\times{\rm k}^\times$ for some fix $r_0\in\mathbb N_+$ $($called the
index of $(\mathcal X,\mathcal L)$$)$.
\end{itemize}
A test configuration $(\mathcal X,\mathcal L)$ is called
\begin{itemize}
\item a special test configuration if its central fibre $\mathcal X_0:={\rm pr}^{-1}(0)$ is normal;
\item a product test configuration if there is a ${\rm k}^\times$-equivariant isomorphism $(\mathcal X,\mathcal L)\cong(X,L ^{r_0})\times{\rm k}$, and the ${\rm k}^\times$-action on the right-hand side is given diagonally by a ${\rm k}^\times$-action on $(X,L)$ with the standard multiplication on ${\rm k}$.
\end{itemize}
Assume further that a connected, reductive group ${\mathbf G}$ acts on $(X,L)$. A test configuration $(\mathcal X,\mathcal L)$ is called ${\mathbf G}$-equivariant if ${\mathbf G}$ acts on $(\mathcal X,\mathcal L)$, commutes with the ${\rm k}^\times$-action of the test configuration, and the ${\mathbf G}$-action on $(\mathcal X,\mathcal L)\times_{\rm k}{\rm k}^\times\cong(X,L^{r_0})\times{\rm k}^\times$ coincides with the ${\mathbf G}$-action on (the first factor of) $(X,L^{r_0})\times{\rm k}^\times$. 
\end{defi}

There is also a geometric way to construct test configuration of $(X,L)$ (cf. \cite[Section 2]{Do}, \cite{Ross-Thomas} or \cite[Section 2.3]{Boucksom-Hisamoto-Jonsson}). Suppose that there is a Kodaira embedding of $X$ by $|L ^{r_0}|$ for some $r_0\in\mathbb N_+$,
$$i:X\to\mathbb P({\rm H}^0(X,L ^{r_0}))=:\mathbb P^{N-1}.$$
Choose a vector $\Lambda\in\mathfrak{psl}_N$ so that $\Lambda$ generates a rank $1$ torus of ${\rm PSL}_N$. Then it defines a test configuration $(\mathcal X,\mathcal L)$ via
$$\mathcal X_{t}:=\exp(z\Lambda)\cdot i(X),~t=e^z\in{\rm k}^\times,$$
and
$$\mathcal L|_{\mathcal X_t}:=\mathcal O_{\mathbb P^{N-1}}(1)|_{\mathcal X_t}.$$
Also define $\mathcal X_0:=\lim_{t\to0}\mathcal X_t$ in the sense of the flat limit of $\mathcal X_t$ as $t\to0$, and $\mathcal L_0:=\mathcal L|_{\mathcal X_0}$. Moreover, the total space of $(\mathcal X,\mathcal L)$ can be interpreted as the Zariski closure in ${\mathbb P}^{N-1}\times{\rm k}$ of the image of the closed embedding $X\times{\rm k}^\times\to\mathbb P^{N-1}\times{\rm k}$ mapping $(x,t)$ to $(\exp(t\Lambda)\cdot i(x),t)$. Indeed, any test configuration can be realized in this way (cf. \cite[Section 2.3]{Boucksom-Hisamoto-Jonsson}). Product test configurations are precisely those with $\Lambda\in\mathfrak{aut}(X)$.

For our later use, in the following we compactify $(\mathcal X,\mathcal L)$ to a family over $(\mathcal X,\mathcal L)\overset{{\rm pr}}{\to}\mathbb P^1$ by adding a trivial fibre $(X,L^{r_0})$ at $\infty\in\mathbb P^1$. Alternatively, we glue $(\mathcal X,\mathcal L)$ with $(X,L ^{r_0})\times{\rm k}\cong(X,L ^{r_0})\times({\rm k}^\times\cup\{\infty\})$ along the common part $(X,L ^{r_0})\times{\rm k}^\times$. Also, if $\mathcal L$ is relative ample with respect to ${\rm pr}$, then by replacing $\mathcal L$ by $\mathcal L+c{\rm pr}^*\mathscr O_{\mathbb P^1}(1)$ with sufficiently large $c\gg1$, we can always assume that $\mathcal L$ is ample (cf. \cite[Proposition 1.7.10]{Lazarsfield-1}). From now on, by $(\mathcal X,\mathcal L)$ we always refer to the compactified family with ample $\mathcal L$. 
Finally, in the following, we always assume in addition that any test configurations under consideration is \emph{normal}. That is, its total space $\mathcal X$ is a normal variety. Note that by Hironaka's lemma (cf. \cite[Chapter \uppercase\expandafter{\romannumeral3}, Lemma 9.12]{GTM52}), a special test configuration is always normal.

Without loss of generality we assume that $r_0=1$ and denote $d_k:=\dim{\rm H}^0(X,L^k)$. Otherwise, we replace $L$ by $L ^{r_0}$. Choose a basis $\{e_p\}_{p=1}^{d_1}$ of ${\rm H}^0(\mathcal X_0,\mathcal L_0)$ so that each $e_p$ is an eigenvector of the $\exp(t\Lambda)$-action. Denote by $\{e^{\Lambda_p^k}\}_{p=1}^{d_k}$ the eigenvalues of the canonical lifting of the $\exp(t\Lambda)$-action on ${\rm H}^0(\mathcal X_0,\mathcal L^k_0)$. Here we use the fact that $d_k=\dim{\rm H}^0(\mathcal X_0,\mathcal L_0^k)$. Define the total weight
\begin{align*}
w_k(\mathcal X,\mathcal L):=&\sum_{p=1}^{d_k}\Lambda_p^k.
\end{align*}
It is showed in \cite[Section 2.2]{Do} that there are constants $A,B,C,D$ so that
$$w_k(\mathcal X,\mathcal L)=Ak^{n+1}+Bk^{n}+O(k^{n-1}),~k\to+\infty,$$
and
$$\dim{\rm H}^0(X,L^k)=Ck^n+Dk^{n-1}+O(k^{n-2}),~k\to+\infty.$$
The Futaki invariant of $(\mathcal X,\mathcal L)$ is then defined as
\begin{align}\label{Fut-def}
{\rm Fut}(\mathcal X,\mathcal L):=\frac{AD-BC}{C^2}.
\end{align}
The K-stability is defined as
\begin{defi}\label{K-stab-def-Fut}
We say that a polarized ${\mathbf G}$-variety $(X,L)$ is ${\mathbf G}$-equivariantly K-semistable if the Futaki invariant for any ${\mathbf G}$-equivariant test configuration is nonnegative, and is ${\mathbf G}$-equivariantly K-polystable if in addition the Futaki invariant vanishes precisely on product ${\mathbf G}$-equivariant test configurations. When $X$ is not ${\mathbf G}$-equivariantly K-semistable, we say it is K-unstable.
\end{defi}

Let $(X,L)$ be a polarized variety. Denote by ${\rm Aut}(X)$ the automorphism group of $X$. Let ${\mathbf G}\subset{\rm Aut}(X)$ be any of its connected, reductive subgroup and ${\mathbf T}$ the identity component of its centre. Hisamoto \cite{Hisamoto} introduced the twist of equivariant test configurations and ${\mathbf G}$-uniform K-stability. We briefly recall the construction below and refer to \cite{Hisamoto} for details.

Let $(\mathcal X,\mathcal L)$ be a ${\mathbf G}$-equivariant test configuration of $(X,L)$. Denote by $\Lambda$ the vector field induced by the ${\rm k}^\times$-action, it defines a grading (still denoted by $\Lambda$) on
\begin{align}\label{Kodaira-ring}
R(X,L):=\bigoplus_{k=0}^{+\infty}R_k(X,L),~R_k(X,L):={\rm H}^0(X,L^k),
\end{align}
the Kodaira ring of $(X,L)$. Let $\sigma\in{\rm Lie}({\mathbf T})$. The twist $(\mathcal X_\sigma,\mathcal L_\sigma)$ of $(\mathcal X,\mathcal L)$ is defined as following: The uncompactified total space $(\mathcal X_\sigma\setminus\mathcal X_{\sigma\,\infty},\mathcal L_\sigma|_{\mathcal X\setminus\mathcal X_{\sigma\,\infty}})$ is isomorphic to that of $(\mathcal X,\mathcal L)$, but grading induced by $(\mathcal X_\sigma,\mathcal L_\sigma)$ is $\Lambda+\sigma$. That is, if the grading $\Lambda$ on $R_k$ has weight $\Lambda_1,...,\Lambda_{n_k}$, then the same basis diagonalizes $\sigma$ so that the weights for the grading $\Lambda+\sigma$ are $\Lambda_1+\sigma_1,...,\Lambda_{n_k}+\sigma_{n_k}$. When $\sigma$ is integral, $(\mathcal X_\sigma,\mathcal L_\sigma)$ is a test configuration so that the ${\rm k}^\times$-action induces a vector field $\Lambda+\sigma$. However, the compactified total space of $(\mathcal X_\sigma,\mathcal L_\sigma)$ is different from that of $(\mathcal X,\mathcal L)$. When $\sigma$ is irrational, $(\mathcal X_\sigma,\mathcal L_\sigma)$ in general may not be a test configuration in the above sense, and is called an $\mathbb R$-test configuration (cf. \cite{Hisamoto,Li19}).

The ${\mathbf G}$-uniform K-stability is phrased in terms of non-Archimedean functionals. For any test configuration $(\mathcal X,\mathcal L)$, denote by ${\rm DH}(\mathcal X,\mathcal L)$ its induced Duistermaat-Heckman measure, and $\Lambda_{\max}(\mathcal X, \mathcal L)$ the upper bound of the support of ${\rm DH}(\mathcal X,\mathcal L)$ (cf. \cite{Boucksom-Hisamoto-Jonsson}). Define the non-Archimedean J-functional
\begin{align}\label{J-NA-def}
{\rm J}^{\rm NA}(\mathcal X, \mathcal L):=&\Lambda_{\max}(\mathcal X, \mathcal L)-\frac1{(n+1)V}\mathcal L^{\cdot(n+1)},
\end{align}
where $\mathcal L^{\cdot(n+1)}$ denotes the top intersection number of $\mathcal L$. Also, the non-Archimedean Mabuchi functional
\begin{align*}
{\rm M}^{\rm NA}(\mathcal X,\mathcal L):=\frac{1}VK^{\rm log}_{\mathcal X/\mathbb P^1}\cdot\mathcal L^{\cdot n}+\frac{\bar S}{V(n+1)}\mathcal L^{\cdot(n+1)},
\end{align*}
where $V=L^{\cdot n}$ is the volume of $(X,L)$,
\begin{align*}
\bar S=\frac nV(K_X^{-1}\cdot L^{\cdot n})
\end{align*}
the mean value of the scalar curvature of $(X,L)$, and
\begin{align*}
K^{\log}_{\mathcal X/\mathbb P^1}=K_{\mathcal X}+\mathcal X_0^{\rm red}-{\rm pr}^*(K_{\mathbb P^1}+[0])
\end{align*}
a Weil divisor on $\mathcal X$. Here by $\mathcal X_0^{\rm red}$ we denote the reduced structure of the central fibre ${\rm pr}^{-1}(0)$. We have
\begin{defi}\label{K-stab-def-Fut}
A polarized ${\mathbf G}$-variety $(X,L)$ is called ${\mathbf G}$-uniformly K-stable if there is a constant $\epsilon_0>0$ so that for any ${\mathbf G}$-equivariant test configuration it holds
\begin{align}\label{uni-sta-def}
{\rm M}^{\rm NA}(\mathcal X,\mathcal L)\geq\epsilon_0\inf_{\sigma\in {\rm Lie}({\mathbf T})}{\rm J}^{\rm NA}(\mathcal X_\sigma, \mathcal L_\sigma).
\end{align}
\end{defi}
When there is no confusion, we may say ``uniform K-stability" in short of ``${\mathbf G}$-uniform K-stability" when the group ${\mathbf G}$ is fixed.

Let us give a few remarks. In general it holds ${\rm Fut}(\mathcal X,\mathcal L)\geq{\rm M}^{\rm NA}(\mathcal X,\mathcal L)$. However they coincide on test configurations with reduced central fibre. It is also proved that a test configuration can always have reduced central fibre after a suitable base change $t\to t^d$ on ${\rm k}^\times$ (and then taking normalization of the total space). Both ${\rm M}^{\rm NA}(\cdot)$ and ${\rm J}^{\rm NA}(\cdot)$ vary linearly under base change. That is, if $(\mathcal X^{(d)},\mathcal L^{(d)})$ is the base change of $(\mathcal X,\mathcal L)$, then
\begin{align}\label{base-change-of-NA}
{\rm M}^{\rm NA}(\mathcal X^{(d)},\mathcal L^{(d)})=d{\rm M}^{\rm NA}(\mathcal X,\mathcal L)~\text{and}~{\rm J}^{\rm NA}(\mathcal X^{(d)},\mathcal L^{(d)})=d{\rm J}^{\rm NA}(\mathcal X,\mathcal L).
\end{align}
For a detailed study of the non-Archimedean functionals and K-stability, we refer to the readers \cite{Boucksom-Hisamoto-Jonsson, Li19}. On the other hand, we recall that for a $\mathbb Q$-Fano variety with klt singularities, Li-Xu \cite{Li-Xu-Annl} (Li \cite{Li19}, resp.) proved that to check K-stability (equivariantly uniform K-stability, resp.) of $(X,K_X^{-1})$, it suffices to consider only special test configurations, which always has reduced central fibre.


The K-stability is closely related to the existence of K\"ahler-Einstein metrics when ${\rm k}=\mathbb C$. The Yau-Tian-Donaldson conjecture for Fano manifolds with discrete auto morphism group was proved by Tian \cite{Tian-15} (see also \cite{CDS}). For the equivariant version, Datar-Sz\'ekelyhidi \cite{DS} proved that equivariant K-stability is equivalent to the existence of K\"ahler-Einstein metrics for Fano manifolds. Later Li \cite{Li19} proved the uniform Yau-Tian-Donaldson conjecture (where the conjecture is stated for a stronger uniform K-stability condition) for $\mathbb Q$-Fano varieties with klt-singularities. We refer to the readers \cite{Tian-15,DS,Boucksom-Hisamoto-Jonsson,Li19} and references therein for further knowledge.

\subsection{$G$-varieties of complexity 1}
\subsubsection{General Luna-Vust theory} In this section we recall the general theory of normal varieties with a reductive group action, which is now referred as Luna-Vust theory. This theory was originally developed by Luna-Vust \cite{Luna-Vust} and has been extended and modified by Timash\"ev \cite{Timashev-1997}. In the following we use \cite{Timashev-1997,Timashev-book} as our main references and collect useful to us information.

Let $G$ be a connected, reductive, linear algebraic group over an algebraic closed field $\rm k$ of characteristic 0. Suppose that $G$ acts on a function field $K$. A $G$-\emph{model} of $K$ is a normal variety $X$ with regular $G$-action so that ${\rm k}(X)=K$. A $G$-model $X$ is uniquely determined by the set of local rings $\mathscr O_{X,Y}\subset K$ of all $G$-stable subvarieties $Y\subset X$ (cf. \cite[Section 1]{Timashev-1997}).

Fix a Borel subgroup $B$ in $G$. Denote by $\Gamma\subset\mathfrak X(B)$ the lattice of weights of all (non-zero) $B$-semiinvariant functions $f\in K^{(B)}\subset K$. Denote by $\mathscr V$ the set of $G$-invariant discrete $\mathbb Q$-valued geometric valuations (called $G$-\textit{valuation}) of $K$. Knop \cite{Knop93} proved that such a valuation is uniquely determined by its restriction on $K^{(B)}$. The valuation ring of a valuation $v$ on $K$ will be denoted by $\mathscr O_v$ below.

We would like to recall a characteristic of $\mathscr V$ for our later use, which we learned from \cite[Section 20]{Timashev-book}. Let us begin with some definitions there. Note that we have an exact sequence
\begin{align}\label{ex-seq-func}
1\to(K^B)^\times\to K^{(B)}\to\Gamma\to0.
\end{align}
Let $\nu$ be any geometric valuation of $K^B$. Factoring the sequence \eqref{ex-seq-func} by $\mathscr O_\nu^\times(\subset\mathscr O_\nu,~\text{the valuation ring of}~\nu)$ we get an exact sequence of lattices
$$1\to\mathbb Z_\nu\to\Gamma_\nu\to\Gamma\to0,$$
where $\mathbb Z_\nu\cong\mathbb Z$ or $0$ is the value group of $\nu$. Taking $\mathbb Q$-dual we get
\begin{align*}
&0&    &\longleftarrow&       &\mathbb Q_\nu&      &\longleftarrow&    &\mathscr Q_\nu&    &\longleftarrow&  &\mathscr Q_{0}&   &\longleftarrow&  &0&\\
&\,&    &\,&                     &\cup&               &\,&              &\cup&        &\,&                  &||&                   &\,&         &\,&\\
&0&    &\longleftarrow&       &\mathbb Q_{\nu,+}&   &\longleftarrow&  &\mathscr Q_{\nu,+}&  &\longleftarrow&  &\mathscr Q_{0,+}& &\longleftarrow&  &0.&
\end{align*}
Here $\mathbb Q_\nu=\mathbb Q$ and $\mathscr Q_{\nu,+}$ is the preimage of the positive ray $\mathbb Q_{\nu,+}$ for $\nu\not=0$, and $\mathbb Q_0=\mathbb Q_{0,+}=0$, $\mathscr Q_{0,+}=\mathscr Q_0\cong{\rm Hom}(\Gamma,\mathbb Q)$ (which will be simply denoted by $\mathscr Q$ below). The \emph{hyperspace of} $K$ is defined as
$$\mathscr E:=\cup_\nu\mathscr Q_{\nu,+},$$
where $\nu$ runs over all a geometric valuation of $K^B$ up to proportionality. More precisely, $\mathscr E=\mathscr Q$ if $c(X)=0$, while when $c(X)>0$, each $\mathscr Q_{\nu,+}$ is a half-space, and $\mathscr E$ is obtained by gluing together the $\mathscr Q_{\nu,+}$'s along their common boundary hyperplane $\mathscr Q$. This boundary $\mathscr Q$ is called the \emph{centre} of $\mathscr E$. The valuation cone $\mathscr V$ can be embedded into $\mathscr E$ as follows: Suppose that $v\in\mathscr V$ is a geometric valuation dominating $\nu$. Then via restriction $v|_{K^{(B)}}$, $v$ is mapped to an element in $\mathscr Q_{\nu,+}$ which is linear on $\Gamma_\nu$ and non-negative on $\mathbb Z_\nu$. It is further proved by \cite{Knop93} that this map is injective.

Fix any $G$-model $X$ of $K$ and denote by $\mathscr D$ the set of all prime divisors that is not $G$-stable. This set in fact depends only on $K$ but not on the choice of $X$. $B$-stable elements $\mathscr D^B$ are called the \emph{colours}. For any $D\in\mathscr D$, denote by $v_D={\rm ord}_D$ the corresponding normalized valuation of $K$. Similarly, there is a restriction map $\varrho:\mathscr D^B\to\mathscr E$ that maps a colour $D$ to $v_D|_{K^{(B)}}$, which can be identified with an element of $\mathscr E$. In the following we often denote this element by $v_D$ when there is no confusions. The map $\varrho$ is in general not injective. For more knowledge on hyperspaces, we refer to \cite[Sections 19-20]{Timashev-book}.

We may assume that $K={\rm Quot}R$ for some rational $G$-algebra $R$. For example we can take $R={\rm k}[X]$ if $X$ is a quasi-affine model of $K$ or the Kodaira ring of a complete model. Let $f_1,...,f_s\in R^{(B)}$ and $f\not=f_1...f_s$ be any $B$-eigenvector in $\langle G\cdot f_1\rangle...\langle G\cdot f_s\rangle$. Here by $\langle G\cdot f\rangle$ we mean the linear span of the orbit $G\cdot f$ in $R$. Then $f/f_1...f_s$ is called a \emph{tail vector} of $R$ and its weight is called a \emph{tail}. The tails are negative linear combination of simple roots of $G$ (with respect to $B$). For an affine $G$-variety, as proved by \cite{Alexeev-Brion} (see also \cite[Corollary E.15]{Timashev-book}), the tails span a finitely generated semigroup. It is showed that (cf. \cite[Section 20.5, proof of Proposition 20.11]{Timashev-book})
\begin{prop}\label{tail-valuation-cone-dual}
Assume that $K={\rm Quot}R$ for some rational $G$-algebra $R$. An element $v\in\mathscr E$ lies in $\mathscr V$ if and only if $v$ is non-negative on all tail vectors of $R$.
\end{prop}
Moreover, it is known that each $\mathscr V_\nu=\mathscr V\cap\mathscr Q_{\nu,+}$ is a finitely generated solid convex cone in $\mathscr Q_{\nu,+}$ (cf. \cite[Section 20]{Timashev-book}).

Let $X$ be a $G$-model of $K$. $B$-stable affine open subsets of $X$ are called $B$-charts. Recall the sets $\mathscr D$ and $\mathscr D^B$ defined above. For any $D\in\mathscr D$, denote by $\mathscr O_D(=\mathscr O_{v_D})$ its valuation ring. Then given any $B$-chart $\mathring X$, there are subsets $\mathscr W\subset\mathscr V$ and $\mathscr R\subset\mathscr D^B$ so that (cf. \cite[Section 1.4]{Timashev-1997})
\begin{align*}
{\rm k}[\mathring X]=(\bigcap_{w\in\mathscr W}\mathscr O_w)\cap(\bigcap_{D\in\mathscr R\sqcup(\mathscr D\setminus\mathscr D^B)}\mathscr O_D).
\end{align*}
It is of great interest to find a criterion on when the ring
\begin{align}\label{A(W,R)}
A(\mathscr W,\mathscr R)=(\bigcap_{w\in\mathscr W}\mathscr O_w)\cap(\bigcap_{D\in\mathscr R\sqcup(\mathscr D\setminus\mathscr D^B)}\mathscr O_D)
\end{align}
of a given  pair $(\mathscr W,\mathscr R)$ defines a $B$-chart. It turns out
\begin{theo}\label{B-chart-thm}{\rm (}\cite[Theorem 3]{Timashev-2000}, see also \cite[Section 1.4]{Timashev-1997}{\rm )}
A pair $(\mathscr W,\mathscr R)$ determines a $B$-chart if and only if it satisfies the conditions:
\begin{itemize}
\item[(C)] For any finite subset $\mathscr V_0\subset\mathscr W\sqcup\mathscr R$, there exists $f\in K^{(B)}$ so that $$(\mathscr W\sqcup\mathscr R)(f)\geq0~\text{and}~\mathscr V_0(f)>0.$$
\item[(F)] ${\rm k}[A(\mathscr W,\mathscr R)^{(B)}]$ is finitely generated.
\end{itemize}
\end{theo}
However, the pair $(\mathscr W,\mathscr R)$ may still contains some redundant data that different choices of $(\mathscr W,\mathscr R)$ may lead to same $A(\mathscr W,\mathscr R)$ (and consequently the resulting $B$-chart). This can be overcame if we further that every valuation $w\in\mathscr W$ is \emph{normalized} and \emph{essential}. By normalized we mean the group of values of a valuation $w$ is precisely $\mathbb Z$, and by essential we mean removing $w$ from the right-hand side of \eqref{A(W,R)} will change $A(\mathscr W,\mathscr R)$. We recall that in \eqref{A(W,R)} any valuation in $\mathscr R\sqcup(\mathscr D\setminus\mathscr D^B)$ is essential (cf. \cite{Luna-Vust,Knop93} and \cite[Proposition 13.7]{Timashev-book}). All valuations in $\mathscr W$ are essential if and only if
\begin{itemize}
\item[(W)] For any $w\in\mathscr W$, there exists $f\in K^{(B)}$ so that for any $w'\in\mathscr W\setminus\{w\}$, $D\in\mathscr R$, it holds $$w'(f),~v_D(f)\geq0,~w(f)<0.$$
\end{itemize}
Theorem \ref{B-chart-thm} then can be refined to
\begin{theo}\label{B-chart-thm-ref} {\rm (}\cite[Section 1.4]{Timashev-1997}{\rm )}
$B$-charts are in bijection with the data $(\mathscr W,\mathscr R)$ satisfying the conditions (C), (F) and (W).
\end{theo}
In the following, unless otherwise stated, when referring to a data $(\mathscr W,\mathscr R)$ we always assume it satisfies the conditions (C), (F) and (W).

The $G$-model $X$ can be recovered by patching the $G$-spans of finitely many $B$-charts that intersect $G$-invariant subvarieties in it (cf. \cite[Section 1.5]{Timashev-1997}). For any $G$-invariant subvariety $Y\subset X$, denote by $\mathscr V_Y$ ($\mathscr D^B_Y$, resp.) the set of normalized valuations corresponding to $G$-stable prime divisors of $X$ containing $Y$ (colours containing $Y$, resp.). Then it is showed by \cite{Luna-Vust} (see also \cite[Section 1.5]{Timashev-1997}) that the local ring $$\mathscr O_{X,Y}=A(\mathscr V_Y,\mathscr D^B_Y).$$
The support $\mathscr S_Y$ of $Y$ is
$$\mathscr S_Y:=\{v\in\mathscr V|\mathscr O_v~\text{dominates}~\mathscr O_{X,Y}\}.$$
Given a $B$-chart $\mathring X$ associated to the data $(\mathscr W,\mathscr R)$, we can recover $(\mathscr V_Y,\mathscr D^B_Y)$ for any $G$-invariant subvariety $Y$ intersecting $\mathring X$ according to the conditions (V), (V'), (D') introduced in \cite[Theorem 1.3]{Timashev-1997}. Finally, we have
\begin{theo}\label{fan-classify}{\rm (}\cite[Theorem 4]{Timashev-2000}, see also \cite[Sections 1.5-1.6]{Timashev-1997}{\rm )}
\begin{itemize}
\item[(1)] Given a finite collection $\{(\mathscr W_i,\mathscr R_i)\}_{i=1}^N$ of data satisfying the conditions (C), (F) and (W), the $G$-spans $\{G\mathring X(\mathscr W_i,\mathscr R_i)\}_{i=1}^N$ patch together to a $G$-model of $K$ if and only if the supports of all $G$-invariant subvarieties intersecting those $B$-charts are disjointed.
\item[(2)] The $G$-model constructed in (1) is complete if and only if the above supports cover $\mathscr V$.
\end{itemize}
\end{theo}

\subsubsection{The case of complexity 1} When $c(K)={\rm tr.deg}_{\rm k}(K^B)=1$, $K^B={\rm k}(C)$ for a unique smooth projective curve $C$ and there is a rational quotient ${\rm pr}_B:X\dasharrow C$ by the $B$-action. In the following we mainly consider the case when ${\rm k}(X)^B\cong{\rm k}(\mathbb P^1)$. This holds, for example when $X$ is unirational (by L\"uroth's Theorem).

Let us recall the reduction in \cite{Timashev-1997} of the general settings above to the complexity 1 case. As in \cite{Timashev-1997}, by fixing any monomorphism $e:\Gamma\to K^{(B)}$ that maps a weight $\lambda$ to a non-zero $B$-semiinvariant function $e_\lambda\in K^{(B)}_\lambda$, the sequence \eqref{ex-seq-func} splits. That is, any $f\in K^{(B)}_\lambda$ can be decomposed as $f=f_0e_\lambda$ with $f_0\in K^B\cong{\rm k}(C)$. On the other hand, any valuation $\nu$ on ${\rm k}(C)$ can be written as $\nu=h{\rm ord}_x$, where $h\in\mathbb Q_+$ and $x\in C$ a point. The hyperspace $\mathscr E$ introduced in the previous section then can be written more precisely as
$$\mathscr E=\bigsqcup_{x\in C}\mathscr Q_{x,+}/\sim,$$
where $\mathscr Q_{x,+}=\{x\}\times\mathscr Q\times\mathbb Q_+$, and $(x,\ell,h)\sim(x',\ell',h')$ if and only if $(x,\ell,h)=(x',\ell',h')$ or $h=h'=0$, $\ell=\ell'$. A valuation $v\in\mathscr V$ (when embedded in $\mathscr E$ via $v|_{K^{(B)}}$) is decomposed as
$$v=h_vq_{x_v}+\ell_v,$$
so that $v|_{K^B}=h_vq_{x_v}$ with $h_v\geq0$, $q_{x_v}=(x_v,O,1)$ the valuation at a point $x_v\in C$, and $\ell_v=v|_{e(\Gamma)}$. We call $h_v$ the \emph{jump} of $v$. Elements in $\mathscr Q$ (i.e. with zero jump) are said to be \emph{central}. For any central valuations $v\in\mathscr Q$, by the relation $v(f)=v(e_\lambda)$, we get $v(e_\lambda)=\ell_v(\lambda)$, and consequently $\mathscr Q\cong\Gamma_\mathbb Q^*$. Denote by $\mathscr B(X)$ the set of $B$-stable divisors in $X$. For any $D\in\mathscr B(X)$, by restricting the valuation corresponding to $D$ on $K^{(B)}$, $D$ is mapped to $v_D\in\mathscr E$ by
$$v_D=h_Dq_{x_D}+\ell_D.$$
It is proved in \cite[Section 3]{Timashev-1997} that for any $x\in C$, the set of divisors in $\mathscr B(X)$ that maps to $\mathscr Q_{x,+}$ is finite, and for almost every (that is, by ruling out finite exceptions) $x\in C$ there is exactly one $D\in\mathscr B(X)$ maps to $\mathscr Q_{x,+}\setminus\mathscr Q$, and this divisor satisfies $v_D=q_x$.

The above setting leads to a combinatorial description of functions in $K^{(B)}$ and $B$-charts on a $G$-model of $K$. For any $f\in K^B$, it induces a collection of linear functionals
$$\varphi=\{\varphi_x:\mathscr Q_{x,+}\to\mathbb Q|x\in C\}$$
so that $\varphi_x|_\mathscr Q\equiv0$ and $\varphi_x(q_x)=q_x(f)$ for any $x\in C$. In particular $\sum_{x\in C}\varphi_x(q_x)\cdot x$ is a principle divisor on $C$. A function $f=f_0e_\lambda\in K^{(B)}_\lambda$ with $f_0\in K^B$ defines the functionals
$$\varphi=\{\varphi_{0,x}+\lambda|x\in C\},$$
where $\varphi_0$ is the functional of $f_0$.

Given a $B$-chart $\mathring X$ with data $(\mathscr W,\mathscr R)$, a function $f\in K^{(B)}$ lies in ${\rm k}[\mathring X]$ if and only if $(\mathscr W\sqcup\mathscr R)(f)\geq0$, or equivalently, the functional of $f$, $\varphi\in(\mathscr W\sqcup\mathscr R)^\vee$. This condition selects regular functions in $K^{(B)}$. There are two types of $B$-charts:
\begin{itemize}
\item[(1)] Type \uppercase\expandafter{\romannumeral1}: ${\rm k}[\mathring X]^{(B)}\not={\rm k}$. In this case there is some $x\in C$ so that no element in $\mathscr W\sqcup\mathscr R$ falls into $\mathscr Q_{x,+}\setminus\mathscr Q$;
\item[(2)] Type \uppercase\expandafter{\romannumeral2}: ${\rm k}[\mathring X]^{(B)}={\rm k}$. In this case any $\mathscr Q_{x,+}$ contains a non-central element in $\mathscr W\sqcup\mathscr R$.
\end{itemize}
Accordingly, a $B$-chart of type \uppercase\expandafter{\romannumeral1} (or type \uppercase\expandafter{\romannumeral2} resp.) is characterised by the image of the rational quotient ${\rm pr}_B$ is a proper subset of $C$ (or equals to $C$, resp.).

For $\mathring X=\mathring X(\mathscr W,\mathscr R)$, it is proved in \cite[Section 3]{Timashev-1997} the set
$$\mathscr C(\mathscr W,\mathscr R)=\{v\in\mathscr E|v(\varphi)\geq0,~\forall \varphi\in(\mathscr W\sqcup\mathscr R)^\vee\}$$
together with $\mathscr R$ forms a \emph{coloured hypercone} in $\mathscr E$ defined as
\begin{defi}{\rm (}\cite[Definition 3.1]{Timashev-1997}{\rm )}
A coloured hypercone is a pair $(\mathscr C, \mathscr R)$ such that $\mathscr R\subset\mathscr D^B$, and $\mathscr C=\cup_{x\in C}\mathscr C_x$ is a union of finitely generated strictly convex cones $\mathscr C_x=\mathscr C\cap\mathscr Q_{x,+}$, each of them is generated by finitely many elements in $\mathscr V$ (and possibly by some vertices of a polytope $\mathscr P$ below) and elements of $\mathscr R$ that maps to $\mathscr Q_{x,+}$. Moreover, for almost every $x\in C$ the cone $\mathscr C_x$ is generated by $\mathscr K=\mathscr C\cap\mathscr Q$ and $q_x$, no elements of $\mathscr R$ is mapped to $O$, and either:
\begin{itemize}
\item[\uppercase\expandafter{\romannumeral1}.] $\exists x\in C$ so that $\mathscr C_x\subset\mathscr Q$;\\
or
\item[\uppercase\expandafter{\romannumeral2}.] $\mathscr P=\sum_{x\in C}\mathscr P_x\subset\mathscr K\setminus\{O\}$, where $q_x+\mathscr P_x$ is the convex hull of the vertices of a polyhedral domain $\mathscr K_x=\mathscr C_x\cap\{q_x+\mathscr Q\}$.
\end{itemize}
Elements in $\mathscr R$ are called the colours of the hypercone. Faces of $\mathscr K$ that does not intersect $\mathscr P$ is called a true face. Otherwise, it is called a pseudoface.
\end{defi}
The admissibility conditions of a coloured hypercone are obtained in \cite{Timashev-1997}, which are equivalent to the Conditions (C), (F) and (W) (cf. \cite[Remark 3.2, Propositions 3.1 and 4.1]{Timashev-1997}). It is proved in \cite{Timashev-1997} that $B$-charts of type \uppercase\expandafter{\romannumeral1} (or type \uppercase\expandafter{\romannumeral2}, resp.) are determined by coloured hypercones of type \uppercase\expandafter{\romannumeral1} (or type \uppercase\expandafter{\romannumeral2}, resp.).

In \cite{Timashev-1997}, the $G$-invariant subvarieties (referred as $G$-\emph{germs} in \cite{Timashev-1997}) are classified by $B$-charts intersecting them. A $G$-germ adimitting a $B$-chart of type \uppercase\expandafter{\romannumeral1} is called a $G$-germ of type \uppercase\expandafter{\romannumeral1}. $G$-germs admitting only $B$-charts of type \uppercase\expandafter{\romannumeral2} are called $G$-germs of type \uppercase\expandafter{\romannumeral2}. Also, a \emph{face} of a coloured hypercone $(\mathscr C, \mathscr R)$ is either
\begin{itemize}
\item a coloured cone $(\mathscr C',\mathscr R')\subset\mathscr Q_{x,+}$, where $\mathscr C'$ is a face of $\mathscr C_x$ intersecting $\mathscr K$ at a true face and $\mathscr R'$ are colours in $\mathscr R$ that maps to $\mathscr C'$;
\end{itemize}
or
\begin{itemize}
\item a hyperface $(\mathscr C',\mathscr R')$ of type \uppercase\expandafter{\romannumeral2}. In this case there is a linear functional $\varphi\in(\mathscr W\sqcup\mathscr R)^\vee$ so that $\mathscr C'=\mathscr C\cap\ker(\varphi)$ that intersects $\mathscr K$ at a pseudoface, and $\mathscr R'$ are colours in $\mathscr R$ that maps to $\mathscr C'$.
\end{itemize}
For a coloured cone $(\mathscr C,\mathscr R)\subset\mathscr Q_{x,+}$, we define its relative interior to be ${\rm RelInt}(\mathscr C)$ in the usual sense; for a coloured hypercone $(\mathscr C,\mathscr R)$ of type \uppercase\expandafter{\romannumeral2}, we define $${\rm RelInt}(\mathscr C):=(\cup_{x\in C}{\rm RelInt}(\mathscr C_x))\bigcup{\rm RelInt(\mathscr C\cap\mathscr Q)}.$$ It is proved in \cite[Sections 3.3 and 4.3]{Timashev-1997} that
\begin{theo}{\rm (}\cite[Theorems 3.2-3.2, 4.1]{Timashev-1997}{\rm )}
\begin{itemize}
\item $G$-germs intersecting a given $B$-chart are determined by faces of the corresponding coloured hypercone whose relative interior intersects $\mathscr V$;
\item The adherence of $G$-germs corresponds to the opposite inclusion of coloured (hyper-)cone as faces;
\item A normal $G$-model of $K$ is given by a set of coloured cones and hypercones of type \uppercase\expandafter{\romannumeral2} obtained from a finite collection of coloured hypercones as the set of all their faces whose relative interior intersects $\mathscr V$ and the relative interiors of these faces are disjoint inside $\mathscr V$.
\end{itemize}
\end{theo}
A collection $\mathfrak F_X$ of coloured cones and hypercones of type \uppercase\expandafter{\romannumeral2} that defines a $G$-model $X$ above is called the \emph{coloured fan} of $X$. A $G$-model $X$ is said to be of \emph{type \uppercase\expandafter{\romannumeral1}} if its coloured fan consists of only coloured cones.\footnote{In Section 6.1 below there is explicit data of a concrete example, the Mukai-Umemura threefold. We also refer to the readers Fig-2 there for an example of coloured hypercone of type \uppercase\expandafter{\romannumeral2}.}

Finally we recall the fact that there are two kinds of $G$-varieties of complexity 1:
\begin{itemize}
\item [(1)] The \emph{one-parameter} case: $K^G=K^B$. In this case, a $G$-orbit in general position has codimension 1 and contains a dense $B$-orbit;
\item [(2)] The \emph{quasihomogeneous} case: $K^G={\rm k}$. In this case, there is a homogeneous space $G/H$ so that any $G$-model of $K$ contains a dense $G$-orbit isomorphic to $G/H$. A quasihomogeneous variety is always unirational (cf. \cite[Section 2.2]{Timashev-1997}).
\end{itemize}

\subsubsection{The central automorphism group}
Let $X$ be a $G$-variety. The central $G$-equivariant automorphism group is defined as
$$\mathfrak A(X):=\{\sigma\in{\rm Aut}_G(X)|\sigma\cdot f= f,~\forall f\in{\rm k}(X)^{B}\}.$$
Then central valuation cone $\mathscr V\cap\mathscr Q$ is closely related to the Lie algebra of $\mathfrak A(X)$ (cf. \cite[Section 8]{Knop93} and \cite[Section 21]{Timashev-book}). It is showed in \cite{Knop93} that there is a torus ${\mathbf T}\subset \mathfrak A(X)$ whose Lie algebra is the linear part $\mathscr A(:=\mathscr V\cap(-\mathscr V)\cap\mathscr Q)$ of $\mathscr V$ (identified with a subalgebra of ${\rm Lie}(T)$, where $T=B\cap B^-$ is a maximal torus of $G$). Consider the group ${\mathbf G}\subset{\rm Aut}(X)$ generate by ${\mathbf T}$ and the $G$-action on $X$. Then ${\mathbf G}$ is a reductive group with ${\mathbf T}$ contained in its centre.

\subsection{Cartier divisors on $G$-varieties}
Let $X$ be a $G$-model of $K$ and $L$ a line bundle on it. Up to replacing $G$ by a finite covering, we may assume that $G$ is of simply connected type. Then by \cite{KKLV}, $L$ is $G$-linearized, and by \cite{FMSS} there always exists a $B$-stable divisor $\mathfrak d$ of $L$. In this section, we recall some basic results on Cartier divisors on $X$ established by \cite{Timashev-2000}, in particular the combinatorial data associated to ample divisors. Finally we give a formula of $B$-stable anti-canonical divisors.

Let $X$ be a $G$-model of $K$ and $\mathscr B(X)$ the set of all $B$-stable prime divisors of $X$. Let
\begin{align}\label{B-stable-divisor}
\mathfrak d=\sum_{D\in\mathscr B(X)}m_DD
\end{align}
be a $B$-stable Weil divisor on $X$. The following criterions of Cartier property and ampleness were proved in \cite{Timashev-2000}, which is a generalization of \cite{Brion89} for spherical varieties.

\begin{theo}\label{Cartier-criterion}{\rm (}\cite[Theorem 5]{Timashev-2000}{\rm )}
The divisor $\mathfrak d$ in \eqref{B-stable-divisor} is Cartier if and only if for any $G$-invariant subvariety $Y\subset X$, there exists $f_Y\in K^{(B)}$ such that for each prime divisor $D\supset Y$, $m_D=v_D(f_Y)$. That is, $\mathfrak d$ is Cartier if and only if it is locally principal at a generic point of each $G$-invariant subvariety.
\end{theo}

Given a $B$-stable Cartier divisor, we have the following ampleness criterion:
\begin{theo}\label{ampleness-criterion}{\rm (}\cite[Theorem 6]{Timashev-2000}{\rm )}
Suppose that the divisor $\mathfrak d$ in \eqref{B-stable-divisor} is Cartier and determined by local data $\{f_Y\}$ as in Theorem \ref{Cartier-criterion}. Then
\begin{itemize}
\item[(1)] $\mathfrak d$ is globally generated if and only if $f_Y$ can be chosen so that for any $G$-invariant subvariety $Y\subset X$, the following two condition hold:
\begin{itemize}
\item[(a)] For any other $G$-invariant subvariety $Y'\subset X$ and $B$-stable prime divisor $D\supset Y'$, $v_D(f_Y)\leq v_D(f_{Y'})$;
\item[(b)] $\forall D\in\mathscr D^B\setminus(\cup_{Y'\subset X}\mathscr D_{Y'}^B):~v_D(f_Y)\leq m_D$.
\end{itemize}
\item[(2)] $\mathfrak d$ is ample if and only if, after replacing $\mathfrak d$ by its certain multiple, the functions $\{f_Y\}$ can be chosen so that for any $G$-invariant subvariety $Y\subset X$, there exists $B$-charts $\mathring X$ of $Y$ such that (a)-(b) are satisfied and
\begin{itemize}
\item[(c)] the inequalities therein are strict if and only if $D\cap\mathring X=\emptyset$.
\end{itemize}
\end{itemize}
\end{theo}

\subsubsection{The ample divisors}
Suppose that the divisor $\mathfrak d$ in \eqref{B-stable-divisor} is ample, and $L=L_\mathfrak d$ its line bundle. Also suppose that $\mathfrak d$ is the divisor of some section ${\rm H}^0(X,L)^{(B)}_{\lambda_0}$ for some $\lambda_0\in\mathfrak X(B)$. In this section we recall results on the space of global sections ${\rm H}^0(X,L^k)$ from \cite[Section 7]{Timashev-2000}. For our later use, we only focus on the results when $K^B$ is rational, in which $C=\mathbb P^1$. For results in general cases we refer to the readers \cite{Timashev-2000}.

Recall the functions
\begin{align*}
A_x(\mathfrak d,\lambda):=\min_{\{v_D=h_Dq_{x_D}+\ell_D\in\mathscr B(X)|x_D=x,h_D>0\}}\frac{\ell_D(\lambda)+m_D}{h_D},~\lambda\in\Gamma_\mathbb R,~x\in C,
\end{align*}
defined in \cite[Section 7]{Timashev-2000}. Note that the minimum taken here is well-defined since there are only finitely many divisors in $\mathscr B(X)$ satisfy $x_D=x$ for every $x\in C$. For almost every $x\in C$ (that is, there are at most finitely many exceptions), $A_x(\mathfrak d,\lambda)\equiv0$. Thus
\begin{align}\label{A(d,lambda)}
A(\mathfrak d,\lambda):=\sum_{x\in C}A_x(\mathfrak d,\lambda)
\end{align}
is well-defined. \cite[Section 7]{Timashev-2000} (see also \cite[Section 17.4]{Timashev-book}) introduced the following polytope 
\begin{align*}
\Delta_\mathscr Z(\mathfrak d):=\{\lambda\in\Gamma_\mathbb R|v_D(\lambda)+m_D\geq0,~\forall{D\in\mathscr B(X),h_D=0},~\text{and}~A(\mathfrak d, \lambda)\geq0\},
\end{align*}
and proved
\begin{prop}\label{Delta(d)-Z}
Assume that $X$ is a $G$-variety of complexity 1 with ${\rm k}(X)^B\cong{\rm k}(\mathbb P^1)$, and let $\mathfrak d$ be an ample divisor of $X$ as above. A rational point $\lambda\in\Gamma_\mathbb Q$ lies in $\Delta_\mathscr Z(\mathfrak d)$ if and only if ${\rm H}^0(X,L^k)^{(B)}_{k(\lambda+\lambda_0)}\not=0$ for some positive integer $k$ satisfying $k\lambda\in\Gamma$. Moreover, for any $\lambda\in\Gamma_\mathbb Q$,
\begin{align}\label{dim-Hk-lambda}
\dim{\rm H}^0(X,L^k)^{(B)}_{k(\lambda+\lambda_0)}=\max\{\sum_{x\in C}[kA_x(\mathfrak d,\lambda)]+1, 0\}.
\end{align}
As a consequence,
$$\Delta_\mathscr Z(L):=\Delta_\mathscr Z(\mathfrak d)+\lambda_0$$
is a solid convex polytope in $\Gamma_\mathbb R+\lambda_0\subset(\mathfrak X_\mathbb R(B))$ that depends only on $L$ but not on the choice of $\mathfrak d$.
\end{prop}

From Proposition \ref{Delta(d)-Z} we directly get the following expression of $\dim{\rm H}^0(X,L^k)$, which will be frequently used later. Set
\begin{align}\label{pi-def}
\pi(\lambda)=\prod_{\alpha^\vee\in\Pi_{G}^\vee\setminus(\Gamma+\mathbb Z\lambda_0)^\perp}\frac{\langle\lambda,\alpha^\vee\rangle}{\langle\rho,\alpha^\vee\rangle},~\lambda\in\mathfrak X_\mathbb R(B).
\end{align}
We have the following asymptotic expression of $\dim{\rm H}^0(X,L^k)$ which is a refinement of \cite[Theorem 8]{Timashev-2000}, and also a counterpart of the intersection formulas established in \cite{Del-2021,Del-2020-09} for spherical varieties.
\begin{lem}\label{h0(X,Lk)}
Suppose that $X$ is a $G$-variety of complexity 1 with ${\rm k}(X)^B\cong{\rm k}(\mathbb P^1)$, and $L$ an ample $G$-linearized line bundle on $X$ with $\mathfrak d$ defined by \eqref{B-stable-divisor} a $B$-stable divisor of $L$. Then up to replace $L$ with $L^d$ for a sufficiently divisible $d$, it holds
\begin{align}\label{h0(X,Lk)-eq}
\dim{\rm H}^0(X,L^k)=&k^n\int_{\Delta_\mathscr Z(L)}A(\mathfrak d,\lambda-\lambda_0)\pi(\lambda)d\lambda+k^{n-1}\int_{\Delta_\mathscr Z(L)}A(\mathfrak d,\lambda-\lambda_0)\langle\nabla\pi(\lambda),\rho\rangle d\lambda\notag\\
&+\frac12k^{n-1}\left(\int_{\partial\Delta_\mathscr Z(L)}A(\mathfrak d,\lambda-\lambda_0)\pi(\lambda)d\sigma
-\sum_{x\in C}\sum_{a=1}^{N}\int_{\Omega_a}(1-\frac1{h_{D_a(x)}})\pi(\lambda)d\lambda\right)\notag\\
&+k^{n-1}\int_{\Delta_\mathscr Z(L)}\pi(\lambda)d\lambda+O(k^{n-2}),~k\to+\infty,
\end{align}
where $d\lambda$ is the standard Lebesgue measure on $\Gamma_\mathbb R+\lambda_0$, normalized by the lattice $\Gamma$, $d\sigma$ is the induced lattice measure on the boundary, $\nabla\pi$ the gradient of $\pi$,
$$\nabla\pi(\lambda)=\frac1{\prod_{\beta^\vee\in\Pi_{G}^\vee\setminus(\Gamma+\mathbb Z\lambda_0)^\perp}\langle\rho,\beta^\vee\rangle}\sum_{\alpha^\vee\in\Pi_{G}^\vee\setminus(\Gamma+\mathbb Z\lambda_0)^\perp}\left(\prod_{\beta^\vee\in\Pi_{G}^\vee\setminus(\Gamma+\mathbb Z\lambda_0)^\perp,\beta^\vee\not=\alpha^\vee}\langle\lambda,\beta^\vee\rangle\right)\cdot\alpha^\vee,$$
and $\Delta_\mathscr Z(L)=\cup_a\Omega_a$ so that $\Omega_a$'s are convex rational polytopes with disjoint interiors and on each $\Omega_a$, every
\begin{align}\label{affine-domain}
A_x(\mathfrak d,\lambda-\lambda_0)=\frac{m_{D_a(x)}+\ell_{D_a(x)}(\lambda-\lambda_0)}{h_{D_a(x)}}~\text{for some}~D_a(x)\in\mathscr B(X)\cap\{x_D=x\},
\end{align}
is affine.
\end{lem}
\begin{proof}
Up to replace $L$ with $L^d$ for sufficiently divisible $d$, we may also assume that $\Delta_\mathscr Z(\mathfrak d)$ is integral, and vertices of the graph of every $A_x(\mathfrak d,\lambda)$ on $\Delta_\mathscr Z(\mathfrak d)$ are also integral. This is possible since
$$\Delta_\mathscr Z(d\mathfrak d)=d\Delta_\mathscr Z(\mathfrak d),~A_x(d\mathfrak d,d\lambda)=dA_x(\mathfrak d,\lambda),~\forall \lambda\in\Delta_\mathscr Z(\mathfrak d)\subset\Gamma_\mathbb R,$$
and there are only finitely many $x\in C$ with $A_x(\mathfrak d,\lambda)\not\equiv0$.

Fix an $s_{\lambda_0}\in{\rm H}^0(X,L)$ as before. Then by the Weyl character formula, for any $\lambda\in (k\lambda_0+\Gamma$), it holds
$$\dim V_\lambda=\pi(\lambda+\rho).$$

Set
\begin{align*}
\mathfrak P_k=&\{\lambda\in k\Delta_\mathscr Z(L)\cap(\Gamma+k\lambda_0)|\sum_{x\in C}[kA_x(\mathfrak d,\frac\lambda k-\lambda_0)]<0\}\\
=&k\{\lambda\in \Delta_\mathscr Z(\mathfrak d)\cap\frac1k\Gamma|\sum_{x\in C}[kA_x(\mathfrak d,\lambda)]<0\}+k\lambda_0.
\end{align*}
By Proposition \ref{Delta(d)-Z},
\begin{align*}
\dim{\rm H}^0(X,L^k)=&\sum_{\lambda\in k\Delta_\mathscr Z(L)\cap(\Gamma+k\lambda_0)}\left(\max\{0,\sum_{x\in C}[kA_x(\mathfrak d,\frac\lambda k-\lambda_0)]+1\}\right)\dim V_{\lambda}\\
=&\sum_{\lambda\in k\Delta_\mathscr Z(L)\cap(\Gamma+k\lambda_0)}\left(\sum_{x\in C}[kA_x(\mathfrak d,\frac\lambda k-\lambda_0)]+1\right)\dim V_{\lambda}\\
&-\sum_{\lambda\in \mathfrak P_k}\left(\sum_{x\in C}[kA_x(\mathfrak d,\frac\lambda k-\lambda_0)]+1\right)\dim V_{\lambda}.
\end{align*}

Note that (cf. \cite[Korollar 3.3]{Knop93})
\begin{align}\label{n=c+r+1}
n={\rm rk}(\Gamma)+\deg(\pi)+1.
\end{align}
Applying Lemmas \ref{bad-point-lem} and \ref{riemann-sum} in the Appendix, the second term
$$0\leq-\sum_{\lambda\in \mathfrak P_k}\left(\sum_{x\in C}[kA_x(\mathfrak d,\frac\lambda k-\lambda_0)]+1\right)\dim V_{\lambda}\leq Ck^{n-2},~k\to+\infty$$
for some uniform constant $C>0$, which has lower order in $k$. Thus it suffices to deal with the first term. By applying Lemma \ref{riemann-sum} to the first sum, it holds
\begin{align*}
\dim{\rm H}^0(X,L^k)=&\sum_{\lambda\in k\Delta_\mathscr Z(L)\cap(\Gamma+k\lambda_0)}\left(\sum_{x\in C}[kA_x(\mathfrak d,\frac\lambda k-\lambda_0)]+1\right)\dim V_{\lambda}\\
=&\sum_{x\in C}\sum_{\lambda\in k\Delta_\mathscr Z(L)\cap(\Gamma+k\lambda_0)}[kA_x(\mathfrak d,\frac\lambda k-\lambda_0)]\pi(\lambda+\rho)+\sum_{\lambda\in k\Delta_\mathscr Z(L)\cap(\Gamma+k\lambda_0)}\pi(\lambda+\rho)\\
=&\sum_{x\in C}\left(k^n\int_{\Delta_\mathscr Z(L)}A_x(\mathfrak d,\lambda-\lambda_0)\pi(\lambda)d\lambda+k^{n-1}\int_{\Delta_\mathscr Z(L)}A_x(\mathfrak d,\lambda-\lambda_0)\langle\nabla\pi(\lambda),\rho\rangle d\lambda\right.\\
&+\frac12k^{n-1}\int_{\partial\Delta_\mathscr Z(L)}A_x(\mathfrak d,\lambda-\lambda_0)\pi(\lambda)d\sigma\left.-\frac12k^{n-1}\sum_{a=1}^{N}\int_{\Omega_a}(1-\frac1{|h_{D_a(x)}|})\pi(\lambda)d\lambda\right)\\
&+k^{n-1}\int_{\Delta_\mathscr Z(L)}\pi(\lambda)d\lambda+O(k^{n-2}),
\end{align*}
where $d\sigma$, $\Omega_a$ and $D_a(x)$ are defined as above. Here and above we used the fact that whenever $\lambda\in k\Delta_\mathscr Z(L)\cap(\Gamma+k\lambda_0)$, $\lambda-k\lambda_0\in k\Delta_\mathscr Z(\mathfrak d)\cap\Gamma$, where $\Gamma$ is a fix lattice on which Lemmas \ref{bad-point-lem} and \ref{riemann-sum} apply. The Lemma then follows from \eqref{A(d,lambda)}.

\end{proof}

\subsubsection{Associated parabolic subgroup and properties of the function $\pi$}

The set $\Pi_{G}^\vee\setminus(\Gamma+\mathbb Z\lambda_0)^\perp$ in \eqref{pi-def} is in fact independent with the choice of the ample line bundle $L$ (cf. \cite{Brion89, Timashev-2000}). Indeed it consists of unipotent coroots of the \emph{associated parabolic subgroup} of $X$. This seems to be well-known to the experts, while for the readers' convenience we give a short explanation below. Recall that for any $G$-variety $X$, its associated parabolic subgroup is defined by (see \cite{Knop91} for definition)
\begin{align}\label{ass-para}
P=\{g\in G|gBx=Bx~\text{for general}~x\in X\},
\end{align}
or equivalently, the maximal common stabilizer group of all colours. Let $L$ be any ample line bundle on $X$. Then $P$ stabilizes any line in ${\rm H^0}(X,L^k)$ generated by a $B$-semiinvariant section for all $k\in\mathbb N_+$, since the divisor of such a section is also $P$-stable. Moreover, since $L$ is ample, $P$ in particular equals to the stabilizer of one $B$-semiinvariant section in some ${\rm H^0}(X,L^{k_0})$ (\cite[Satz 2.5]{Knop93}). Thus for every $k\in\mathbb N_+$, the roots of the Levi group $L_P$ of $P$ are precisely roots of $G$ that are orthogonal to the set
$$\bigcup_{k\in\mathbb N_+}\{\lambda\in\mathfrak X(B)|\dim{\rm H}^0(X,L^k)^{(B)}_\lambda\not=0\}.$$
By Proposition \ref{Delta(d)-Z}, they are precisely the roots of $G$ that are orthogonal to $\Delta_\mathscr Z(L)$, which is a solid polytope in $\Gamma_\mathbb R+\lambda_0$. Hence
\begin{lem}\label{Pi-Pu}
Let $X$ be a projective $G$-variety of complexity 1 and $P$ its associated parabolic subgroup. Let $L$ be any ample $G$-line bundle on $X$. Then $\Pi_{L_P}^\vee=\Pi_G^\vee\cap(\Delta_\mathscr Z(L))^\perp$ and $\Pi_{P_u}^\vee=\Pi_{G}^\vee\setminus(\Delta_\mathscr Z(L))^\perp (=\Pi_{G}^\vee\setminus(\Gamma+\mathbb Z\lambda_0)^\perp)$. Consequently,
\begin{align*}
\pi(\lambda)=\prod_{\alpha^\vee\in\Pi_{P_u}^\vee}\frac{\langle\lambda,\alpha^\vee\rangle}{\langle\rho,\alpha^\vee\rangle},~\lambda\in\mathfrak X_\mathbb R(B).
\end{align*}
\end{lem}

For our later use, we introduce a further property of the gradient $\nabla\pi$ of $\pi$,
following the argument of \cite[Section 10]{Del-2020-09} in the spherical case. Let $L$ be any ample line bundle on $X$ and fix a section $s_L\in{\rm H}^0(X,L)^{(B)}_{\lambda_0}$ for some $\lambda_0\in\mathfrak X(B)$. Consider the Weyl group $W(L_P)$ of $L_P$ with respect to $B$. As discussed above, $P$ stabilizes any point $\lambda\in \Gamma$ and the weight $\lambda_0$. Thus $\lambda$ and $\lambda_0$ are $L_P$-character and hence are $W(L_P)$-invariant. On the other hand, $W(L_P)$ acts on $\Pi_G$ and permutes elements in $\Pi_G\setminus\Pi_{L_P}=\Pi_{P_u}$. By Lemma \ref{Pi-Pu}, $W(L_P)$ precisely permutes factors $\alpha^\vee$'s in \eqref{pi-def}. It follows
$$(w^*\pi)(\lambda)=\pi(\lambda),~\forall w\in W(L_P),~\forall\lambda\in\mathfrak X_\mathbb R(B).$$
Denote
$$\kappa_P:=\sum_{\alpha\in\Pi_{P_u}}\alpha=2\rho-\sum_{\alpha\in\Pi_{L_P}}\alpha.$$
Note that there is a longest element $w_o\in W(L_P)$ that maps every $\alpha\in\Pi_{L_P}$ to $-\alpha$. In particular,
$$w_o(\kappa_P-2\rho)=2\rho-\kappa_P.$$
Combining with the fact
$$w_0(\lambda+\lambda_0)=\lambda+\lambda_0,~\forall\lambda\in\Gamma_\mathbb R,$$
we get
\begin{lem}\label{nabla-pi}
Let $X$ be a projective $G$-variety of complexity 1 and $P$ its associated parabolic subgroup. Let $L$ be an ample $G$-line bundle on $X$. Fix any section $s_L\in{\rm H}^0(X,L)^{(B)}_{\lambda_0}$ for some $\lambda_0\in\mathfrak X(B)$. Then it holds
$$\langle\nabla\pi(\lambda+\lambda_0),2\rho\rangle=\langle\nabla\pi(\lambda+\lambda_0),\kappa_P\rangle,~\forall\lambda\in\Delta_\mathscr Z({\rm div}(s_L))\subset\Gamma_\mathbb R,$$
or equivalently,
$$\langle\nabla\pi(\lambda),2\rho\rangle=\langle\nabla\pi(\lambda),\kappa_P\rangle,~\forall\lambda\in\Delta_\mathscr Z(L)\subset(\Gamma_\mathbb R+\lambda_0).$$
\end{lem}
Here in the second line we used Proposition \ref{Delta(d)-Z}.

\subsubsection{Relation with the divisorial polytope construction}
Finally we introduce some combinatorial properties of ample line bundles. When $G$ is a torus, Ilten-S\"u\ss\,  classified polarized $T$-varieties of complexity 1 via the ``divisorial polytopes" introduced in \cite[Section 3]{Ilten-Suss-2011}. This is a family of concave functions on $\Delta_\mathscr Z(L)$ satisfying certain conditions. In particular, in the case of $T$-varieties of complexity 1 there is no colour and the fan of a polarized variety can be recovered from its divisorial polytopes. The following Proposition is a partial generalization of this direction in the case of a general reductive $G$.

For each $x\in C$, consider the following polyhedron
\begin{align}
\Delta_x(\mathfrak d):=\{(t,\lambda)\in\mathbb R\times\Delta_\mathscr Z(\mathfrak d)|t\geq -A_x(\mathfrak d,\lambda)\}.
\end{align}
Then we have
\begin{prop}\label{cone-generator-normal}
For each $x\in C$, every maximal cone $\mathscr C_x$ in the fan $\mathfrak F_{X,x}:=\{(\mathscr C_x,\mathscr R_x)|(\mathscr C,\mathscr R)\in\mathfrak F_X\}$ is an inner normal cone of $\Delta_x(\mathfrak d)$ at some of its vertex. 
\end{prop}

\begin{proof}
Each coloured hypercone $(\mathscr C,\mathscr R)$ in $\mathfrak F_X$ is associated to a $G$-subvariety $Y$, and by Theorem \ref{Cartier-criterion}, we associate to each $Y$ a function $f_\mathscr C\in K^B$ which locally defines $\mathfrak d$. Furthermore, the collection $\{f_{\mathscr C}\}$ can be chosen so that it satisfies the ampleness conditions in Theorem \ref{ampleness-criterion}. Suppose that $\mathscr C=\mathscr C(\mathscr W,\mathscr R)$ and denote by $\mathring X_\mathscr C$ the $B$-chart defined by $(\mathscr C,\mathscr R)$. It holds
\begin{align*}
\left\{\begin{aligned}{\rm div}(f_\mathscr C)|_{\mathring X_\mathscr C}=&\mathfrak d\cap\mathring X_\mathscr C,\\ \mathfrak d-{\rm div}(f_\mathscr C)>&0.\end{aligned}\right.
\end{align*}
The first condition is the Cartier condition in Theorem \ref{Cartier-criterion}, and the second is the ampleness condition in Theorem \ref{ampleness-criterion}. Let $f_\mathscr C=f^0_\mathscr Ce_{\lambda_\mathscr C}$ for some $\lambda_\mathscr C\in\Gamma$ and $f^0_\mathscr C\in K^B$. This equivalences to
\begin{align}\label{eq-Cartier}
v_D(f_\mathscr C)=&h_Dq_{x_D}(f^0_\mathscr C)+\ell_D(\lambda_\mathscr C)=m_D,~\forall D\in\mathscr B(X)~\text{with}~v_D\in(\mathscr W\sqcup\mathscr R),
\end{align}
and
\begin{align}\label{eq-ample}
\left\{\begin{aligned} m_D\geq&h_Dq_{x_D}(f^0_\mathscr C)+\ell_D(\lambda_\mathscr C),~\forall D\in\mathscr B(X),\\ m_D>&h_Dq_{x_D}(f^0_\mathscr C)+\ell_D(\lambda_\mathscr C),~\text{whenever}~v_D\not\in(\mathscr W\sqcup\mathscr R).\end{aligned}\right.
\end{align}
We see that each point $(-q_{x_D}(f^0_\mathscr C),-\lambda_\mathscr C)$ lies in $\Delta_{x_D}(\mathfrak d)$.

It remains to show that each $(-q_{x}(f^0_\mathscr C),-\lambda_\mathscr C)$ is in fact a vertex of $\Delta_{x}(\mathfrak d)$. It remains to deal with cones $\mathscr C_x=\mathscr C\cap\mathscr Q_{x,+}$ of maximal dimension. Suppose that $\mathscr C_x$ is a coloured cone. Then $\mathscr C_x$ is generated by $(\mathscr W\sqcup\mathscr R)\cap\mathscr Q_{x,+}$, and the equations \eqref{eq-Cartier} has a unique solution $(-q_{x}(f^0_\mathscr C),-\lambda_\mathscr C)$. It is further a vertex of $\Delta_{x}(\mathfrak d)$ due to \eqref{eq-ample}. When $\mathscr C$ is a hypercone of type \uppercase\expandafter{\romannumeral2}, it may happens that $\mathscr C_x$ has generators on its pseudofaces. Let $y$ be any point in $C$ and $v_D\in(\mathscr W\sqcup\mathscr R)\cap\mathscr Q_{y,+}$. From the Cartier condition it holds
\begin{align*}
m_D-h_Dq_{y}(f^0_\mathscr C)-\ell_D(\lambda_\mathscr C)=0.
\end{align*}
For any other $D'\in\mathscr B(X)$ so that $v_{D'}\in\mathscr Q_{y,+}$, by ampleness.
\begin{align*}
m_{D'}-h_{D'}q_{y}(f^0_\mathscr C)-\ell_{D'}(\lambda_\mathscr C)\geq0.
\end{align*}
Thus for any $y\in C$,
\begin{align*}
\frac{m_{D}-\ell_{D}(\lambda_\mathscr C)}{h_D}=q_{y}(f^0_\mathscr C)\leq\frac{m_{D'}-\ell_{D'}(\lambda_\mathscr C)}{h_{D'}},~\forall v_D\in(\mathscr W\sqcup\mathscr R)\cap\mathscr Q_{y,+}~\text{and}~v_{D'}\in\mathscr Q_{y,+}.
\end{align*}
Hence
$$A_y(\mathfrak d,-\lambda_\mathscr C)=\frac{m_{D}-\ell_{D}(\lambda_\mathscr C)}{h_D}$$ for some $v_D\in(\mathscr W\sqcup\mathscr R)\cap\mathscr Q_{y,+}$. By definition of the pseudoface of $\mathscr C$, we get
$$A(\mathfrak d,-\lambda_\mathscr C)\geq0~\text{if and only if}~(-\lambda_\mathscr C)\in(\mathscr C\cap\mathscr Q)^\vee.$$
Thus, if $\mathscr C_x$ is of maximal dimension, the condition \eqref{eq-Cartier} together with $A(\mathfrak d,-\lambda)=0$ uniquely determined the point $(-q_{x}(f^0_\mathscr C),-\lambda_\mathscr C)$. It is again a vertex of $\Delta_{x}(\mathfrak d)$ according to \eqref{eq-ample}.

Clearly, in both cases, $\mathscr C_x$ is the inner normal cone of $\Delta_x(\mathfrak d)$ at $(-q_{x}(f^0_\mathscr C),-\lambda_\mathscr C)$. 

\end{proof}

\begin{rem}
When $G=T$ is a torus, the divisorial polytope of $(X,L)$ constructed in \cite{Ilten-Suss-2011} is precisely $(\Delta_\mathscr Z(\mathfrak d),\{A_x(\mathfrak d,\cdot):\Gamma_\mathbb R\to \mathbb R\})$, where $\Delta_\mathscr Z(\mathfrak d)$ is precisely the projection of $\Delta_x(\mathfrak d)$ to $\mathscr Q$. This correspondence can be derived from \cite[Section 6]{Timashev-2008}.
\end{rem}

\begin{rem}
From Proposition  \ref{cone-generator-normal} we see that $\Delta_x(\mathfrak d)$ consists of points $(\lambda,t)\in\Gamma_\mathbb R\times\mathbb R$ that satisfy the following system:
\begin{align}\label{polyhedral-eq}
\left\{
\begin{aligned}
&m_D+th_D+\lambda(\ell_D)\geq0,~\forall D\in\mathscr B(X)~\text{so that}~v_D=h_Dq_x+\ell_D\in\mathscr Q_{x,+},\\
&m_p+\lambda(p)\geq0,~\forall p\in~\text{an extremal line of a pseudoface of a hypercone $\mathscr C$ of type \uppercase\expandafter{\romannumeral2}.}
\end{aligned}\right.
\end{align}
Here in \eqref{polyhedral-eq}, the second inequality occures only when $\mathscr C$ is a hypercone of type \uppercase\expandafter{\romannumeral2}. The inequality is taken over all $p=\sum_{y\in C}p_{Dy}$ that lies in an extremal line of $\mathscr C$, where each 
$$p_{Dy}=\sum_{\{D\in\mathscr B(X),~v_D=h_Dq_y+\ell_D\in\mathscr C_y,~h_D\not=0\}}c_D(\frac1{h_D}\ell_D)\in\mathscr P_y,$$
is a convex combination of vertices of $\mathscr P_y$ (with $c_D$'s the coefficients), the corresponding
$$m_{Dy}=\sum_{\{D\in\mathscr B(X),~v_D=h_Dq_y+\ell_D\in\mathscr C_y\}}c_Dm_D,$$
and $m_p=\sum_{y\in C}m_{Dy}$.
\end{rem}

\begin{rem}
Suppose that $X$ is complete. We see that coloured cones in $\mathfrak F_X$ are precisely those inner normal cones of $\Delta_x(\mathfrak d)$ whose generators lie in $\mathscr B(X)$ and relative interior intersects $\mathscr V$. The (maximal) coloured hypercones of type \uppercase\expandafter{\romannumeral2} are as follows: Consider a maximal inner normal cone at a vertex $p_*\in\Delta_{x_*}(\mathfrak d)$ so that not all its generators lie in $\mathscr B(X)$. Then $p_*$ projects to a boundary point $\lambda_*$ of $\Delta_\mathscr Z(\mathfrak d)$ with $A(\mathfrak d,\lambda_*)=0$. Denote by $p_*(x)$ the point on $\Delta_x(\mathfrak d)\cap\{t=-A_x(\mathfrak d,\lambda_*)\}$ which projects to $\lambda_*$, and the set
$$\mathscr S(x,\lambda_*):=\{D\in\mathscr B(X)|A_x(\mathfrak d,\lambda_*)=\frac{m_D+\ell_D(\lambda_*)}{h_D}\}.$$
Then there is a hypercone $(\mathscr C,\mathscr R)\in\mathfrak F_X$ of type \uppercase\expandafter{\romannumeral2} associated to the coloured data $(\mathscr W=\cup_{x\in C}\mathscr W_x,\mathscr R=\cup_{x\in C}\mathscr R_x)$, where
\begin{align*}
\mathscr C_x=&\text{the inner normal cone of}~\tilde\Delta_x(\mathfrak d)~\text{at}~p_*(x),\\
\mathscr W_x=&\{D\in\mathscr B(X)^G|h_D=0,v_D(\lambda_*)=-m_D\}\cup\left(\mathscr B(X)^G\cap\mathscr S(x,\lambda_*)\right),\\
\mathscr R_x=&\{D\in\mathscr B(X)\setminus\mathscr B(X)^G|h_D=0,v_D(\lambda_*)=-m_D\}\cup\left((\mathscr B(X)\setminus\mathscr B(X)^G)\cap\mathscr S(x,\lambda_*)\right),
\end{align*}
and any (maximal) hypercone of type \uppercase\expandafter{\romannumeral2} in $\mathfrak F_X$ is constructed in this way.
\end{rem}

\subsection{The affine cone over a projectively normal embedding}

Let $X$ be a projective $G$-variety of complexity 1, and $L$ a $G$-linearized ample line bundle on it. As the coloured fan of $X$ is a bit complicated, in this section we will study $X$ via its affine cone under certain embedding. This approach was used in \cite{Ilten-Suss-Duke} in the study of $T$-varieties.

Suppose that $L=L_\mathfrak d$ is a $G$-linearized ample line bundle on $X$ which has a divisor $\mathfrak d$ given in \eqref{B-stable-divisor}. Up to replacing $L$ by its tensor power, we may assume that $X$ can be embedded into a projective space by $|L|$ and the embedding is projectively normal. In this case, the affine cone $\hat X$ over $X$ is an affine normal variety $$\hat X={\rm Spec}R(X,L),$$
where $R(X,L)$ is the Kodaira ring of $(X,L)$ given by \eqref{Kodaira-ring}. As $X$ is complete, it holds ${\rm H}^0(X,L^0)={\rm k}$. Hence the vertex of the cone $\hat X$ is precisely the origin point $O$. Let ${\rm pr}:\hat X\setminus\{O\}\to X$ be the projection. Then for any subvariety $Y\subset X$, $\hat Y:={\rm pr}^{-1}(Y)$ is the affine cone over $Y$ and it also contains $O$ as its vertex.

\subsubsection{Combinatorial data of the affine cone}
Obviously that the affine cone $\hat X$ is a $(G\times {\rm k}^\times)$-variety. In fact, the group $\hat G:=G\times {\rm k}^\times$ acts on $\hat X$ as following: $G$ acts on each piece $R_k(X,L)$ naturally and $t\in k^\times$ acts on $R_k(X,L)$ by
$$(t\cdot s)(\hat x)(=:s(t^{-1}\cdot\hat x))=t^{-k}(s(\hat x)),~\forall s\in R_k(X,L),~\hat x\in \hat X~\text{and}~k\in\mathbb N.$$
In particular, $s$ has ${\rm k}^\times$-weight $-k$.\footnote{Note that this is opposite to the usual convention in the case of torus actions.} In the following we denote objects of  $\hat X$ by adding a hat.

It is direct to see that ${\rm k}(\hat X)^{B\times{\rm k}^\times}={\rm k}(X)^B$, the lattice $\hat \Gamma\cong\Gamma\oplus\mathbb Z$ and the hyperspace $\hat{\mathscr E}\cong\mathscr E\times\mathbb Q$. In fact, by choosing $B$-semiinvariant $s_*\in R_1(X,L)_{\lambda_*}^{(B)}$ so that $\lambda_*\in(\Delta_\mathscr Z(\mathfrak d)+\lambda_0)$, the map
\begin{align}
{\rm k}(X)^{(B)}\times\mathbb Z&\rightarrow {\rm k}(\hat X)^{(\hat B)}\label{k(hatX)-k(X)}\\
(f_\lambda,p)&\rightarrow f_\lambda s_*^{p}\in {\rm k}(\hat X)^{\hat B}_{(\lambda+p\lambda_*,p)},\notag
\end{align}
gives a desired isomorphism.

Then we determine the combinatorial data of $\hat B$-stable prime divisors $\hat{\mathscr B}(\hat X)$ in $\hat X$. The $\hat B$-stable divisors in $\hat X$ consist of all $\hat D$'s, where each $\hat D$ is the affine cone over $D$ with $D\in\mathscr B(X)$. For convenience, without loss of generality we can assume that the divisor \eqref{B-stable-divisor} is effective and choose $s_*=s_0$. Then for every $D\in\mathscr B(X)$, it holds
$$v_{\hat D}(f_\lambda)=v_D(f_\lambda),~\forall f_\lambda\in {\rm k}(X)^{(B)}_\lambda~\text{and}~v_{\hat D}(s_0)=-m_D.$$
Thus we get
\begin{align*}
v_{\hat D}=(v_{D},-m_D)\in\mathscr E\times\mathbb Q(=:\hat {\mathscr E}).
\end{align*}

Finally we derive the combinatorial data of $\hat X$. The affine cone $\hat X$ itself is a $\hat B$-chart determined by $\{\hat D|D\in\mathscr B(X)\}$. In fact, $\hat X$ is the $\hat B$-chart intersecting $O\in\hat X$. Since $X$ is complete, its coloured fan $\mathfrak F_X$ covers $\mathscr V$. In particular, for each $x\in C$, there is at least one $D\in\mathscr B(X)$ so that $v_D\in\mathscr Q_{x,+}\setminus\mathscr Q$. Clearly such a $D$ satisfies $v_{\hat D}\in(\mathscr Q_{x,+}\times\mathbb Q)\setminus(\mathscr Q\times\mathbb Q)$. Hence the coloured hypercone $(\hat{\mathscr C},\hat{\mathscr R})$ is a coloured hypercone of type \uppercase\expandafter{\romannumeral2}. We then determine its pseudofaces. For each $D\in\mathscr B(X)$ with $v_D\in\mathscr Q_{x_D,+}$, we have $x_{\hat D}=x_D$ and $h_{\hat D}=h_D$.
Set
$$\hat{\mathscr P_x}:={\rm Conv}(\{\frac{v_{\hat D}}{h_D}-q_{x_D}|D\in\mathscr B(X)~\text{so that}~x_D=x~\text{and}~h_D\not=0\}),$$
and
\begin{align}\label{affine-cone-P-eq}
\hat{\mathscr P}=\sum_{x\in C}\hat{\mathscr P_x},
\end{align}
where the right-hand side is the Minkowski sum. Then the pseudofaces of $\hat{\mathscr C}$ are exactly those intersect $\hat{\mathscr P}$. We get

\begin{prop}\label{affine-cone-data-prop}
Let $X$ be a complete normal $G$-variety of complexity 1 which is embedded into a projective space by sections of $L=L_\mathfrak d$ as a projectively normal subvariety. Let $\hat X$ be the affine cone over this embedding. Then $\hat X$ is a $\hat B$-chart of type  \uppercase\expandafter{\romannumeral2} determined by the data $(\hat{\mathscr W},\hat{\mathscr R})$, where
$$\hat{\mathscr W}=\{(v_D,-m_D)|D\in\mathscr B(X)^G\},$$
and
$$\hat{\mathscr R}=\{(v_D,-m_D)|D\in\mathscr B(X)\setminus\mathscr B(X)^G\}.$$
Moreover, for each $x\in C$ it holds
\begin{align}\label{comb-data-hatC-Cone}
\hat{\mathscr C}_x&={\rm Cone}(\{v_{\hat D}|D\in\mathscr B(X),~x_D=x\}\cup\hat{\mathscr P}),
\end{align}
and
\begin{align}\label{comb-data-hatC-Colour}
\hat{\mathscr R}_x&=\{\hat D|D\in\mathscr B(X)\setminus\mathscr B(X)^G,~x_D=x\},
\end{align}
where $\hat{\mathscr P}$ is given by \eqref{affine-cone-P-eq}.
\end{prop}

\subsubsection{Admissibility}
It is also direct to check that the combinatorial data \eqref{comb-data-hatC-Cone}-\eqref{comb-data-hatC-Colour} is admissible. More precisely, in the following we show that they fulfill the conditions in \cite[Definition 3.1]{Timashev-1997}.

For each $x\in C$, consider the cone\footnote{Again, note the opposite sign convention used here compared with the case of torus actions.}
$$\tilde{\mathscr C}_x=\{(\epsilon,m)\in\Gamma_\mathbb R\times\mathbb R\times\mathbb R|-\frac \epsilon m\in\Delta_x(\mathfrak d)\}.$$
In fact, it is the cone over $\Delta_x(\mathfrak d)\times\{-1\}\subset(\Gamma_\mathbb R\times\mathbb R\times\mathbb R)$. We will see that under the ampleness assumption, $\hat{\mathscr C}_x$ is the dual cone of $\tilde{\mathscr C}_x$, as showed in \cite{Ilten-Suss-Duke} for $T$-varieties.

It is clear that a point $(\epsilon,m)\in(\Gamma_\mathbb R\times\mathbb R\times\mathbb R_{\leq0})$ lies in the dual cone $(\hat{\mathscr C}_x)^\vee$ of $\hat{\mathscr C}_x$ if and only if
\begin{align}\label{hatC-dual-eq-1}
v_D(\epsilon)-m_Dm\geq0,~\forall D\in\mathscr B(X)~\text{so that}~x_D=x,
\end{align}
and
\begin{align}\label{hatC-dual-eq-2}
p(\epsilon)-m_pm\geq0,~\forall p=\sum_{y\in C}p_y\in\hat{\mathscr P},
\end{align}
where each
\begin{align}\label{hatC-py}
p_y=\sum_{D_y}c_{D_y}(\frac{v_{\hat D_y}}{h_{D_y}}-q_y)\in\hat{\mathscr P}_y
\end{align}
is a convex combination of some $D_y\in\mathscr B(X)$ so that $h_{D_y}\not=0$ and $x_{D_y}=y$, and $$m_p=\sum_{y\in C}\sum_{D_y}c_{D_y}\frac {m_{D_y}}{h_{D_y}}.$$

Let us show that some inequalities in the system \eqref{hatC-dual-eq-1}-\eqref{hatC-dual-eq-2} can be removed while keeping the solution set unchanged. In fact, suppose that there are two points $(v_i,m_i)$, $i=1,2$, which lie on edges of two different cones $\mathscr C_{i,x}\in\mathfrak F_{X,x},~i=1,2$, respectively, and $m_i=m_D$ if $v_i=v_D$ or $m_i=m_p$ if $v_i$ lies in a pseudoface. Then
$$\Delta_x(\mathfrak d)\subset\{\epsilon\in\mathscr Q_{x,+}|v_i(\epsilon)+m_i\geq0,~i=1,2\}.$$
Consequently for their convex combination $v=tv_1+(1-t)v_2$ with $t\in[0,1]$, it holds $m_v=tm_1+(1-t)m_2$ and
$$\Delta_x(\mathfrak d)\subset\{\epsilon\in\mathscr Q_{x,+}|v(\epsilon)+m_v\geq0\}.$$
The same holds true for convex combination of arbitrarily many points $\{v_i\}$ where each $v_i$ lies on an edge of some $\mathscr C_{i,x}\in\mathfrak F_{X,x}$. 
Hence in \eqref{hatC-dual-eq-2} it suffices to consider only inequalities given by $p$'s for which all $D_y$'s in \eqref{hatC-py} belong to a same hypercone of type \uppercase\expandafter{\romannumeral2} in $\mathfrak F_X$ (i.e. removing other inequalities in \eqref{hatC-dual-eq-2} does not change the solution set). In this way the system \eqref{hatC-dual-eq-1}-\eqref{hatC-dual-eq-2} then reduces to \eqref{polyhedral-eq} and we get that $(\hat{\mathscr C}_x)^\vee=\tilde{\mathscr C}_x$. Thus $\hat{\mathscr C}_x=(\tilde{\mathscr C}_x)^\vee$ is a convex cone lies in $\mathscr Q_{x,+}\times\mathbb Q_{\leq0}$. Clearly $O\not\in\hat{\mathscr P}$ and the conditions in \cite[Definition 3.1]{Timashev-1997} are fulfilled.

\subsubsection{Relationship with the coloured fan of $X$}
Let $\hat X$ be the affine cone as before. We see that except the origin $O$, any $\hat G$-invariant subvariety of $\hat X$ is the affine cone over a $G$-invariant subvariety in $X$. Thus by \cite[Theorem 3.2]{Timashev-1997}, $G$-invariant subvarieties in $X$ are in one-one correspondence with faces (coloured cone intersecting only true faces or hypercones of type \uppercase\expandafter{\romannumeral2} that intersecting pseudofaces) of $(\hat{\mathscr C},\hat{\mathscr R})$, whose relative interior intersects the valuation cone of $\hat X$.

Let $\hat Y\subset\hat X$ be a $\hat G$-invariant subvariety and $(\hat{\mathscr C}',\hat{\mathscr R}')$ the face of $(\hat{\mathscr C},\hat{\mathscr R})$ that is associated to $\hat Y$. Then it defines a $\hat B$-chart $\widehat{X'}\subset\hat X\setminus\{O\}$ intersecting $\hat Y$ (when $(\hat{\mathscr C}',\hat{\mathscr R}')$ is a coloured cone, we choose $\widehat{X'}$ any $\hat B$-chart that corresponds to a coloured hypercone of type \uppercase\expandafter{\romannumeral1} containing $(\hat{\mathscr C}',\hat{\mathscr R}')$). Then $\widehat{X'}$ is a normal affine open set and $\mathring X':={\rm Proj}(\widehat{X'})$ is a $B$-chart of $X$ intersecting $Y={\rm Proj}(\hat Y)$. It remains to determine the coloured hypercone $({\mathscr C}',{\mathscr R}')$ defining $\mathring X'$. From \eqref{k(hatX)-k(X)} we see that the cone ${\mathscr C}'$ associated to $\mathring X'$ is the image of $\hat{\mathscr C}'$ under the canonical projection from $\hat{\mathscr E}$ to $\mathscr E$. We then select the colours. Suppose that $D\in\mathscr B(X)\setminus\mathscr B(X)^G$ is a colour with $x_D=x$. By Lemma \ref{cone-generator-normal}, $\mathscr C'_x$ is the inner normal cone of $\Delta_x(\mathfrak d)$ at some face $F$. From the ampleness condition, $D\in\mathscr R'$ if and only if the hyperplane $\{\epsilon\in\Gamma_\mathbb R\times\mathbb R|v_D(\epsilon)+m_D=0\}$ intersects $F$.

As showed in the previous subsection, $(\hat{\mathscr C})_x^\vee\subset\{(\epsilon,m)\in(\Gamma_\mathbb R\times\mathbb R\times\mathbb R)|m\leq0\}$ for each $x\in C$. Hence, $(0,-1)\in\hat{\mathscr C}\subset\hat{\mathscr E}\cong\mathscr E\times\mathbb Q$. We conclude that the image of proper faces of $(\hat{\mathscr C},\hat{\mathscr R})$ under the canonical projection to $\mathscr E$ forms a coloured fan in $\mathscr E$, and it is indeed the coloured fan of $X$.



\section{The anti-canonical divisors}

Let $X$ be a normal variety. The anti-canonical sheaf is defined by $\check{\omega}_X=i_*\check{\omega}_{X_{\rm reg}}$, where $X_{\rm reg}$ is the regular locus, $\check{\omega}_{X_{\rm reg}}$ its anti-canonical sheaf, and $i$ denotes the inclusion. Then $\check{\omega}_X$ is isomorphic to $\mathscr O_X(\mathfrak d)$ for some Weil divisor $\mathfrak d$ on $X$, which we call \emph{anti-canonical divisor}. In this section, we derive a formula of $B$-stable anti-canonical divisor on a $G$-variety of complexity 1. The one-parameter and quasihomogeneous cases will be treated separately. 

\subsection{The one-parameter case}
In Section 3.1 we do not need to assume that ${\rm k}(X)^B$ is rational. For our later use, we first study the structure of a generic $G$-orbit in $X$, and derive some coloured data of $X$ from that of this orbit.

\subsubsection{Generic $G$-orbits}
In the one-parameter case, consider the rational quotient for the $B$-action
$${\rm pr}_B: X\dashrightarrow C.$$
Then the fibre $X_z:={\rm pr}_B^{-1}(z)$ over a generic $z\in C$ contains a $G$-spherical homogeneous space $O_z$. In fact, by \cite[Theorem 3.1]{Alexeev-Brion} there exists a $G$-stable dense open subset $X_O\subset X$ and a spherical subgroup $H\subset G$ such that in an open set $\mathring C\subset C$, any $G$-orbit $X_z\cap X_O=O_z$ is $G$-equivariantly isomorphic to a fixed spherical homogeneous space $O:=G/H$. It is further showed in \cite[Section 3.1]{Langlois} that
\begin{theo}\label{birational-modle}
Let $X$ be a one-parameter $G$-variety with generic $G$-orbit $O$. Then there exists a smooth projective curve $\tilde C$, a $G$-equivariant rational map
\begin{align}\label{Galois-covering}
\Psi:\tilde X:=O\times\tilde C\dashrightarrow X_O,
\end{align}
which is a Galois covering on a $G$-stable dense open subset. After shrinking $C$, $\Psi$ also induces a Galois covering $\tilde C\to C$ which gives rise to a $G$-equivariant isomorphism between ${\rm k}(\tilde X)$ and ${\rm k}(\tilde C)\otimes_{{\rm k}(C)}{\rm k}(X)$. The corresponding Galois group $A$ acts on $\tilde X$ (as $G$-equivariant birational transformations) via a generically free $A$-action on $\tilde C$ and a $G$-equivariant $A$-action on $O$. Moreover, up to shrinking $X_O$, it holds
\begin{align}\label{Galois-quot-model}
X_O=\tilde X/A.
\end{align}
\end{theo}
The first part of Theorem \ref{birational-modle} was originally proved in \cite[Theorem 2.13]{colliot-kunyavskii-popov-reichstein}. Here we use its reformulation \cite[Theorem 3.4]{Langlois}. The statement on the $A$-action on $\tilde X$ was proved in \cite[Proposition 3.9]{Langlois}. The last point is due to \cite[Corollary 3.5]{Langlois}.

We want to compute central birational invariants of $X$. By Theorem \ref{birational-modle}, it suffices to consider the open dense set $X_O$ and with out loss of generality we can assume that \eqref{Galois-quot-model} holds.

First we consider the lattice $\Gamma$ of weights of $B$-semiinvariant functions on $X$. It suffices to find ${\rm k}(\tilde X)^{(B)\times A}$. Denote by $\mathfrak M(O)$ the lattice of $B$-semiinvariant functions of $O$. Suppose that $e_\lambda\in{\rm k}(O)^{(B)}_\lambda$ for $\lambda\in\mathfrak M(O)$. Then $e_\lambda$ pulls back to a function (still denoted by $e_\lambda$) in ${\rm k}(\tilde X)^{(B)}_\lambda$. Since $A$ acts on $O$ as a subgroup in ${\rm Aut}_G(O)$, there is an $A$-character $\chi_\lambda\in\mathfrak X(A)$ such that
$$a\cdot e_\lambda=\chi_\lambda(a)e_\lambda,~\forall a\in A.$$
On the other hand, since the field extension ${\rm k}(\tilde C):{\rm k}(C)$ is Galois, ${\rm k}(\tilde C)$ is a regular $A$-module (cf. \cite[Section 5.4]{Kostrikin-3}). 
There always exists an $\tilde f_0\in{\rm k}(\tilde C)^{(A)}_{-\chi_\lambda}$, and consequently $\tilde f_0e_\lambda\in{\rm k}(\tilde X)^{(B)\times A}_\lambda$. Also, since ${\rm k}(\tilde X):{\rm k}(X)$ is a Galois extension, ${\rm k}(X)={\rm k}(\tilde X)^A$. We get $\tilde f_0e_\lambda\in{\rm k}(X)^{(B)}_\lambda$, whence $\mathfrak M(O)\subset \Gamma$. The converse inclusion is obvious (cf. \cite[Section 2.3]{Timashev-1997}). Hence we can identify $\Gamma$ with $\mathfrak M(O)$. Also, the central part $\mathscr V\cap\mathscr Q$ of the valuation cone $\mathscr V$ of $X$ can be identified with the valuation cone $\mathscr V(O)$ of $O$, as showed in \cite[Section 2.3]{Timashev-1997}.

It remains to deal with the colours. By \cite[Lemma 2.1]{Timashev-1997}, all colours in one-parameter case are central. Let $D\in\mathscr D^B$ be a colour of $X$. We assign to $D$ any colour $\hat D$ in $O$ so that $\hat D\times \tilde C$ (which is a colour in $\tilde X$) that maps to $D$ via $\Psi$. Suppose that $\hat D$ and $\hat D'$ are any two such divisors in $O$. Then they differ from each other by an $A$-action. By Lemma \ref{spherical-quot} in the Appendix, both $\hat D$ and $\hat D'$ map to a same point $v_{\hat D}=v_{\hat D'}$ in the hyperspace of $G/H$, which can be identified with $\Gamma^*_\mathbb Q$. The assignment
$$\mathscr D^B\ni D\to v_{\hat D}\in\Gamma^*_\mathbb Q$$
is well-defined. Note that for any $\tilde f_0e_\lambda\in{\rm k}(\tilde X)^{(B)\times A}$ as above,
$${\rm ord}_D(\tilde f_0e_\lambda)={\rm ord}_{\hat D}(e_\lambda).$$
We can further identify $v_D$ with $v_{\hat D}$. Thus we have
\begin{prop}\label{central-elements-one-para}
Let $X$ be one-parameter $G$-spherical variety with generic $G$-orbit $O$. Then
$$\Gamma\cong\mathfrak M(O),~(\mathscr V\cap\mathscr Q)\cong\mathscr V(O),$$
and for any $D\in\mathscr D^B$,
$$v_D=v_{\hat D},$$
where $\hat D$ is any colour in $O$ such that $\hat D\times\tilde C$ maps to $D$ under the Galois covering \eqref{Galois-covering}.
\end{prop}

Recall that any colour in a spherical variety is assigned to one of certain types (type-a, a', b) according to the minimal parabolic subgroup of $G$ that moves it. We refer to the readers \cite{Ga-Ho-datum} and \cite[Section 30.10]{Timashev-book} for a convenient survey on properties of colours of different types. By Lemma \ref{spherical-quot} in the Appendix, different colours in $O$ that are assigned to a same colour $D\in\mathscr D^B$ of $X$ have the same type. Thus we have
\begin{defi}
Let $X$ be a one-parameter $G$-variety of complexity 1. A colour $D\in\mathscr D^B$ is said to be of  type-a (or a', b, respectively) if any (hence all) $\hat D$ assigned to $D$ is a colour of type-a (or a', b, respectively) of $O$.
\end{defi}

For our later use, we make the following remarks:

\begin{rem}\label{v-colour}

Let $O=G/H$ be a $G$-spherical homogeneous space. For each simple root $\alpha$ of $G$, denote by $P_\alpha$ the corresponding minimal standard parabolic subgroup of $G$ that contains $B$. We say that a colour $D$ of $O$ lies in the set $\mathscr D^B(O;\alpha)$ if $P_\alpha\cdot D\not=D$. Let $P(O)$ be the associated parabolic subgroup of $O$. Then $P(O)=P$ is the associated parabolic subgroup of $X$. Also, denote by $\Sigma(O)$ the spherical roots of $O$. We have:
\begin{itemize}
\item If $D$ is of type-a, then $D\in\mathscr D^B(O;\alpha)$ for some simple root $\alpha$ of $G$ such that $\alpha\in \Sigma(O)$, and $v_D(\alpha)=1$. In particular, $v_D$ is primitive;
\item If $D$ is of type-a', then $D\in\mathscr D^B(O;\alpha)$ for some simple root $\alpha$ of $G$ such that $2\alpha\in \Sigma(O)$. In this case $v_D=\frac12\alpha^\vee|_{\mathfrak M(O)}$;
\item If $D$ is of type-b and $D\in\mathscr D^B(O;\alpha)$ for some simple root $\alpha$ of $G$, then $v_D=\alpha^\vee|_{\mathfrak M(O)}$.
\end{itemize}

For a one-parameter $G$-variety $X$ with generic $G$-orbit $O$, and $D\in\mathscr D^B$ a colour of it, we write $D\in\mathscr D^B(\alpha)$ if any (hence all) corresponding $\hat D\in\mathscr D^B(O;\alpha)$. As in the spherical cases \cite[Section 4]{Brion97}, set
\begin{align}\label{coe-colour-one-para}
\bar m_D=\left\{\begin{aligned}&\frac12\langle\alpha^\vee,\kappa_P\rangle=1,~&\text{for}~D\in \mathscr D^B(\alpha)~\text{of type-a or a'},\\ &\langle\alpha^\vee,\kappa_P\rangle\geq2,~&\text{for}~D\in \mathscr D^B(\alpha)~\text{of type-b}. \end{aligned}\right.
\end{align}
\end{rem}

\subsubsection{The anti-canonical divisor}
In Theorem \ref{anti-can-div-thm} below we give a formula of $B$-stable anti-canonical divisors on a one-parameter $G$-variety $X$. The first part of Theorem \ref{anti-can-div-thm} (i.e. equation \eqref{anti-can-div}) for $T$-varieties of complexity 1 was proved by Petersen-S\"u\ss\,\cite[Theorem 3.21]{Petersen-Suss-2011}, and Langlois-Terpereau \cite[Theorem 2.18]{Langlois-Terpereau-2016} for horospherical complexity 1 varieties. Langlois \cite[Theorem 5.1]{Langlois} gives a formula of canonical divisors for general normal $G$-varieties with spherical orbits, which in fact leads to \eqref{anti-can-div} in our case. We remark that \cite[Theorem 5.1]{Langlois} uses a different approach and the formula there is given in terms of certain Galois covering of $X$.

Recall Remark \ref{v-colour}. For a one-parameter $G$-variety $X$, its associated parabolic subgroup $P$ equals to that of $O$. In the following we fix a Levi decomposition $P=L_PP_u$, where $P_u$ is the unipotent radical and $L_P$ is the Levi subgroup that contains the maximal torus $B\cap B^-$. We have
\begin{theo}\label{anti-can-div-thm}
Let $X$ be a $\mathbb Q$-Gorenstein one-parameter $G$-variety of complexity 1. Denote by $P$ its associated parabolic subgroup. Then there is a $B$-stable anti-canonical $\mathbb Q$-divisor of $X$,
\begin{align}\label{anti-can-div}
\mathfrak d=\sum_{D\in\mathscr B(X), h_D\not=0}(1-h_D+h_Da_{x_D})D+\sum_{D\in\mathscr B(X)^G, h_D=0}D+\sum_{D\in\mathscr D^B, h_D=0}\bar m_DD,
\end{align}
where $\mathfrak a:=\sum_{x\in\mathbb P^1}a_x[x]$ is an anti-canonical $\mathbb Q$-divisor of $C$.\footnote{Thus $a_x=0$ for almost every $x\in C$.} Moreover, if $m\mathfrak d$ is a Cartier divisor for some $m\in\mathbb N_+$, then $m\mathfrak d$ is the divisor of a $B$-semiinvariant rational section of $K_X^{-m}$ with weight $m\kappa_P$.
\end{theo}
\begin{proof}
We will adopt the method of \cite[Section 4]{Brion97}. It suffices to deal with the case when $X$ is Gorenstein and $m=1$. Suppose that $\mathfrak d={\rm div}(s)$ for some $B$-semiinvariant global section of $K^{-1}_X$. Then $\mathfrak d$ can be decomposed as \eqref{B-stable-divisor}. Also by removing singular locus of $X$ (which is of codimension at least 2), we may assume that $X$ is regular.

We compute $m_D$ in \eqref{B-stable-divisor} for $D\in\mathscr B(X)^G$ following the argument of \cite[Proposition 4.1]{Brion97}. Consider the divisor
$$\delta:=\sum_{D\in\mathscr D^B}D.$$
Note that since $\delta$ is not $G$-stable, the union of all $G$-orbits contained in $\delta$ is a set of codimension at least 2. By \cite[Lemma 2.2]{Knop94}, each $D\in\mathscr D^B$ is Cartier on
$$X'=X\setminus\{G\text{-orbits contained in}~\delta\},$$
and $\delta$ is an effective $B$-stable divisor on $X'$. At the same time, $P={\rm Stab}_{G}(\delta)$ is the associated parabolic subgroup of $X'$. Thus by \cite[Theorem 2.3 and Proposition 2.4]{Knop94},
\begin{align}\label{local-structure}
X'\setminus\delta\cong P_u\times  Z\cong P\times^{L_P} Z
\end{align}
for some $T$-variety of complexity 1 with respect to the torus $T:=L_P/[L_P,L_P]$-action. Moreover, any $G$-stable divisor in $X'$ descends to a $T$-stable divisor of $Z$. Thus the restriction of the anti-canonical section $s$ on $\mathring X'$ can be written as $s=s_1\wedge s'$ with $s_1$ a section of $K^{-1}_Z$ and $s'$ a section of $K^{-1}_{P_u}$. The coefficients in \eqref{B-stable-divisor} for $D\in\mathscr B(X)^G$ then follows from the case of $T$-varieties proved in \cite[Theorem 3.21]{Petersen-Suss-2011}.

Then we determine the coefficients of colours in \eqref{B-stable-divisor}. By removing $G$-germs of type \uppercase\expandafter{\romannumeral2} (which has codimension at least 2), we may assume that $X$ contains only $G$-invariant subvarieties of type \uppercase\expandafter{\romannumeral1}. Then by Lemma \ref{X-C-map-lem} in the Appendix, ${\rm pr}_B:X\to C$ is a morphism. And by \cite[Theorem 16.25]{Timashev-book}, for $z$ in an open subset $\mathring C\subset C$, $X_z={\rm pr}_B^{-1}(z)$ is normal, whence a spherical embedding of $O$. On the other hand, up to shrinking $\mathring C$ we may assume both $\mathfrak d$ and the central divisors in $\mathscr B(X)$ intersects with $X_z$ transversally. In particular, $X_z$ intersects $D$ transversally and we get
\begin{align*}
m_D={\rm ord}_D(s)=&{\rm ord}_{\hat D}(s|_{X_z}),
\end{align*}
where $\hat D$ is any colour of $O$ that is assigned to $D$ (i.e. a component of $X_z\cap D$). 

Note that for any $z'\in C\setminus\{z\}$, $X_z\cap {\rm pr}_B^{-1}(z')=\emptyset$. 
Consider the line bundle $L_{X_z}={\rm pr}_B^*\mathscr O_{C}([z])$. By the moving lemma there is another divisor $D'$ of $L_{X_z}$ which does not have $z$ in its support. It follows the divisor ${\rm pr}_B^*(D')$ does not intersect with $X_z$. Thus $L_{X_z}$ restricts to a trivial line bundle on $X_z$. By the adjunction formula, $s|_{X_z}$ is a section of $K^{-1}_{X_z}$. We may also require that $a_z=0$ and there is no $D\in\mathscr B(X)^G\setminus\{X_z\}$ so that $h_D\not=0$ and $v_D\in\mathscr Q_{z,+}$. Thus every $G$-stable divisor in $X_z$ is the intersection of $X_z$ with a central $G$-stable divisor of $X$. It then follows from the previous discussion that the order of $s|_{X_z}$ is $1$ along any $G$-stable divisor of $X_z$. Combining with \cite[Theorem 4.2]{Brion97} we get
\begin{align*}
{\rm ord}_{\hat D}(s|_{X_z})=\bar m_D,
\end{align*}
whence $m_D=\bar m_D$ for every $D\in\mathscr D^B$. 

It remains to compute the $B$-weight of $s$. Since $s_Z$ is $T$-invariant, it suffices to compute the $B$-weight of $s_{P_u}$. Recall \eqref{local-structure}. Fix any $u\in P_u$ and $z\in Z$. For any $b=b_L\cdot u_b\in B$, where $b_L\in L_P\cap B$ and $u_b\in P_u$, on the fibre product $P\times^{L_P}Z\cong P_u\times Z$ it holds
$$b\cdot[(u,z)]=[(bu,z)]=[(b_L(u_bu)b_L^{-1},b_L\cdot z)],$$
where $b_L(u_bu)b_L^{-1}\in P_u$. Thus $B$ acts on $K_{P_u}^{-1}$ by the adjoint action and clearly the corresponding weight is $\kappa_P$. At the same time, $Z$ is a $T$-variety of complexity 1 and $s_Z$ is $(B\cap L_P)$-invariant. We see that the $B$-weight of $s$ is $\kappa_P$.

\end{proof}

\begin{rem}
Note that any (rational) $B$-semiinvariant section of $K_X^{-m}$ with weight $m\kappa_P$ differs from each other by a divisor of some $B$-invariant rational function. We see that the divisor of any such section is of form \eqref{anti-can-div}.
\end{rem}

\subsection{The quasihomogeneous case}
\subsubsection{Reduction to the one-parameter case}
In this section we deal with the quasihomogeneous case. Let $X$ be a quasihomogeneous $G$-variety of complexity 1, which contains the homogeneous space $G/H$ as an open $G$-orbit. To find an anti-canonical divisor of $X$ we first reduce the problem to the one-parameter case, which has already been solved above. Put
$$P':=\{g\in G|gD=D~\text{for any $B$-stable divisor $D$ in general position}\}.$$
Then $P'$ is a parabolic subgroup satisfying
\begin{align}\label{stab-parabolics}
B\subset P\subset P'\subsetneq G,
\end{align}
where the last inequality is strict since $X$ is quasihomogeneous. We will show that
\begin{lem}\label{stab-parabolics-lem}
$P'={\rm Stab}_G(D)$ for any $D\in\mathscr D^B$ in general position. Moreover, $P'$ stabilizes all colours with non-zero jump.
\end{lem}
Before the proof of Lemma \ref{stab-parabolics-lem}, let us recall some facts on colours with non-zero jump from \cite[Section 2.2]{Timashev-1997}. Note that via pull-back, there is a one-to-one correspondence between $B$-stable divisors in $G/H$ and $B$-semiinvariant sections (up to multiplication by an invertible function) of homogeneous line bundles on $G$. Here a homogeneous line bundle is $L_\chi:=G\times^H{\rm k}_\chi$, where ${\rm k}_\chi$ is the 1-dimensional $H$-representation with $\chi\in\mathfrak X(H)$. Recall that we have assumed $G$ is of simply connected type, then by \cite[Section 3]{KKLV} the $B$-semiinvariant section further reduces to a function in ${\rm k}[G]^{(B\times H)}$, which is a defining equation of the pull-back of this $B$-stable divisor.

With this correspondence, as showed in \cite[Section 2.2]{Timashev-1997}, the rational quotient for the $B$-action ${\rm pr}_B$ is defined by a 1-dimensional system of colours as follows: There is a $G$-linearized line bundle $L$ on $G/H$ and a 2-dimensional subspace $M\subset{\rm H}^0(G/H,L)^{(B)}_{\mu_0}\cong{\rm k}[G]^{(B\times H)}_{(\mu_0,\chi_0)}$ with $\mu_0\in\mathfrak X(B)$ and $\chi_0\in\mathfrak X(H)$ so that ${\rm pr}_B$ is given by
$${\rm pr}_B: G/H\dashrightarrow C:=\mathbb P^1(M^*)(\cong\mathbb P^1).$$
Clearly ${\rm pr}_B$ is a morphism on $X^o:=X\setminus{\rm Bs}(M)$, i.e. outside the base locus ${\rm Bs}(M)$ of $M$. Moreover, it is further shown in \cite[Section 2.2]{Timashev-1997} that $M$ (considered as a subset of ${\rm k}[G]$) consists of elements of the following types:
\begin{itemize}
\item Generic members in $M$ are indecomposable, which correspond to colours at a general position. These colours are called the \emph{regular} colours. Any regular colour $D$ satisfies $D={\rm pr}_B^*([x])$ (on $X^o$) for some $x\in C$, and is mapped to $v_D=q_{x}+\ell_D\in\mathscr E$. Moreover, $\ell_D=O$ for all but finitely many regular colours. Such a regular colour will be denoted by $X_{x}$. Conversely, for all but finitely many $x$ in $C$, the closure of ${\rm pr}_B^*([x])$ (considered as a divisor in $X^o$) in $G/H$ is the regular colour $X_x$.
\item There are finitely many lines in $M$ such that elements on these lines are decomposable. Components of their corresponding $B$-stable divisors are the remaining colours with non-zero jump. These colours are called the \emph{subregular} colours, and there are finitely many subregular colours. A subregular colour $D$ occurs in ${\rm pr}_B^*([x])$ (on $X^o$) for some $x\in C$ with multiplicity $h_D$, and is mapped to $v_D=h_Dq_{x}+\ell_D$.
\end{itemize}
The remaining colours in $G/H$ are precisely the central ones, i.e. those $D$ with $v_D\in\mathscr Q$.

Consider the variety $X$ that contains $G/H$. For any colour $D$ of $G/H$, its closure in $X$ is a colour of $X$, which will still be denoted by $D$ for short. Any colour in $X$ arises in this way. Obviously, a colour $D$ of $G/H$ is central if and only if its closure in $X$ is. Thus for convenience, below we will call a colour $D$ in a quasihomogeneous $G$-variety $X$ regular (subregular, central, resp.) if $D\cap(G/H)$ is regular (subregular, central, resp.).

\begin{proof}[Proof of Lemma \ref{stab-parabolics-lem}]
It suffices to prove the Lemma on the homogeneous space $G/H$. Under the above conventions, for any regular colour $D$, it holds
$${\rm Stab}_G(D)=P[\mu_0],$$
where $P[\mu_0]$ is the parabolic subgroup associate to $\mu_0$ (that is, the unipotent roots of $P[\mu_0]$ are precisely those positive roots not orthogonal to $\mu_0$). Clearly $P'=P[\mu_0]$ stabilises every line in $M$, hence stabilises every colour with non-zero jump. The Lemma is proved.
\end{proof}

Denote by $P'=L_{P'}P'_u$ its Levi decomposition so that $L_{P'}$ is a standard Levi subgroup. Take any $x\in C$ so that the only member in $\mathscr B(X)$ with non-zero jump that mapped into $\mathscr Q_{x,+}$ is the regular colour $X_x$, and consider
$$X'=X\setminus \{\text{$G$-orbits contained in}~X_x\}.$$
Then by \cite[Lemma 2.2]{Knop94}, $X_x$ is an effective $B$-stable Cartier divisor on $X'$. By \cite[Theorem 2.3]{Knop94},
$$X'\setminus X_x\cong P'_u\times^{L_{P'}}Z'\cong P'_u\times Z',$$
where $Z'$ is an $L_{P'}$-variety of complexity 1. Note that any colour $D$ in $X'$ of non-zero jump except $X_x$ (in fact every divisor in $\mathscr B(X)\setminus\{X_x\}$ with non-zero jump) descends to a $L_{P'}$-stable divisor in $Z'$ since $D$ is $P'$-stable. We conclude $Z'$ is a one-parameter $L_{P'}$-variety. Also any central colour $D$ in $X$ descends to a central $B\cap L_{P'}$-stable prime divisor $D'$ in $Z'$. Set
\begin{align}\label{bar-m-D-quasi}
\bar m_D=\left\{\begin{aligned}&1,~\text{if}~D'~\text{is}~L_{P'}\text{-stable},\\ &\bar m_{D'},~\text{if}~D'~\text{is a colour of}~Z', \end{aligned}\right.
\end{align}
where $\bar m_{D'}$ is given by \eqref{coe-colour-one-para} for $D'\subset Z'$. We say that a central colour $D\subset X$ is of type-a (a', b, resp.) if $D'$ is a central colour of $Z'$ of type-a (a', b, resp.).

\subsubsection{The anti-canonical divisor}

\begin{theo}\label{anti-can-div-thm-quasi-homo}
Let $X$ be a projective quasihomogeneous $G$-variety of complexity 1. Denote by $P\subset G$ its associated parabolic subgroup. Then there is a $B$-stable anti-canonical $\mathbb Q$-divisor of $X$,
\begin{align}\label{anti-can-div-quasi-homo}
\mathfrak d=\sum_{D\in\mathscr B(X), h_D\not=0}(1-h_D+h_Da_{x_D})D+\sum_{D\in\mathscr B(X)^G, h_D=0}D+\sum_{D\in\mathscr D^B, h_D=0}\bar m_DD,
\end{align}
where $\mathfrak a:=\sum_{x\in\mathbb P^1}a_x[x]$ is an anti-canonical $\mathbb Q$-divisor of $C$. Moreover, if $m\mathfrak d$ is a Cartier divisor for some $m\in\mathbb N_+$, then $m\mathfrak d$ is the divisor of a $B$-semiinvariant rational section of $K_X^{-m}$ with weight $m\kappa_P$.
\end{theo}

\begin{proof}
Recall the construction in the previous section. 
As before, $s$ can be decomposed into $s=s_{Z'}\wedge s_{P'_u}$, where $s_{Z'}$ is a $B\cap L_{P'}$-semiinvariant section of $K_{Z'}^{-1}$ and $s_{P'_u}$ a section of $K_{P'_u}^{-1}$. By \eqref{anti-can-div} we see that coefficients of $B$-stable prime divisors in ${\rm div}(s)$ are given by \eqref{anti-can-div-quasi-homo}, except possibly the one of $X_x$. More precisely,
\begin{align}\label{ax-unknown}
\mathfrak d=&\sum_{D\in\mathscr B(X)\setminus\{X_x\}, h_D\not=0}(1-h_D+h_Da_{x_D})D+\bar a_xX_x\notag\\
&+\sum_{D\in\mathscr B(X)^G, h_D=0}D+\sum_{D\in\mathscr D^B, h_D=0}\bar m_DD,
\end{align}
where $\bar a_x$ remains to be determined. At the same time, as in the one-parameter case, the $B$-weight of $s$ equals to
$$\sum_{\alpha\in \Pi_{P'_u}}\alpha+\sum_{\alpha\in\Pi_{L_{P'}}\cap\Pi_{P_u}}\alpha=\sum_{\alpha\in\Pi_{P_u}}\alpha=\kappa_P,$$
where the second term in the left-hand side appears when applying Theorem \ref{anti-can-div-thm} to the one-parameter $L_{P'}$-variety ${Z'}$.

We claim that
\begin{align*}
\bar a_x=a_x,
\end{align*}
and hence the coefficient of $X_x$ is also given by \eqref{anti-can-div-quasi-homo}. It suffices to show that
\begin{align}\label{deg(a)-lem}
\bar a_x+\sum_{y\in C\setminus\{x\}}a_y=2.
\end{align}
We divide the proof into two steps:

\emph{Step-1. The case when $X$ is $\mathbb Q$-Gorenstein.} Without loss of generality, we may assume that $X$ is Gorenstein so that $K_X^{-1}$ is Cartier. Since $X$ is projective, there is a very ample $G$-linearized line bundle $L$ on $X$. Let $$\mathfrak d_L=\sum_{D\in\mathscr B(X)}m_DD$$ be the ($B$-stable) divisor some section $s_L\in{\rm H}^0(X,L)^{(B)}_{\lambda_0}$ for some $\lambda_0\in\mathfrak X(B)$.

Recall the anti-canonical divisor \eqref{ax-unknown}. For convenience in the following we also denote $\bar m_D=1$ for $D\in\mathscr B(X)^G\cap\{h_D=0\}$ in \eqref{ax-unknown}.
Consider a Cartier divisor
$$\mathfrak d_\epsilon:=\mathfrak d_L+\epsilon\mathfrak d,~0\leq\epsilon\ll1.$$
Then $\mathfrak d_\epsilon$ is a divisor of the ample line bundle $L_\epsilon:=L-\epsilon K_X$, which is the divisor of some (rational) section of $L_\epsilon$ with weight $\lambda_0+\epsilon\kappa_P$. Clearly, the corresponding piecewise concave functions
\begin{align*}
A_y(\mathfrak d_\epsilon,\lambda)=\min_{x_D=x}\frac{m_D+\epsilon(1-h_D+h_D\bar a_x)+\ell_D(\lambda)}{h_D},~\forall\lambda\in\Gamma_\mathbb R.
\end{align*}
By \cite[Theorem 8]{Timashev-2000},
\begin{align*}
L_\epsilon^{\cdot n}=n!\int_{\Delta_\mathscr Z(\mathfrak d_\epsilon)}A(\mathfrak d_\epsilon,\lambda)\pi(\lambda+\lambda_0+\epsilon\kappa_P)d\lambda,
\end{align*}
where $A(\mathfrak d_\epsilon,\lambda)=\sum_{x\in C}A_x(\mathfrak d_\epsilon,\lambda)$, and
\begin{align*}
\Delta_\mathscr Z(\mathfrak d_\epsilon)=\cap_{D\in\mathscr B(X),h_D=0}\{m_D+\epsilon\bar m_D+\ell_D(\lambda)\geq0\}\cap\{A(\mathfrak d_\epsilon,\lambda)\geq0\}.
\end{align*}
Denote by $F_D$ the facet of $\Delta(\mathfrak d_L)$ that lies on $\{\ell_D(\lambda)+m_D=0\}$ for a central $D\in\mathscr B(X)$, and $\{F_{D_b}|b=1,...,M\}$ the facets of $\Delta_\mathscr Z(\mathfrak d_L)$ where $A(\mathfrak L,\cdot)\not=0$. Also denote by $\{\Omega'_a\}_{a=1}^N$ the common domains of linearity of all $\{A_y(\mathfrak d_L,\lambda)|y\in C\}$ so that for each $y\in C$,
$$A_y(\mathfrak d_L,\lambda)=\frac{m_{D_a(y)}+\ell_{D_a(y)}(\lambda)}{h_{D_a(y)}},~\forall\lambda\in\Omega'_a.$$
Taking variation in the previous integration we get\footnote{To be precise, we add a short explanation to \eqref{var-int}. In the last summation term it may happen $F_{D}=F_{D'}$ for different $D, D'\in\mathscr B(X)$ with $h_D=h_{D'}=0$. If one of them is of type-a' or b, then $\pi(\lambda+\lambda_0)=0$ on this facet. Otherwise $n_D\bar m_D=n_{D'}\bar m_{D'}=1$. Thus each term in the last sum does not depend on the choice in $\{D,D'\}$.}
\begin{align}\label{var-int}
\frac1{(n-1)!}K_X^{-1}\cdot L^{\cdot(n-1)}=&\frac1{n!}\left.\frac\partial{\partial\epsilon}\right|_{\epsilon=0}L_\epsilon^{\cdot n}\notag\\
=&\sum_{y\in C\setminus\{x\}}\sum_{a=1}^N\int_{\Omega'_a}(\frac1{h_{D_a(y)}}-1+a_y)\pi(\lambda+\lambda_0)d\lambda+\bar a_x\notag\int_{\Delta_\mathscr Z(\mathfrak d_L)}\pi(\lambda+\lambda_0)d\lambda\\
&+\int_{\Delta_\mathscr Z(\mathfrak d_L)}A(\mathfrak d,\lambda)\langle\nabla\pi(\lambda+\lambda_0),\kappa_P\rangle d\lambda\notag\\
&+\sum_{b=1}^Mn_{D_b}\bar m_{D_b}\int_{F_{D_b}}A(\mathfrak d,\lambda)\pi(\lambda+\lambda_0)d\sigma,
\end{align}
where the measure $d\sigma|_{F_D}$ is the induced lattice measure on the facet $F_D$, and $n_D\in\mathbb Q_+$ so that $n_Dv_D$ is primitive. Note that $n_D\bar m_D\not=1$ only if $D$ descends to a colour $D'$ of type-a' or b of $Z'$. Suppose also that $D'\in\mathscr D^B(\alpha)$. As $P_\alpha$ moves $D'$ (hence also moves $D$), $\alpha\in\Pi_{P_u}$, where $P$ is the associated parabolic subgroup of $X$ defined by \eqref{ass-para}. Thus there is a factor $\langle\alpha^\vee,\lambda+\lambda_0\rangle$ in $\pi(\lambda+\lambda_0)$, and by Lemma \ref{face-vanishing} in the Appendix,
$$\int_{F_D}A(\mathfrak d,\lambda)\pi(\lambda+\lambda_0)d\sigma=0.$$
Also, applying Lemma \ref{nabla-pi} to the second last term of \eqref{var-int}, one gets
\begin{align*}
&\frac1{(n-1)!}K_X^{-1}\cdot L^{\cdot(n-1)}\notag\\=&\sum_{y\in C}\sum_{a=1}^N\int_{\Omega'_a}(\frac1{h_{D_a(y)}}-1)\pi(\lambda+\lambda_0)d\lambda+\left(\bar a_x+\sum_{y\in C\setminus\{x\}}a_y\right)\cdot\int_{\Delta_\mathscr Z(\mathfrak d)}\pi(\lambda+\lambda_0)d\lambda\notag\\
&+2\int_{\Delta_\mathscr Z(\mathfrak d)}A(\mathfrak d,\lambda)\langle\nabla\pi(\lambda+\lambda_0),\rho\rangle d\lambda+\int_{\partial\Delta_\mathscr Z(\mathfrak d)}A(\mathfrak d,\lambda)\pi(\lambda+\lambda_0)d\sigma.
\end{align*}
Here we used in the first term in \eqref{var-int} the fact that the only divisor in $\mathscr B(X)\cap\{x_D=x\}$ with non-zero jump is $X_x$, which has $h_{X_x}=1$, and as a consequence
$$\sum_{y\in C\setminus\{x\}}\sum_{a=1}^N\int_{\Omega_a}(\frac1{h_{D_a(y)}}-1)\pi(\lambda+\lambda_0)d\lambda=\sum_{y\in C}\sum_{a=1}^N\int_{\Omega_a}(\frac1{h_{D_a(y)}}-1)\pi(\lambda+\lambda_0)d\lambda.$$
On the other hand, by the Riemann-Roch formula (cf. \cite[Appendix A]{Boucksom-Hisamoto-Jonsson}),
\begin{align*}
\frac1{(n-1)!}K_X^{-1}\cdot L^{\cdot(n-1)}=\lim_{k\to+\infty}\frac2{k^{n-1}}(\dim{\rm H}^0(X,L^k)-\frac{L^{\cdot n}}{n!}k^n),
\end{align*}
where again by \cite[Theorem 8]{Timashev-2000},
$$L^{\cdot n}=n!\int_{\Delta_\mathscr Z(\mathfrak d)}A(\mathfrak d,\lambda)\pi(\lambda+\lambda_0)d\lambda.$$
Comparing with \eqref{h0(X,Lk)-eq} we get \ref{deg(a)-lem}.

\emph{Step-2. The general case.} For general $X$, by a deep result of Hironaka (cf. \cite{Hironaka-1964,Bogolomov-Pantev-1996}) and Koll\'ar \cite[Chapter 3]{Kollar-lecture}, we can take a $G$-equivariant resolution of singularities $\gamma:\tilde X\to X$ so that $\tilde X$ is also projective. Note that all exceptional divisors are in $\mathscr B(\tilde X)^G$. Hence the strict transformation of $X_x$ is $\tilde X_x$. Also, the difference between $K_{\tilde X}^{-1}$ and the strict transformation of $K_X^{-1}$ consists of only exceptional divisors. Apply the result in \emph{Step-1} to $\tilde X$ we get \eqref{deg(a)-lem} again. Hence we get the Theorem.
\end{proof}

\begin{rem}\label{coef-relation}
Set
$$\kappa_{P'}:=\sum_{\alpha\in \Pi_{P'_u}}\alpha~\text{and}~\kappa_{L_{P'}\cap P}:=\sum_{\alpha\in\Pi_{L_{P'}}\cap\Pi_{P_u}}\alpha.$$
In the quasihomogeneous case, by construction it holds (cf. \eqref{coe-colour-one-para})
\begin{align*}
\bar m_D=\left\{\begin{aligned}&\frac12\langle\alpha^\vee,\kappa_{L_{P'}\cap P}\rangle,~\text{if}~D'~\text{is of type-a or a' in $Z'$},\\ &\langle\alpha^\vee,\kappa_{L_{P'}\cap P}\rangle,~\text{if}~D'~\text{is of type-b in $Z'$,}  \end{aligned}\right.
\end{align*}
and $\alpha$ is a simple root in $\Pi_{L_{P'}}$. Note that the set $\Pi_{P'_u}=\Pi_G\setminus\Pi_{L_{P'}}$, which is invariant under the Weyl group $W(L_{P'})$ generated by simple roots in $\Pi_{L_{P'}}$. The weight $\kappa_{P'}$ is $W(L_{P'})$-invariant, and hence orthogonal to $\Pi_{L_{P'}}$. Since $\kappa_P=\kappa_{P'}+\kappa_{L_{P'}\cap P}$,
we have
\begin{align*}
\bar m_D=\left\{\begin{aligned}&\frac12\langle\alpha^\vee,\kappa_{P}\rangle,~\text{if}~D'~\text{is of type-a or a' in $Z'$},\\ &\langle\alpha^\vee,\kappa_{P}\rangle,~\text{if}~D'~\text{is of type-b in $Z'$.}  \end{aligned}\right.
\end{align*}
\end{rem}

\section{Equivariant test configurations}
\subsection{The structure of an equivariant test configuration} Let $(X,L)$ be a polarized, projective $G$-variety of complexity 1 with ${\rm k}(X)^B\cong{\rm k}(\mathbb P^1)$, and $\mathfrak d$ an effective divisor of $L$. In this section we classify $G$-equivariant normal test configurations of $(X,L)$ with irreducible central fibres. The case of $T$-varieties of complexity 1 has been solved by \cite{Ilten-Suss-Duke}.

Let $(\mathcal X,\mathcal L)\overset{{\rm pr}}{\to}\mathbb P^1$ be a $G$-equivariant normal test configuration of $(X,L)$ with index $r_0$. Then $(\mathcal X,\mathcal L)$ is a polarized projective $G\times{\rm k}^\times$-variety of complexity 1.\footnote{We do not write $\hat G$ instead of $G\times{\rm k}^\times$ here in order not to be confounded with the previous section.} It is clear that the hyperspace $\bar{\mathscr E}$ of $\mathcal X$ is $\bar{\mathscr E}\cong\mathscr E\times\mathbb Q$, and the valuation cone $\bar{\mathscr V}\cong\mathscr V\times\mathbb Q$. We have

\begin{prop}\label{tc-to-v0}
Suppose that $(\mathcal X,\mathcal L)\overset{{\rm pr}}{\to}\mathbb P^1$ is a $G$-equivariant normal test configuration of $(X,L)$ and the reduced structure $\mathcal X_0^{\rm red}$ of the central fibre $\mathcal X_0:={\rm pr}^{-1}(0)$ is irreducible. Then $\mathcal X_0^{\rm red}$ is mapped to some primitive $(v_0,m)\in\mathscr V\times\mathbb Z_{<0}\subset\bar{\mathscr E}$, and ${\rm pr}^{-1}(0)=-m\mathcal X_0^{\rm red}$. Consequently, $\mathcal X$ has integral central fibre if and only if $m=-1$ and $v_0$ is integral.
\end{prop}
\begin{proof}
By definition of a $G$-equivariant test configuration, there are exactly three kinds of $B\times{\rm k}^\times$-stable divisors in $\mathcal X$:
\begin{itemize}
\item The fibre $\mathcal X_\infty$ at $\infty$, which is isomorphic to $X$. We have $v_{\mathcal X_\infty}=(0,1)$;
\item The closure $\bar D$ of $D\times{\rm k}^\times$ in $\mathcal X$ for $D\in\mathscr B(X)$. Clearly $v_{\bar D}=(v_D,0)$;
\item The reduced structure $\mathcal X_0^{\rm red}$ of the central fibre ${\rm pr}^{-1}(0)$, which is a $G\times{\rm k}^\times$-stable prime divisor by our assumption. Thus $v_{\mathcal X_0^{\rm red}}$ is a primitive vector in $\bar{\mathscr V}$, and can be written as
$$v_{\mathcal X_0^{\rm red}}=(v_0,m)$$ for some $v_0\in\mathscr V$ and $m\in\mathbb Z$.
\end{itemize}

Now we show ${\rm pr}^{-1}(0)=-m\mathcal X_0^{\rm red}$. Suppose that $\mathfrak d$ is a divisor of $L$ defined by \eqref{B-stable-divisor} with respect to some weight $\lambda_0$ and $(\mathcal X,\mathcal L)$ has index $r_0$. Then $\mathcal L$ has a $B\times{\rm k}^\times$-stable divisor
\begin{align}\label{div-mathcal-L}
\mathfrak D=r_0(\sum_{D\in\mathscr B(X)}m_D\bar D)+m_0\mathcal X_0^{\rm red}+m_\infty\mathcal X_\infty.
\end{align}
Since $[0]-[\infty]$ is a principle divisor on $\mathbb P^1$, up to linear equivalence we may assume that $m_\infty=0$. Also up to replacing $L$ by $L^{r_0}$ we can assume that the index $r_0=1$, and $\mathcal X$ can be embedded in certain projective space by sections of $\mathcal L$ as a projectively normal variety. In this case let $\hat{\mathcal X}$ be the affine cone over $\mathcal X$. Then $\hat{\mathcal X}$ is an affine normal $\widehat{G\times{\rm k}^\times}(:=G\times{\rm k}^\times\times{\rm k}^\times)$-variety of complexity 1. Here the middle ${\rm k}^\times$-factor stands for the ${\rm k}^\times$-action on $\mathcal X$ and the last stands for the cone direction. Also, the projection ${\rm pr}$ on the test configuration induces a projection $\hat{\rm pr}:\hat{\mathcal X}\to\mathbb P^1$.



Consider a $\widehat{B\times{\rm k}^\times}$-chart $\hat{\mathcal U}_0$ intersecting $\hat{\mathcal X}_0^{\rm red}$ with data $(\hat{\mathscr W}_0,\hat{\mathscr R}_0)$. Then $(v_0,m,-m_0)\in \hat{\mathscr W}_0$. Denote by $t$ the coordinate on $\mathbb P^1$. Since $\hat{\rm pr}$ is regular, it must hold $\hat{\rm pr}^*(t)\in{\rm k}[\hat{\mathcal U}_0]$ and ${\rm ord}_{\hat{\mathcal X}_0^{\rm red}}(\hat{\rm pr}^*(t))>0$. Any divisor in $(\hat{\mathscr W}_0\sqcup\hat{\mathscr R}_0)\setminus\{\hat{\mathcal X}_0^{\rm red}\}$ is of form $\hat{\bar D}$ and
$$v_{\hat{\bar D}}(\hat{\rm pr}^*(t))=\langle(v_D,0,-m_D),(0,-1,0)\rangle=0,~\forall D\in\mathscr B(X).$$
Together with
$$v_{\hat{\mathcal X}_0^{\rm red}}(\hat{\rm pr}^*(t))=\langle(v_0,m,-m_0),(0,-1,0)\rangle=-m,$$
we see that $m\in\mathbb Z_{<0}$ and $\hat{\rm pr}^{-1}(0)=-m\hat{\mathcal X}_0^{\rm red}$. 
The Proposition then follows.
\end{proof}

\begin{rem}\label{base-change-rmk}
Let $(\mathcal X,\mathcal L)$ be a $G$-equivariant normal test configuration which is associated to data $(v_0,m)$ as in Proposition \ref{tc-to-v0}. Let $\mathfrak F_{\mathcal X}$ be its coloured fan. Taking base change $t\to t^{-m}$ (note that $m<0$) and then a normalization, we get another $G$-equivariant normal test configuration $(\mathcal X^{(-m)},\mathcal L^{(-m)})$. Under this process, each ray $\mathbb Q_{\geq0}(v,u)\subset\bar{\mathscr E}\cong\mathscr E\times\mathbb Q$ is mapped to $\mathbb Q_{\geq0}(-mv,u)$. In particular $\mathbb Q_{\geq0}(v_0,m)$ is mapped to $\mathbb Q_{\geq0}(v_0,-1)$ which has primitive generator $(v_0,-1)$. Thus $(\mathcal X^{(-m)},\mathcal L^{(-m)})$ is the $G$-equivariant normal test configuration associated to $(v_0,-1)$. In particular it has integral central fibre.
\end{rem}

\subsection{Filtration induced by the test configuration}
Suppose that $(\mathcal X,\mathcal L)$ is the $G$-equivariant normal test configuration of $(X,L)$ with associated data $(v_0,m)$ as in Proposition \ref{tc-to-v0}, and \eqref{div-mathcal-L} gives a divisor of $\mathcal L$ (with $m_\infty=0$). Then by \cite{Popov-1986,Witt-Ny} $(\mathcal X,\mathcal L)$ induces a $G$-equivariant ($\mathbb Z$-)filtration $\mathscr F_{(\mathcal X,\mathcal L)}$ on the Kodaira ring $R(X,L)$ by setting
\begin{align}\label{Fil-XL}
\mathscr F_{(\mathcal X,\mathcal L)}^\tau R_k:=\{\sigma\in R_k|t^{-\tau}\sigma\in R(\mathcal X,\mathcal L)\},~\forall k,\tau\in\mathbb N,
\end{align}
and
\begin{align*}
\mathscr F_{(\mathcal X,\mathcal L)}^\tau R_k:=\mathscr F_{(\mathcal X,\mathcal L)}^{\lceil\tau\rceil}R_k,~\forall k\in\mathbb N~\text{and}~\tau\in\mathbb R.
\end{align*}
Then $\mathscr F_{(\mathcal X,\mathcal L)}$ is left-continuous and decreasing. It is further multiplicative, pointwise left-bounded at $\tau=0$ and linearly right-bounded by the finite generation of $R(X,L)$ (cf. \cite[Proposition 6.4]{Witt-Ny}). Since each $\mathscr F_{(\mathcal X,\mathcal L)}^\tau R_k$ is a finite dimensional $G$-invariant linear space, to determine $\mathscr F_{(\mathcal X,\mathcal L)}^\tau R_k$ it suffices to determine all $B$-semiinvariants in it.
\begin{lem}\label{FsRk-prop}
With the above conventions, for any $s_\lambda\in{\rm H}^0(X,L^k)^{(B)}_\lambda$,  $s_\lambda$ lies in $\mathscr F_{(\mathcal X,\mathcal L)}^\tau R_k$ if and only if
\begin{align*}
v_0(s_\lambda)+\tau m+km_0\geq0.
\end{align*}
Here $m_0$ is the coefficient given in \eqref{div-mathcal-L}.
\end{lem}
\begin{proof}
By \eqref{Fil-XL}, $s_\lambda\in\mathscr F_{(\mathcal X,\mathcal L)}^\tau R_k$ if and only if $t^{-\tau}s_\lambda\in{\rm H}^0(\mathcal X,\mathcal L^k)$. The Proposition then follows from \eqref{div-mathcal-L}.
\end{proof}

Suppose that $R(X,L)$ is generated by the first piece $R_1$ over $R_0\cong{\rm k}$. Fix the section $s_0\in (R_1)^{(B)}_{\lambda_0}$ so that the divisor of $s_0$ is $\mathfrak d$ given by \eqref{B-stable-divisor}. Then any $s_\lambda\in{\rm H}^0(X,L^k)^{(B)}_\lambda$ can be written as
\begin{align}\label{sigma-lambda-decomp}
s_\lambda=f_0e_{\lambda-k\lambda_0}s_0^k
\end{align}
for some $f_0\in{\rm k}(X)^B$. By Lemma \ref{FsRk-prop} we have
\begin{prop}\label{F-tau-Rk}
Suppose that $v_0=\ell_0+h_0q_{x_0}\in\mathscr Q_{x_0,+}$ for some $x_0\in C$ (if $h_0=0$ then $v_0\in\mathscr Q\subset\mathscr Q_{x,+}$ for all $x\in C$). Then:
\begin{itemize}
\item[(1)] If $h_0=0$,
\begin{align*}
\dim(\mathscr F_{(\mathcal X,\mathcal L)}^\tau R_k)^{(B)}_\lambda=&\max\{0,\deg(\delta_k(\lambda))+1\},\\~&\text{for}~\lambda\in k\Delta_\mathscr Z(L)\cap(\Gamma+k\lambda_0)~\text{and}~\tau m+km_0+\ell_0(\lambda-k\lambda_0)\geq0,
\end{align*}
with
\begin{align}\label{delta(d,lambda)-tsentr}
\delta_k(\lambda):=\sum_{x\in C}[kA_x(\mathfrak d,\frac\lambda k-\lambda_0)]\cdot x,
\end{align}
and $$\dim(\mathscr F_{(\mathcal X,\mathcal L)}^\tau R_k)^{(B)}_\lambda=0,~\text{otherwise};$$
\item[(2)] If $h_0\not=0$,
\begin{align*}
\dim(\mathscr F_{(\mathcal X,\mathcal L)}^\tau R_k)^{(B)}_\lambda=\max\{0,\deg(\delta_k(\lambda,\tau))+1\},~\text{for}~\lambda\in k\Delta_\mathscr Z(L)\cap(\Gamma+k\lambda_0),
\end{align*}
with
\begin{align}\label{delta(d,lambda)}
\delta_k(\lambda,\tau):=&[\min\{kA_{x_0}(\mathfrak d,\frac\lambda k-\lambda_0),\frac{\tau m+km_0+\ell_0(\lambda-k\lambda_0)}{h_0}\}]\cdot x_0\notag\\
&+\sum_{x\in C,x\not=x_0}[kA_x(\mathfrak d,\frac\lambda k-\lambda_0)]\cdot x,
\end{align}
and $$\dim(\mathscr F_{(\mathcal X,\mathcal L)}^\tau R_k)^{(B)}_\lambda=0~\text{for}~\lambda\not\in k\Delta_\mathscr Z(L).$$
\end{itemize}
\end{prop}

\begin{proof}
By \eqref{sigma-lambda-decomp} and Lemma \ref{FsRk-prop}, $s_\lambda\in \mathscr F_{(\mathcal X,\mathcal L)}^\tau R_k$ if and only if
\begin{align*}
0\leq&v_0(s_\lambda)+\tau m+km_0
=\ell(\lambda-k\lambda_0)+h_0q_{x_0}(f_0)+\tau m+km_0,
\end{align*}
and
\begin{align*}
v_D(s_\lambda)=\ell_D(\lambda-k\lambda_0)+h_Dq_{x_D}(f_0)+km_D\geq0,~\forall D\in\mathscr B(X).
\end{align*}
Here the second relation confirms $s_\lambda\in R_k$. Thus $s_\lambda\in \mathscr F_{(\mathcal X,\mathcal L)}^\tau R_k$ if and only if
\begin{align*}
h_0q_{x_0}(f_0)\geq-\tau m-km_0-\ell_0(\lambda-k\lambda_0),
\end{align*}
and
\begin{align*}
h_{x_D}q_{x_D}(f_0)\geq-km_D-\ell_D(\lambda-k\lambda_0),~\forall D\in\mathscr B(X).
\end{align*}
Combining with Proposition \ref{Delta(d)-Z}, this is equivalents to $\lambda\in k\Delta_\mathscr Z(L)$ and $f_0\in {\rm H}^0(C,\delta(\lambda,\tau))$. Thus we get the Proposition.

\end{proof}

It is known that the Kodaira ring of the central fibre $\mathcal X_0$ with respect to $\mathcal L_0:=\mathcal L|_{\mathcal X_0}$ is the graded algebra of the filtration $\mathscr F_{(\mathcal X,\mathcal L)}$,
\begin{align}\label{Gr(F)}
{\rm Gr}(\mathscr F_{(\mathcal X,\mathcal L)})=\bigoplus_{k=0}^{+\infty}\bigoplus_{\tau=0}^{+\infty}\mathscr F_{(\mathcal X,\mathcal L)}^\tau R_k/\mathscr F_{(\mathcal X,\mathcal L)}^{>\tau}R_k,
\end{align}
and the induced $\rm k^\times$-action acts on the $(\mathscr F_{(\mathcal X,\mathcal L)}^\tau R_k/\mathscr F_{(\mathcal X,\mathcal L)}^{>\tau}R_k)$-piece with weight $\tau$. In the following we also say that a non-zero element $s\in\mathscr F_{(\mathcal X,\mathcal L)}^\tau R_k\setminus\mathscr F_{(\mathcal X,\mathcal L)}^{>\tau}R_k$ has weight $\tau$ (induced by $(\mathcal X,\mathcal L)$). Note that each $(\mathscr F_{(\mathcal X,\mathcal L)}^\tau R_k/\mathscr F_{(\mathcal X,\mathcal L)}^{>\tau}R_k)$-piece in \eqref{Gr(F)} is also a ${\rm k}^\times$-module of weight $\tau$. Thus
$$(\mathscr F_{(\mathcal X,\mathcal L)}^\tau R_k/\mathscr F_{(\mathcal X,\mathcal L)}^{>\tau}R_k)^{(B)}_\lambda\subset{\rm Gr}(\mathscr F_{(\mathcal X,\mathcal L)})$$
is a $B\times{\rm k}^\times$-module of weight $(\lambda,\tau)$ in the degree $k$-piece of \eqref{Gr(F)}.

For the study of K-stability, its suffices to consider only special test configurations (cf. \cite{Li-Xu-Annl}). In this case, $v_0\in\mathscr E$ is an integral element and $m=-1$. Using Proposition \ref{F-tau-Rk} we get
\begin{cor}\label{central-spherical}
Suppose that $(\mathcal X,\mathcal L)$ is a special test configuration associated to some integral $v_0\in\mathscr E$. Then the central fibre $\mathcal X_0$ is a $G\times{\rm k}^\times$-spherical variety if and only if $v_0$ is not a central valuation. That is, $v_0=\ell_0+h_0q_{x_0}\in\mathscr Q_{x_0,+}$ for some $x_0\in C$ and $h_0\not=0$.
\end{cor}

\begin{proof}
Note that when $(\mathcal X,\mathcal L)$ is special, the central fibre $\mathcal X_0$ is a normal $G\times{\rm k}^\times$-variety embedded in projective space by sections of $\mathcal L_0$. Up to replace $\mathcal L$ (as well as $\mathcal L_0$) by a sufficiently large multiple, we can assume that $\mathcal X_0$ is embedded as a projectively normal variety. Denote by $\hat {\mathcal X}_0:={\rm Spec}({\rm Gr}(\mathscr F_{(\mathcal X,\mathcal L)}))$ the affine cone of $\mathcal X_0$. Then it admits a $({\rm k}^\times\times G\times {\rm k}^\times)$-action, where the first ${\rm k}^\times$-factor stands for the cone direction (the homothety). Clearly the $({\rm k}^\times\times G\times {\rm k}^\times)$-invariants in ${\rm Gr}(\mathscr F_{(\mathcal X,\mathcal L)})$ is the 0-th piece, which is isomorphic to ${\rm k}$.

We will show that
\begin{align}\label{mul-X0}
\dim(\mathscr F_{(\mathcal X,\mathcal L)}^\tau R_k/\mathscr F_{(\mathcal X,\mathcal L)}^{>\tau}R_k)^{(B)}_\lambda\leq1,~\forall k,\tau\in\mathbb N~\text{and}~\lambda\in\mathfrak X(B),
\end{align}
if and only if $v_0$ is not central. Once \eqref{mul-X0} holds, ${\rm Gr}(\mathscr F_{(\mathcal X,\mathcal L)})$ will be multiplicity-free, and by \cite[Theorem 5.16]{Timashev-book}, $\hat{\mathcal X}_0$ is a $({\rm k}^\times\times G\times {\rm k}^\times)$-variety of complexity 0. Consequently, $\hat {\mathcal X}_0$ will be a $({\rm k}^\times\times G\times {\rm k}^\times)$-spherical variety by \cite{Vinberg-Kimel'fel'd}. Taking quotient we get $\mathcal X_0$ is a $G\times {\rm k}^\times$-spherical variety.

It remains to check \eqref{mul-X0}. Suppose that $v_0=\ell_0+h_0q_{x_0}\in\mathscr Q_{x_0,+}$ for some $x_0\in C$ and $h_0\not=0$. Since $v_0$ is integral, $h_0\geq1$. For any $ k,\tau\in\mathbb N$ and $\lambda\in k\Delta_\mathscr Z(L)\cap(\Gamma+k\lambda_0)$, it holds
\begin{align}\label{pt-disconti-t-GrR}
0\leq&\deg(\delta_k(\lambda,\tau))-\deg(\delta_k(\lambda,\tau+1))\notag\\=&[k\min\{A_{x_0}(\mathfrak d,\frac\lambda k-\lambda_0),\frac{-\frac\tau k+m_0+\ell_0(\frac\lambda k-\lambda_0)}{h_0}\}]\notag\\
&-[k\min\{A_{x_0}(\mathfrak d,\frac\lambda k-\lambda_0),\frac{-\frac{\tau+1} k+m_0+\ell_0(\frac\lambda k-\lambda_0)}{h_0}\}]\leq\max\{\frac1{h_0},1\}\leq1.
\end{align}
By Proposition \ref{F-tau-Rk} (2),
\begin{align*}
\dim(\mathscr F_{(\mathcal X,\mathcal L)}^\tau R_k/\mathscr F_{(\mathcal X,\mathcal L)}^{>\tau}R_k)^{(B)}_\lambda=&\dim(\mathscr F_{(\mathcal X,\mathcal L)}^\tau R_k/\mathscr F_{(\mathcal X,\mathcal L)}^{\tau+1}R_k)^{(B)}_\lambda\leq1,
\end{align*}
and we get \eqref{mul-X0}. Hence $\mathcal X_0$ is a $G\times{\rm k}^\times$-spherical vareity.

When $v_0$ is a central element, by Proposition \ref{F-tau-Rk} (1) we have
\begin{align*}
\dim(\mathscr F_{(\mathcal X,\mathcal L)}^\tau R_k/\mathscr F_{(\mathcal X,\mathcal L)}^{\tau+1}R_k)^{(B)}_\lambda=\dim (R_k)^{(B)}_\lambda\chi_{\{\tau+1> m_0+\ell_0(\frac\lambda k-\lambda_0)\geq\tau\}},
\end{align*}
where by $\chi_S$ we denote the characteristic function of a set $S$. Suppose that $\mathcal X_0$ is spherical. Then
$$\dim(\mathscr F_{(\mathcal X,\mathcal L)}^\tau R_k/\mathscr F_{(\mathcal X,\mathcal L)}^{\tau+1}R_k)^{(B)}_\lambda\leq1,~\forall\lambda\in k\Delta_\mathscr Z(L)\cap(\Gamma+k\lambda_0),$$
which is equivalent to
$$\dim (R_k)^{(B)}_\lambda
\leq1,~\forall k\in\mathbb N~\text{and}~\lambda\in k\Delta_\mathscr Z(L)\cap(\Gamma+k\lambda_0).$$
Again, by \cite[Theorem 5.16]{Timashev-book} and \cite{Vinberg-Kimel'fel'd} we conclude that $X$ itself is a $G$-spherical variety, a contradiction.
\end{proof}

\subsubsection{Product test configurations}

\begin{prop}\label{Aut-cone-prop}
Let $(X,L)$ be a polarized $G$-variety of complexity 1. Then for each $x\in C$, there is a cone $\mathscr A_x\subset \mathscr V_x$ which is a face of $\mathscr V$ so that $(v_0,-1)$ defines a product test configuration of $(X,L)$ if and only if $v_0\in\mathscr A_x$ for some $x\in C$. Moreover, each $\mathscr A_x$ intersects with $\mathscr Q$ with a common cone $\mathscr V\cap(-\mathscr V)\cap\mathscr Q(=:\mathscr A)$.
\end{prop}

\begin{proof}
We see that the test configuration $(\mathcal X,\mathcal L)$ associated to $(v_0,-1)$ is a product test configuration if and only if its centre $\mathcal X_0\cong X$, or equivalently,
\begin{align}\label{prod-centre}
{\rm Gr}(\mathscr F_{(\mathcal X,\mathcal L)})\cong R(X,L).
\end{align}
In the following we write $R$ instead of $R(X,L)$ in short. Since $t\in {\rm k}(\mathcal X)^G$. For any non-zero
\begin{align}\label{choice-si}
s_i\in (R_{k_i})^{(B)}_{\lambda_i},~i=1,...,p
\end{align}
and
$$\tau_i=v_0(s_i)+k_im_0,~i=1,...,p,$$
$t^{-\tau_i}s_i$ lies in the Rees algebra of $\mathscr F_{(\mathcal X,\mathcal L)}$ (cf. \cite{Popov-1986}) and maps to a non-zero element in ${\rm Gr}(\mathscr F_{(\mathcal X,\mathcal L)})$.
Note that in the Rees algebra,
$$\prod_{i=1}^p\langle G\cdot t^{-\tau_i}s_i\rangle=t^{-\sum_{i=1}^p\tau_i}\prod_{i=1}^p\langle G\cdot s_i\rangle.$$
Here as in Section 2.2, for an $s\in R$, we denote by $\langle G\cdot s\rangle$ the linear span of the $G$-orbit $G\cdot s$ in $R$. Then \eqref{prod-centre} holds if and only if $t^{-\sum_{i=1}^p\tau_i}s'$ maps to a non-zero element in ${\rm Gr}(\mathscr F_{(\mathcal X,\mathcal L)})$ for any non-zero $s'\in (R_{\sum_{i=1}^pk_i})^{(B)}_{\lambda'}\cap \prod_{i=1}^p\langle G\cdot s_i\rangle$. In particular $\frac{s'}{\prod_{i=1}^ps_i}$ is a tail vector of $R(X,L)$. By \eqref{Gr(F)}, this is equivalent to
$$s'\not\in\mathcal F^{>\sum_{i=1}^p\tau_i}R_{\sum_{i=1}^pk_i},$$
for any such $s'$. On the other hand, it always holds $s'\in\mathcal F^{\sum_{i=1}^p\tau_i}R_{\sum_{i=1}^pk_i}$. Combining with Lemma \ref{FsRk-prop}, we get
\begin{align*}
v_0(s')-\sum_{i=1}^p\tau_i+\sum_{i=1}^pk_im_0=0,
\end{align*}
and we conclude that for any choice of $s_i$'s in \eqref{choice-si},
$$v_0(\frac{s'}{\prod_{i=1}^ps_i})=0,~\forall s'\in\prod_{i=1}^p\langle G\cdot s_i\rangle~\text{such that}~\frac{s'}{\prod_{i=1}^ps_i}~\text{is a tail vector of}~R(X,L).$$
Since the fraction field of $R(X,L)$ is ${\rm k}(X)$, by Proposition \ref{tail-valuation-cone-dual}, $v_0$ lies in a face of some $\mathscr V_x$. In fact, the above equations define an intersection of certain linear subspaces of $\mathscr Q\times\mathbb Q$ and the (upper) half-space $\mathscr Q_{x,+}$. On the other hand, as $\mathscr Q(=\mathscr Q\times\{0\})$ is a  linear subspace, the above equations cut out a linear subspace in it. Thus it holds
\begin{align*}
\mathscr A_x\cap\mathscr Q=&\{v=h_vq_x+\ell_v\in\mathscr V_x|h_v=0~\text{and}~\ell_v~\text{vanishes on any tail of}~R(X,L)\}\\
=&\{v=h_vq_x+\ell_v\in\mathscr V_x|h_v=0,~\text{and both}~\pm\ell_v\in\mathscr V\}=\mathscr A.
\end{align*}
\end{proof}

\begin{rem}\label{central-auto-rmk}
Recall the central automorphism group defined in Section 2.2.1, and the torus ${\mathbf T}$ in it whose Lie algebra is isomorphic to $\mathscr A$. The action of the one-parameter group generated by $v\in\mathscr A$ on $X$ precisely coincides with the induced ${\rm k}^\times$-action on the central fibre $\mathcal X_0\cong X$ of $(\mathcal X,\mathcal L)$ associated to $v$ (cf. \cite[Section 8]{Knop93} or \cite[Theorem 21.5]{Timashev-book}).  The last point follows from \cite[Section 8]{Knop93} (see also \cite[Section 21]{Timashev-book}). Thus, test configurations associated to $v\in\mathscr A$ are product test configurations induced by a one-parameter subgroup in $\mathfrak A(X)$.
\end{rem}

Combining with Proposition \ref{central-spherical}, we get
\begin{cor}\label{non-spherical-cor}
Suppose that $X$ is not a $G\times{\rm k}^\times$-spherical variety. Then $\mathscr A_x=\mathscr A$ for any $x\in C$.
\end{cor}
\begin{proof}
Otherwise, there is some $v_0=h_0q_{x_0}+\ell_0$ with $h_0\not=0$ which defines a product test configuration $\mathcal X$ of $X$, whose central fibre $\mathcal X_0\cong X$. By Proposition \ref{central-spherical}, $\mathcal X_0$ is a $G\times{\rm k}^\times$-spherical variety, a contradiction.
\end{proof}

\subsubsection{Some combinatorial data}
We hope to compute the polytope $\Delta_\mathscr Z(\mathcal L)$ and related combinatorial data of $(\mathcal X,\mathcal L)$ as done for $(X,L)$ before. We mainly interest in the case when $m=-1$, that is, $\mathcal X$ has integral central fibre.

Keep the notations above. By direct computation we have
\begin{align}\label{A-of-D}
A_{x}(\mathfrak D,\lambda,t)=\left\{\begin{aligned}&\min\{A_{x_0}(\mathfrak d,\lambda),\frac{-t+m_0+\ell_0(\lambda)}{h_0}\},~\text{when}~x=x_0,\\
&A_{x}(\mathfrak d,\lambda),~\text{when}~x\not=x_0,\end{aligned}\right.~\text{for}~\lambda\in\Gamma_\mathbb R,
\end{align}
and
\begin{align}\label{A-sum-of-D}
A(\mathfrak D,\lambda,t)=A_{x_0}(\mathfrak D,\lambda,t)+\sum_{x\not=x_0}A_x(\mathfrak d,\lambda),~\lambda\in\Delta_\mathscr Z(\mathfrak d).
\end{align}
Here when $h_0=0$, we formally take $\frac c{h_0}=\pm\infty$ or $0$ according to the sign of $c$. Under this convention $A_{x_0}(\mathfrak D,\lambda,t)$ reduces to
$$ A_{x_0}(\mathfrak D,\lambda,t)=A_{x_0}(\mathfrak d,\lambda)\chi_{\{-t+m_0+\ell_0(\lambda)\geq0\}},~\lambda\in\Delta_\mathscr Z(\mathfrak d),$$
where by $\chi_S$ we denote the characteristic function of a set $S$, and
$$A(\mathfrak D,\lambda,t)=A(\mathfrak d,\lambda)\chi_{\{-t+m_0+\ell_0(\lambda)\geq0\}},~\lambda\in\Delta_\mathscr Z(\mathfrak d).$$
Thus, we get
\begin{align}\label{tilde-DZ(d)}
{\Delta}_\mathscr Z(\mathcal L)=&\{(\lambda,t)\in\Delta_\mathscr Z(L)\times\mathbb R|0\leq t\leq \tau^0(\lambda-\lambda_0)\}\notag\\=&(\Delta_\mathscr Z(L)\times\mathbb R)\cap\{A(\mathfrak D,\lambda-\lambda_0,\tau)\geq0\},
\end{align}
and ${\Delta}_\mathscr Z(\mathfrak D)={\Delta}_\mathscr Z(\mathcal L)-(\lambda_0,0)$, where
\begin{align}\label{tau-0}
\tau^0(\lambda):=m_0+\ell_0(\lambda)+h_0(A(\mathfrak d,\lambda)-A_{x_0}(\mathfrak d,\lambda)),~\lambda\in\Gamma_\mathbb R.
\end{align}

Also, for each $x\in C$,
\begin{align*}
{\Delta}_x(\mathfrak D)=\{(\lambda,h, t)\in\Gamma_\mathbb R\times\mathbb R_+\times\mathbb R|(\lambda,t)\in {\Delta}_\mathscr Z(\mathfrak D),~h\geq-A(\mathfrak D,\lambda,t)\},
\end{align*}
and the coloured fan $\mathfrak F_\mathcal X$ of $\mathcal X$ at each $x\in C$ consists of inner normal cones of ${\Delta}_x(\mathfrak D)$ whose relative interior intersects $\bar{\mathscr V}$.

\subsection{The classification}
In this section we consider the inverse direction of Proposition \ref{tc-to-v0}. Given any integral $v_0$ in $\mathscr V_{x_0}$, we will construct a $G$-equivariant normal test configuration of $(X,L)$ whose central fibre is integral.

Suppose that $v_0=\ell_0+h_0q_{x_0}\in\mathscr V_{x_0}$ (when $h_0=0$, we can choose $x_0$ to be any point in $C$). We can choose a sufficiently large integer $m_0$ so that the function $\tau^0(\cdot)$ defined by \eqref{tau-0} is positive on $\Delta_\mathscr Z(\mathfrak d)=\Delta_\mathscr Z(L)-\lambda_0$. As in \eqref{A-of-D}, set
\begin{align}\label{A-of-v0}
\tilde A_{x}(\lambda,t):=\left\{\begin{aligned}&\min\{A_{x_0}(\mathfrak d,\lambda),\frac{-t+m_0+\ell_0(\lambda)}{h_0}\},~\text{when}~x=x_0,\\
&A_{x}(\mathfrak d,\lambda),~\text{when}~x\not=x_0,\end{aligned}\right.~\text{for}~\lambda\in\Gamma_\mathbb R,
\end{align}
and
$$\tilde A(\lambda,t):=\tilde A_{x_0}(\lambda,t)+\sum_{x\not=x_0}A_x(\mathfrak d,\lambda),~\lambda\in\Delta_\mathscr Z(\mathfrak d).$$
Then $\tilde A(\lambda,t)\geq0$ for any $(\lambda,t)$ lies in
\begin{align*}
\tilde {\Delta}_\mathscr Z:=\{(\lambda,t)\in\Delta_\mathscr Z(\mathfrak d)\times\mathbb R|0\leq t\leq \tau^0(\lambda)\}.
\end{align*}
Also, for each $x\in C$, set
\begin{align*}
\tilde{\Delta}_x=\{(\lambda+hq_x, t)\in\Gamma_\mathbb R\times\mathbb R_+\times\mathbb R|(\lambda,t)\in\tilde{\Delta}_\mathscr Z,~h\geq-\tilde A_x(\lambda,t)\},
\end{align*}
We call $\tilde{\Delta}_x\cap\{h=-\tilde A_x(\lambda,t)\}$ the bottom of $\tilde{\Delta}_x$.

Now we construct a coloured fan $\bar{\mathfrak F}$. Since the total space $\mathcal X$ is assumed to be complete, it suffices to give its maximal coloured cones and coloured hypercones of type \uppercase\expandafter{\romannumeral2}. We say a coloured hypercone $\mathscr C$ is maximal if at least one $\mathscr C_x=\mathscr C\cap\mathscr Q_{x,+}$ is of maximal dimension. Consider the inner normal cone of every vertex of $\tilde{\Delta}_x$ for all $x\in C$. If the relative interior of such a cone $\mathscr C$ intersects $\bar{\mathscr V}$ and all its generators lies in $$\tilde{\mathscr B}:=\{(v_D,0)|D\in\mathscr B(X)\}\cup\{(v_0,-1)\},$$ then $\mathscr C$ will be a coloured cone in $\bar{\mathfrak F}$. It remains to select colours. Suppose that $\mathscr C$ is the inner normal cone of $\tilde{\Delta}_x$ at a vertex $p_0$. Then we choose
$$\mathscr R=\{\bar D|D\in\mathscr B(X)\setminus\mathscr B(X)^G,~m_D=-p_0(v_D)\},$$
and put the coloured cone $(\mathscr C,\mathscr R)$ in $\bar{\mathfrak F}$.

It remains to construct hypercones of type \uppercase\expandafter{\romannumeral2}. Consider a remaining normal cone as above which has generators out of $\tilde{\mathscr B}$. If $\mathscr C$ is the inner normal cone of $\tilde{\Delta}_x$ at a vertex $p_*$, then $p_*$ projects to some $p_*'$, where $p_*'=(\lambda_*,t_*)$ is a boundary point of $\tilde\Delta_\mathscr Z$ such $\tilde A(p_*')=0$. For any $x\in C$, set
$$\mathscr S(x,\lambda_*):=\{D\in\mathscr B(X)|v_D\in\mathscr Q_{x,+},~\tilde A_x(\lambda_*,t)=\frac{m_D+\ell_D(\lambda_*)}{h_D}\}.$$
Define a coloured hypercone $(\mathscr C,\mathscr R)$ by setting
\begin{align*}
\mathscr C_x=&\text{the inner normal cone of $\tilde{\Delta}_x$ at the point on its bottom that projects to}~p_*'\in\tilde {\Delta}_\mathscr Z,\\
\mathscr W_x=&\{\bar D|D\in\mathscr B(X)^G,h_D=0,m_D=-\lambda_*(v_D)\}\cup\left\{\begin{aligned}&\mathscr B(X)^G\cap\mathscr S(x,\lambda_*),~\text{if}~x\not=x_0,\\&(\mathscr B(X)^G\cap\mathscr S(x_0,\lambda_*))\cup\{\mathcal X_0\},~\text{if}~x=x_0,\end{aligned}\right.\\
\mathscr R_x=&\{\bar D|D\in\mathscr B(X)\setminus\mathscr B(X)^G,h_D=0,m_D=-\lambda_*(v_D)\}\cup\left((\mathscr B(X)\setminus\mathscr B(X)^G)\cap\mathscr S(x,\lambda_*)\right).
\end{align*}
Then $(\mathscr C,\mathscr R)=\{(\mathscr C_x,\mathscr R_x)|x\in C\}$ is the coloured hypercone defined by the data $(\mathscr W,\mathscr R):=(\cup_{x\in C}\mathscr W_x,\cup_{x\in C}\mathscr R_x)$. Such a cone will also be taken into account in $\bar{\mathfrak F}$ if its relative interior intersects $\bar{\mathscr V}$.

Together with the set 
$$\{(\mathscr C\times\mathbb Q_+,\mathscr R\times\{0\})|(\mathscr C,\mathscr R)\in\mathfrak F_X\},$$
and all their faces, we get a coloured fan $\bar{\mathfrak F}$ that covers $\bar{\mathscr V}$. Let $\mathcal X$ be the $G\times{\rm k}^\times$-variety defined by $\bar{\mathfrak F}$. From our contraction, we can directly check that the divisor $\mathfrak D$ given by \eqref{div-mathcal-L} is ample on $\mathcal X$ by Theorem \ref{ampleness-criterion}. On the other hand, by removing the divisor $\mathcal X_0$ that corresponds to $(v_0,-1)$, there is a ${\rm k}^\times$-equivariant morphism ${\rm pr}:\mathcal X\setminus\{\mathcal X_0\}(\cong X\times{\rm k})\to{\rm k}(\cong\mathbb P^1\setminus\{0\})$. We then show ${\rm pr}$ extends to a ${\rm k}^\times$-equivariant projection ${\rm pr}:\mathcal X\to\mathbb P^1$. It suffices to look at ${\rm pr}$ near $\mathcal X_0$. Recall that the hyperspace $\tilde{\mathscr E}$ of the $G\times{\rm k}^\times$-variety $\mathcal X$ is isomorphic to $\mathscr E\times\mathbb Q$, where $\mathscr E$ is the hyperspace of $X$ and $\mathbb Q$ stands for the ${\rm k}^\times$-factor. Under the projection ${\rm pr}:\mathscr E\times\mathbb Q\to\mathbb Q$ to the second factor, the coloured cones and hypercones of type \uppercase\expandafter{\romannumeral2} in $\bar{\mathfrak F}$ are mapped into either $\mathbb Q_{\geq0}$ or $\mathbb Q_{\leq0}$. In particular, if $\mathcal Y\subset\mathcal X_0$ is a $G\times{\rm k}^\times$-orbit, then its coloured cone/hypercone of type \uppercase\expandafter{\romannumeral2} in $\bar{\mathfrak F}$ maps surjectively to $\mathbb Q_{\leq0}$. Also, the colours in $\mathcal X$ are precisely $\{\overline{D\times{\rm k}^\times}|D\in\mathscr D^B\}$, which all map dominantly to ${\rm k}^\times$. Following the argument of \cite[Theorem 5.1]{Knop91} we get $\mathscr O_{\mathcal X,\mathcal Y}$ dominates $\mathscr O_{\mathbb P^1,0}$. Hence ${\rm pr}$ extends to a regular map on the whole $\mathcal X$ (cf. \cite[Proposition 12.12]{Timashev-book}), and $\mathcal X$ is indeed a test configuration associated to $(v_0,-1)$ in the sense of Proposition \ref{tc-to-v0}.

Combining with Proposition \ref{tc-to-v0}, we get
\begin{theo}\label{tc-classification}
Let $(X,L)$ be a polarized projective $G$-variety of complexity 1 with ${\rm k}(X)^B\cong{\rm k}(\mathbb P^1)$. Then $G$-equivariant normal test configurations of $(X,L)$ with integral central fibre are in one-one correspondence with pairs $(v_0,m_0)$ in the set $$\mathscr T:=\{(v_0,m_0)\in\mathscr V\times\mathbb N_+|v_0~\text{integral and}~\tau^0(\lambda)>0~\text{on}~\Delta_\mathscr Z(\mathfrak d)\}.$$
In particular, any $G$-equivariant special test configuration of $(X,L)$ is defined in this way.
\end{theo}
In the following, we often say a test configuration $(\mathcal X,\mathcal L)$ is associated to $v_0$ if it corresponds to some $(v_0,m_0)\in\mathscr T$.

\subsection{Twist of a test configuration}
Given a test configuration $(\mathcal X,\mathcal L)$ as above. It is proved by \cite[Section 2.1]{Hisamoto} that ${\rm J}^{\rm NA}(\mathcal X_{\ell'},\mathcal L_{\ell'})$ is a rational, convex piecewise linear and proper function of $\ell'\in{\rm Lie}({\mathbf T})$. In particular, it is continuous. With the help of the continuity, for our latter use it suffices to study the twist $(\mathcal X_{\ell'},\mathcal L_{\ell'})$ of $(\mathcal X,\mathcal L)$ when $\ell'$ is rational.

Let $(\mathcal X,\mathcal L)$ be a $G$-equivariant normal test configuration associated to $v_0$ as above, and $\ell'$ a rational element in the linear part $\mathscr A(\cong{\rm Lie}({\mathbf T}))$ of $\mathscr V$ so that $q\ell'$ in primitive for some $q\in\mathbb N_+$. Recall that the uncompactified total space $(\mathcal X_{\ell'}\setminus\mathcal X_{\ell'\,\infty},\mathcal L_{\ell'}|_{\mathcal X\setminus\mathcal X_{\ell'\,\infty}})$ is isomorphic to $(\mathcal X\setminus\mathcal X_{\infty},\mathcal L|_{\mathcal X\setminus\mathcal X_{\infty}})$, but the grading on $R(X,L)$ is shifted by $\ell'$. More precisely, if the $s$ in the $G$-span of ${\rm H}^0(X,L^k)^{(B)}_\lambda\subset R(X,L)$ and has grading $\tau(s)$ induced by $(\mathcal X,\mathcal L)$, then the grading of $s$ induced by $(\mathcal X_{\ell'},\mathcal L_{\ell'})$ is
\begin{align}\label{grading-twist}
\tau'(s)=\tau(s)+\ell'(\lambda).
\end{align}

We first deal with integral $\ell'$. In this case $q=1$, $(\mathcal X_{\ell'},\mathcal L_{\ell'})$ is a test configuration and we shall determine the compactified total space of $(\mathcal X_{\ell'},\mathcal L_{\ell'})$. By the above discussion, the coloured cones and hypercones of type \uppercase\expandafter{\romannumeral2} in $\mathfrak F_{\mathcal X}$ that lies in $\mathscr E\times\mathbb Q_{\leq0}$, which precisely give the coloured fan of $(\mathcal X_{\ell'}\setminus\mathcal X_{\ell'\,\infty},\mathcal L_{\ell'}|_{\mathcal X\setminus\mathcal X_{\ell'\,\infty}})\cong(\mathcal X\setminus\mathcal X_{\infty},\mathcal L|_{\mathcal X\setminus\mathcal X_{\infty}})$, keep the same in $\mathfrak F_{\mathcal X_{\ell'}}$. On the other hand, the ${\rm k}^\times$-action of $(\mathcal X_{\ell'},\mathcal L_{\ell'})$ corresponds to $(\ell',1)\in{\rm Lie}(\mathbf T\times{\rm k}^\times)\cong\mathscr A\times\mathbb Q$. Thus the fibre $\mathcal X_{\ell'\,\infty}$ at $\infty$ in $\mathbb P^1$ corresponds to the ray $\mathbb Q_{\geq0}(\ell',q)\in\bar{\mathscr E}$, and the other coloured cones and hypercones of type \uppercase\expandafter{\romannumeral2} in $\mathfrak F_{\mathcal X_\ell}$ are
\begin{align}\label{fan-trivial-part}
\{{\rm Cone}((\mathscr C,\mathscr R),(\ell',1))|~(\mathscr C,\mathscr R)\in\mathfrak F_{\mathcal X}~\text{is a coloured cone or hypercone of type \uppercase\expandafter{\romannumeral2} in}~\mathfrak F_{\mathcal X}\}.
\end{align}
Taking an ${\rm SL}_{{\rm rk}(\Gamma)+1}$-transformation
$$(v,u)\to(v-u\ell',u),~\forall (v,u)\in\mathscr E\times \mathbb Q,$$
which in particular transforms $(\ell',1)$ to $(0,1)$, we see that $(\mathcal X_{\ell'},\mathcal L_{\ell'})$ is isomorphic to the test configuration associated to $v_0+\ell'$.

Then we turn to the general rational case. Denote by $(\mathcal X^{(q)},\mathcal L^{(q)})$ the base change (and then a normalization) $t\to t^q$, $t\in{\rm k}^\times$ of $(\mathcal X,\mathcal L)$. In this case, the base change $(\mathcal X_{\ell'}^{(q)},\mathcal L_{\ell'}^{(q)})$ of the twist $(\mathcal X_{\ell'},\mathcal L_{\ell'})$ is a test configuration of $(X,L)$, which is defined as the twist of $(\mathcal X^{(q)},\mathcal L^{(q)})$ by the integral element $q\ell'$. Let $\mathfrak F_{\mathcal X}$ be the coloured fan of $\mathcal L$. The coloured fan of $\mathcal X^{(q)}$ consists of the rescalling coloured cones and hypercones of type \uppercase\expandafter{\romannumeral2} in $\mathfrak F_{\mathcal X}$. The rescalling maps each ray $\mathbb Q_{\geq0}(v,\pm1)$ to $\mathbb Q_{\geq0}(qv,\pm1)$. In particular, $\mathbb Q_{\geq0}(v_0,-1)$ is mapped to $\mathbb Q_{\geq0}(qv_0,-1)$ with primitive generator $(qv_0,-1)$, and $\mathbb Q_{\geq0}(\ell',1)$ is mapped to $\mathbb Q_{\geq0}(q\ell',1)$ with primitive generator $(q\ell',1)$. From the previous case, we see that $(\mathcal X_{\ell'}^{(q)},\mathcal L_{\ell'}^{(q)})$ is associated to $q(v_0+\ell')$. 

There is an alternative approach to see the relation. Taking a base change $t\to t^p$, we get the corresponding grading of $s$ induced by $(\mathcal X^{(q)}_{\ell'},\mathcal L^{(q)}_{\ell'})$ is $q\tau(s)+q\ell'(\lambda)$. We then concludes that $(\mathcal X^{(q)}_{\ell'},\mathcal L^{(q)}_{\ell'})$ coincides with the twist of $(\mathcal X^{(q)},\mathcal L^{(q)})$ by the integral element $q\ell'$. By the previous case we see that
\begin{lem}\label{twist-lem}
The base change $(\mathcal X^{(q)}_{\ell'},\mathcal L^{(q)}_{\ell'})$ of $(\mathcal X_{\ell'},\mathcal L_{\ell'})$ is isomorphic to the test configuration associated to $q(v_0+\ell')$.
\end{lem}


\section{The Futaki invariant and K-stability}

In this section we compute the Futaki invariant of $G$-equivariant normal test configurations with $m=-1$, and derive a K-stability criterion of $X$.
\subsection{Preparations}
Let $(X,L)$ be a polarized $G$-variety of complexity 1  with ${\rm k}(X)^B\cong{\rm k}(\mathbb P^1)$, and $(\mathcal X,\mathcal L)$ a test configuration of $(X,L)$ associated to some integral $v_0=h_0q_{x_0}+\ell_0\in\mathscr V_{x_0}(\subset\mathscr V)$. Recall that we have already calculated $\dim{\rm H}^0(X,L^k)$ in Lemma \ref{h0(X,Lk)}. In this section we calculate the total weight $w_k(\mathcal X,\mathcal L)$ of $(\mathcal X,\mathcal L)$ for general cases. Both expressions will be simplified in the next section when $X$ is $\mathbb Q$-Fano and $L$ equals to a multiple of $K_X^{-1}$.

We shall first introduce some notations here. By Proposition \ref{F-tau-Rk}, when $m=-1$, we directly conclude that $\mathscr F_{(\mathcal X,\mathcal L)}^\tau R_k=0$ whenever
\begin{align*}
\tau\geq[km_0+\ell_0(\lambda-k\lambda_0)+h_0k(A(\mathfrak d,\frac\lambda k-\lambda_0)-A_{x_0}(\mathfrak d,\frac\lambda k-\lambda_0))]+1=[k\tau^0(\frac\lambda k-\lambda_0)]+1,
\end{align*}
where the function $\tau^0(\cdot)$ is defined by \eqref{tau-0}. Also, recall the functions $A_x(\mathfrak D,\cdot,\cdot)$ defined by \eqref{A-of-D}-\eqref{A-sum-of-D} and the polytope $\Delta_\mathscr Z(\mathcal L)$ defined by \eqref{tilde-DZ(d)}. For our later use, we need to figure out domains of linearity of all $\{A_x(\mathfrak D,\cdot,\cdot)\}_{x\in C}$. When $h_0=0$, they are precisely those $\{\Omega_a\}$ defined above (see \eqref{affine-domain}). It remains to consider the cases when $h_0\not=0$. Note that when $h_0\not=0$ and $m_0\gg1$,
\begin{align}\label{tilde-A}
A(\mathfrak D,\lambda,\tau)=\left\{\begin{aligned}&A(\mathfrak d,\lambda),~&&\text{when}~0\leq\tau\leq\tilde\tau^0(\lambda),\\
&\sum_{x\not=x_0}A_x(\mathfrak d,\lambda)+\left(\frac{-\tau+m_0+\ell_0(\lambda)}{h_0}\right),~&&\text{when}~\tilde\tau^0(\lambda)\leq\tau\leq\tau^0(\lambda),\end{aligned}  \right.
\end{align}
where
$$\tilde\tau^0(\lambda):=m_0+\ell_0(\lambda)-h_0A_{x_0}(\mathfrak d,\lambda),~\lambda\in\Delta_{\mathscr Z}(\mathfrak d).$$
We divide ${\Delta}_\mathscr Z(\mathcal L)$ into two parts
\begin{align}\label{polytope-D-Zo(mathcal-L)}
{\Delta}_\mathscr Z^o(\mathcal L):={\Delta}_\mathscr Z(\mathcal L)\cap\{\tilde\tau(\lambda-\lambda_0)\leq\tau\leq\tau(\lambda-\lambda_0)\},
\end{align}
and ${\Delta}_\mathscr Z'(\mathcal L):={\Delta}_\mathscr Z(\mathcal L)\setminus{\Delta}_\mathscr Z^o(\mathcal L)$. Then by concavity of $A_{x_0}(\mathfrak d,\cdot)$, ${\Delta}_\mathscr Z^o(\mathcal L)$ is a convex polytope. The common domains of linearity of all $\{A_x(\mathfrak D,\cdot,\cdot)\}_{x\in C}$ precisely consist of
\begin{align}\label{linear-domains}
\tilde \Omega_a^o:=(\Omega_a\times\mathbb R)\cap{\Delta}_\mathscr Z^o(\mathcal L)~\text{and}~\tilde \Omega_a':=(\Omega_a\times\mathbb R)\cap{\Delta}_\mathscr Z'(\mathcal L)~\text{for}~a=1,...,N.
\end{align}
Sometimes we write those domains in total as $\{\tilde\Omega_{\tilde a}\}_{\tilde a}^{\tilde N}$ and denote
$$\tilde D_{\tilde a}(x_0)=(h_0,\ell_0),~\tilde D_{\tilde a}(x_0)=m_0~\text{if}~\tilde\Omega_{\tilde a}=\tilde \Omega_a^o~\text{for some}~a,$$
when there is no confusions.

With the conventions introduced above, we have
\begin{lem}\label{wk}
Suppose that $(\mathcal X,\mathcal L)$ is a test configuration associated to some integral $v_0=\ell_0+h_0q_{x_0}\in\mathscr Q_{x_0,+}$ for some $x_0\in C$ (if $h_0=0$ then $v_0\in\mathscr Q\subset\mathscr Q_{x,+}$ for all $x\in C$), $m=-1$. Then the total weight $w_k(\mathcal X,\mathcal L)$ of $(\mathcal X,\mathcal L)$ is given by
\begin{itemize}
\item[(1)] When $h_0=0$,
\begin{align}\label{wk-h0=0}
w_k(\mathcal X,\mathcal L)=&k^{n+1}\int_{\Delta_\mathscr Z(L)}\tau^0(\lambda-\lambda_0)A(\mathfrak d,\lambda-\lambda_0)\pi(\lambda)d\lambda\notag\\
&+k^{n}\int_{\Delta_\mathscr Z(L)}\tau^0(\lambda-\lambda_0)A(\mathfrak d, \lambda -\lambda_0)\langle\nabla\pi(\lambda),\rho\rangle d\lambda\notag\\
&+\frac12k^n\int_{\partial\Delta_\mathscr Z(L)}\tau^0(\lambda-\lambda_0)A(\mathfrak d,\lambda-\lambda_0)\pi(\lambda)d\sigma\notag\\
&+\frac12k^n\sum_{x\in C}\sum_{a=1}^N\int_{\Omega_a}\tau^0(\lambda-\lambda_0)(\frac1{|h_{D_a(x)}|}-1)\pi(\lambda)d\lambda\notag\\
&+k^n\int_{\Delta_\mathscr Z(L)}\tau^0(\lambda-\lambda_0)\pi(\lambda)d\lambda+O(k^{n-1}),~k\to+\infty;
\end{align}
\item[(2)] When $h_0\not=0$,
\begin{align}\label{wk-h0not=0}
w_k(\mathcal X,\mathcal L)=&k^{n+1}\int_{{\Delta}_\mathscr Z(\mathcal L)}A(\mathfrak D,\lambda-\lambda_0,t)\pi(\lambda)dt\wedge d\lambda+k^{n}\int_{{\Delta}_\mathscr Z(\mathcal L)}A(\mathfrak D,\lambda-\lambda_0,t)\langle\nabla\pi(\lambda),\rho\rangle dt\wedge d\lambda\notag\\
&+\frac12k^n\left(\int_{\partial{\Delta}_\mathscr Z(\mathcal L)}A(\mathfrak D,\lambda-\lambda_0,t)\pi(\lambda)d\tilde\sigma
+\sum_{x\in C}\sum_{a=1}^N\int_{\tilde \Omega_a'}(\frac1{h_{D_{a}}(x)}-1)\pi(\lambda)dt\wedge d\lambda\notag\right.\\
&\left.+\sum_{x\not=x_0}\sum_{a=1}^N\int_{\tilde \Omega_a^o}(\frac1{h_{D_{a}}(x)}-1)\pi(\lambda)dt\wedge d\lambda\notag+\int_{{\Delta}_\mathscr Z^o(\mathcal L)}(\frac1{h_0}-1)\pi(\lambda)dt\wedge d\lambda\right)\\
&+k^n\int_{{\Delta}_\mathscr Z(\mathcal L)}\pi(\lambda)dt\wedge d\lambda-k^n\int_{\Delta_\mathscr Z(L)}A(\mathfrak d,\lambda-\lambda_0)\pi(\lambda)d\lambda+O(k^{n-1}),~k\to+\infty.
\end{align}
\end{itemize}
\end{lem}

\begin{proof}
As in the proof of Lemma \ref{h0(X,Lk)}, we consider $(\mathcal X,\mathcal L)$ with sufficiently divisible index $r_0$ in \eqref{div-mathcal-L} so that $\mathcal L|_{\mathcal X_t}\cong L^{r_0}$ when $t\not=0$ so that $L^{r_0}$ satisfies the assumption of Lemma \ref{h0(X,Lk)}.

By definition, the total weight
\begin{align*}
w_k(\mathcal X,\mathcal L)=&\sum_{\tau=0}^{+\infty}\tau\dim(\mathscr F_{(\mathcal X,\mathcal L)}^\tau R_k/\mathscr F_{(\mathcal X,\mathcal L)}^{>\tau}R_k)\\
=&\sum_{\tau=0}^{+\infty}\tau(\dim(\mathscr F_{(\mathcal X,\mathcal L)}^\tau R_k)-\dim(\mathscr F_{(\mathcal X,\mathcal L)}^{\tau+1}R_k))\\
=&\sum_{\lambda\in k\Delta_\mathscr Z(L)\cap(\Gamma+k\lambda_0)}\sum_{\tau=1}^{[\tau_k^0(\lambda)]}\dim(\mathscr F_{(\mathcal X,\mathcal L)}^\tau R_k)^{(B)}_{\lambda}\dim V_{\lambda}\notag\\
=&\sum_{(\lambda,t)\in k\tilde\Delta_\mathscr Z(\mathcal L)\cap(\Gamma+k\lambda_0)\times\mathbb Z,t>0}\dim(\mathscr F_{(\mathcal X,\mathcal L)}^\tau R_k)^{(B)}_{\lambda}\dim V_{\lambda}.
\end{align*}

We compute $w_k$ for the cases $h_0=0$ and $h_0\not=0$ separately.

Case-1. $h_0=0$. In this case, by Proposition \ref{F-tau-Rk} (1),
\begin{align*}
w_k(\mathcal X,\mathcal L)
=&\sum_{\lambda\in k\Delta_\mathscr Z(L)\cap(\Gamma+k\lambda_0)}\sum_{\tau=1}^{[\tau_k^0(\lambda)]}(\sum_{x\in C}[kA_x(\mathfrak d, \frac\lambda k-\lambda_0)]+1)\dim V_{\lambda}+O(k^{n-1}),~k\to+\infty,
\end{align*}
where we also apply Lemmas \ref{bad-point-lem} and \ref{riemann-sum} to the divisor $\mathfrak D$ given by \eqref{div-mathcal-L} and as in the proof of Lemma \ref{h0(X,Lk)}.

For each $x\in C$ with $A_x(\mathfrak d,\cdot)\not\equiv0$, using Lemma \ref{riemann-sum} in the Appendix we have
\begin{align*}
&\sum_{\lambda\in k\Delta_\mathscr Z(L)\cap(\Gamma+k\lambda_0)}\sum_{\tau=1}^{[\tau_k^0(\lambda)]}[kA_x(\mathfrak d, \frac\lambda k-\lambda_0)]\dim(V_{\lambda})\\
=&k^{n+1}\int_{{\Delta}_\mathscr Z(\mathcal L)}A_x(\mathfrak d, \lambda -\lambda_0)\pi(\lambda)dt\wedge d\lambda+k^{n}\int_{{\Delta}_\mathscr Z(\mathcal L)}A_x(\mathfrak d, \lambda -\lambda_0)\langle\nabla\pi(\lambda),\rho\rangle dt\wedge d\lambda\\
&+\frac12k^n\int_{\partial{\Delta}_\mathscr Z(\mathcal L)}A_x(\mathfrak d, \lambda -\lambda_0)\pi(\lambda)d\tilde\sigma
+\frac12k^n\sum_{a=1}^N\int_{\tilde\Omega_a}(\frac1{h_{D_a(x)}}-1)\pi(\lambda)dt\wedge d\lambda\\
&-k^n\int_{\Delta_\mathscr Z(L)}A_x(\mathfrak d, \lambda -\lambda_0)\pi(\lambda) d\lambda+O(k^{n-1}),~k\to+\infty.
\end{align*}
Note that $\tau^0$ is integral. We have
\begin{align*}
&\int_{\partial{\Delta}_\mathscr Z(\mathcal L)}A_x(\mathfrak d, \lambda -\lambda_0)\pi(\lambda)d\tilde\sigma\\
=&\int_{\partial\Delta_\mathscr Z(L)}\tau^0(\lambda)A_x(\mathfrak d, \lambda -\lambda_0)\pi(\lambda)d\sigma\\
+&\int_{\Delta_\mathscr Z(L)}A_x(\mathfrak d, \lambda -\lambda_0)\pi(\lambda)d\lambda+\int_{\text{graph}(\tau^0)}A_x(\mathfrak d, \lambda -\lambda_0)\pi(\lambda)d\tilde\sigma\\
=&\int_{\partial\Delta_\mathscr Z(L)}\tau^0(\lambda)A_x(\mathfrak d, \lambda -\lambda_0)\pi(\lambda)d\sigma
+2\int_{\Delta_\mathscr Z(L)}A_x(\mathfrak d, \lambda -\lambda_0)\pi(\lambda)d\lambda.
\end{align*}
Summing over $x\in C$ we get \eqref{wk-h0=0}.

Case-2. $h_0\not=0$. Recall \eqref{tilde-A}. As in Case-1, by Proposition \ref{F-tau-Rk} (2), Lemmas \ref{bad-point-lem} and \ref{riemann-sum} we get
\begin{align}\label{wk-h0-neq-0-def}
w_k(\mathcal X,\mathcal L)=&\sum_{\lambda\in k\Delta_\mathscr Z(L)\cap(\Gamma+k\lambda_0)}\sum_{t=1}^{[\tau^0(\lambda)]}(\deg(\delta_k(\lambda,t))+1)\dim V_\lambda+O(k^{n-1})\notag\\
=&\sum_{(\lambda,t)\in k{\Delta}_\mathscr Z(\mathcal L)\cap((\Gamma+k\lambda_0)\times\mathbb Z)}\left(\sum_{x\in C}[kA_x(\mathfrak D,\frac\lambda k-\lambda_0,\frac tk)]+1\right)\dim V_\lambda\notag\\
&-\sum_{\lambda\in k\Delta_\mathscr Z(L)\cap(\Gamma+k\lambda_0)}\left(\sum_{x\in C}[kA_x(\mathfrak D,\frac\lambda k-\lambda_0,0)]+1\right)\dim V_\lambda+O(k^{n-1}),~k\to+\infty.
\end{align}
Also, using Lemma \ref{riemann-sum} we have,
\begin{align*}
&\sum_{(\lambda,t)\in k\Delta_\mathscr Z(\mathcal L)\cap((\Gamma+k\lambda_0)\times\mathbb Z)}\sum_{x\in C}[kA_x(\mathfrak D,\frac\lambda k-\lambda_0,\frac tk)]\dim V_\lambda\\
=&k^{n+1}\int_{\Delta_\mathscr Z(\mathcal L)}A(\mathfrak D,\lambda-\lambda_0,t)\pi(\lambda)dt\wedge d\lambda+k^{n}\int_{\Delta_\mathscr Z(\mathcal L)}A(\mathfrak D,\lambda-\lambda_0,t)\langle\nabla\pi(\lambda),\rho\rangle dt\wedge d\lambda\\
&+\frac12k^n\int_{\partial\Delta_\mathscr Z(\mathcal L)}A(\mathfrak D,\lambda-\lambda_0,t)\pi(\lambda)d\tilde\sigma+\frac12k^n\sum_{x\in C}\sum_{\tilde a=1}^{\tilde N}\int_{\tilde \Omega_{\tilde a}}(\frac1{h_{D_{\tilde a}}(x)}-1)\pi(\lambda)dt\wedge d\lambda\\
&+O(k^{n-1}),~k\to+\infty,
\end{align*}
where $d\tilde\sigma$ is the induced lattice measure on $\partial\Delta_\mathscr Z(\mathcal L)$ and $\{\tilde \Omega_{\tilde a}\}_{\tilde a=1}^{\tilde N}$ are domains defined in \eqref{linear-domains}. Thus
\begin{align*}
&\sum_{(\lambda,t)\in k\Delta_\mathscr Z(\mathcal L)\cap((\Gamma+k\lambda_0)\times\mathbb Z)}\sum_{x\in C}[kA_x(\mathfrak D,\frac\lambda k-\lambda_0,\frac tk)]\dim V_\lambda\\
=&k^{n+1}\int_{\Delta_\mathscr Z(\mathcal L)}A(\mathfrak D,\lambda-\lambda_0,t)\pi(\lambda)dt\wedge d\lambda+k^{n}\int_{\Delta_\mathscr Z(\mathcal L)}A(\mathfrak D,\lambda-\lambda_0,t)\langle\nabla\pi(\lambda),\rho\rangle dt\wedge d\lambda\\
&+\frac12k^n\int_{\partial\Delta_\mathscr Z(\mathcal L)}A(\mathfrak D,\lambda-\lambda_0,t)\pi(\lambda)d\tilde\sigma
+\frac12k^n\sum_{x\in C}\sum_{a=1}^N\int_{\tilde \Omega_a'}(\frac1{h_{D_{a}}(x)}-1)\pi(\lambda)dt\wedge d\lambda\\
&+\frac12k^n\sum_{x\not=x_0}\sum_{a=1}^N\int_{\tilde \Omega_a^o}(\frac1{h_{D_{a}}(x)}-1)\pi(\lambda)dt\wedge d\lambda\\
&+\frac12k^n\int_{\Delta_\mathscr Z^o(\mathcal L)}(\frac1{h_0}-1)\pi(\lambda)dt\wedge d\lambda+O(k^{n-1}),~k\to+\infty,
\end{align*}

Similarly, it holds
\begin{align*}
\sum_{\lambda\in k\Delta_\mathscr Z(L)\cap(\Gamma+k\lambda_0)}\sum_{x\in C}[kA_x(\mathfrak D,\frac\lambda k-\lambda_0,0)]\dim V_\lambda
=&k^n\int_{\Delta_\mathscr Z(L)}A(\mathfrak D,\lambda-\lambda_0,0)\pi(\lambda)d\lambda+O(k^{n-1})\\
=&k^n\int_{\Delta_\mathscr Z(L)}A(\mathfrak d,\lambda-\lambda_0)\pi(\lambda)d\lambda+O(k^{n-1}),
\end{align*}
as $~k\to+\infty$. Also,
\begin{align*}
&\sum_{(\lambda,t)\in k\Delta_\mathscr Z(\mathcal L)\cap(\Gamma+k\lambda_0)\times\mathbb Z}\dim V_\lambda
=k^n\int_{\Delta_\mathscr Z(\mathcal L)}\pi(\lambda)dt\wedge d\lambda+O(k^{n-1}),~k\to+\infty,
\end{align*}
and
\begin{align*}
&\sum_{\lambda\in k\Delta_\mathscr Z(L)\cap(\Gamma+k\lambda_0)}\dim V_\lambda=O(k^{n-1}),~k\to+\infty.
\end{align*}
Plugging the above relations into \eqref{wk-h0-neq-0-def} we get \eqref{wk-h0not=0}.
\end{proof}

\subsection{The Futaki invariant of $\mathbb Q$-Fano $G$-varieties}
In this section we will express the Futaki invariant in terms of purely combinatorial data. For simplicity we consider the case when $X$ is Gorenstein and choose the divisor $\mathfrak d$ in \eqref{anti-can-div-thm} with weight $\lambda_0=\kappa_P$ for $L=K_X^{-1}$. In general cases, we may choose $L=K_X^{-m}$ and $m\mathfrak d$ with weight $\lambda_0=m\kappa_P$ for a sufficiently divisible $m\in\mathbb N_+$.

In the following we mainly focus on the case $h_0\not=0$ since the case $h_0=0$ is much simpler. In this case we first do some reductions on $\partial\Delta_\mathscr Z(\mathcal L)$. The boundary $\partial\Delta_\mathscr Z(\mathcal L)$ contains three parts (see Fig-1):
\begin{itemize}
\item The base $\Delta_\mathscr Z(K_X^{-1})\times\{0\}$, on which $A(\mathfrak D,\lambda-\kappa_P,\tau)=A(\mathfrak d,\lambda-\kappa_P)$ and $d\tilde\sigma=d\lambda$;
\item The walls $\tilde F:=(F\times\mathbb R)\cap\Delta_\mathscr Z(\mathcal L)$, where $F$ is a facet of $\Delta_\mathscr Z(K_X^{-1})$. Note that if $F\subset\{A(\mathfrak d,\lambda-\kappa_P)=0\}$, then $\tilde\tau^0=\tau^0$ on $\tilde F$, and $\tilde F$ intersects $\Delta_\mathscr Z^o(\mathcal L)$ on a face of codimension at least 2;
\item The graph of $t=\tau^0(\lambda-\kappa_P)$ over $\Delta_\mathscr Z(K_X^{-1})$, on which $A(\mathfrak D,\lambda-\kappa_P,\tau)=0$.
\end{itemize}

\begin{figure}[h]
\begin{center}
\begin{tikzpicture}[scale=1.2]
\fill[color=gray!20] (0,0)--(3,0)--(3,3)--(2,3)--(1,2.5)--(0,1.5)--(0,0);
\fill[color=gray!50] (0,1.5)--(2,1)--(3,2)--(3,3)--(2,3)--(1,2.5)--(0,1.5);
\draw [very thick] (0,0) -- (3,0);
\draw [semithick] (0,0)--(3,0)--(3,3)--(2,3)--(1,2.5)--(0,1.5)--(0,0);
\draw [semithick] (0,1.5)--(2,1)--(3,2);

\draw (-0.2,-0.2) node {\scriptsize{$A(\mathfrak d,\cdot)=0$}};
\draw (3.2,-0.2) node {\scriptsize{$A(\mathfrak d,\cdot)\not=0$}};
\draw (1.5,-0.2) node {\scriptsize{$\Delta_\mathscr Z(K_X^{-1})$}};
\draw (1.5,2) node {\scriptsize{$\Delta_\mathscr Z^o(\mathcal L)$}};
\draw (1.5,0.5) node {\scriptsize{$\Delta_\mathscr Z'(\mathcal L)$}};
\draw (1.5,3.2) node {\scriptsize{$A(\mathfrak D,\cdot,\cdot)=0~(t=\tau^0)$}};
\draw (1.5,1.35) node {\scriptsize{$t=\tilde\tau^0$}};
\end{tikzpicture}
\end{center}
Fig-1: The polytope $\tilde\Delta_\mathscr Z(\mathfrak d)$.
\end{figure}

We need a few more discussions on the $\tilde F$'s with $F\subset\{A(\mathfrak d,\lambda-\kappa_P)\not=0\}$. Such an $F$ is defined by
\begin{align}\label{eq-F}
v_D(\lambda)+m_D-v_D(\kappa_P)=0.
\end{align}
We have the following cases:
\begin{itemize}
\item $D$ is a central $G$-stable divisor, or a colour of type-a, or a central colour in the quasihomogeneous case that descends to a $B\cap L'$-stable divisor in Section 3.2. In all these three cases $m_D=1$ and $v_D$ is primitive. The unit outer normal vector of $F$ is $\nu=-v_D/|v_D|$, and the induced lattice measure on $F$ is
    \begin{align*}
    d\sigma=\frac{\langle\lambda,\nu\rangle}{1-v_D(\kappa_P)}d\sigma_0=\langle\lambda-\kappa_P,\nu\rangle d\sigma_0,
    \end{align*}
    where $d\sigma_0$ is the standard induced Lebesgue measure. Consequently, the induced lattice measure on $\tilde F$ is
    \begin{align*}
    d\tilde\sigma=\langle\lambda-\kappa_P,\nu\rangle d\sigma_0\wedge dt;
    \end{align*}
\item $D$ is a colour of type-a' or b. By Remarks \ref{v-colour} and \ref{coef-relation}, \eqref{eq-F} reduces to $v_D(\lambda)=0$ in both cases. As for a colour of type-a' or b, $v_D$ is proportional to $\alpha^\vee|_{\mathscr Q}$, we get $\pi(\cdot)|_F\equiv0$ in both cases.
\end{itemize}
From the above discussions we get
\begin{lem}\label{bdry-measure}
\begin{itemize}
\item [(1)] If $F$ is a facet of $\Delta_\mathscr Z(K_X^{-1})$ on which \eqref{eq-F} holds for some central prime $B$-stable divisor $D$, then on $F$ it holds
$$\pi(\lambda)d\sigma=\pi(\lambda)\langle\lambda-\kappa_P,\nu\rangle d\sigma_0,$$
where $d\sigma_0$ is the standard induced Lebesgue measure.
\item [(2)] On the boundary of $\Delta_\mathscr Z(\mathcal L)$, it holds
\begin{align*}
A(\mathfrak D,\lambda-\kappa_P,t)\pi(\lambda)d\tilde\sigma=\left\{\begin{aligned}&A(\mathfrak d,\lambda-\kappa_P)\pi(\lambda)d\lambda,~\text{on}~\Delta_\mathscr Z(K_X^{-1})\times\{0\},\\
&A(\mathfrak D,\lambda-\kappa_P,t)\pi(\lambda)\langle\lambda-\kappa_P,\nu\rangle d\sigma_0\wedge dt,~\text{on walls},\\
&A(\mathfrak D,\lambda-\kappa_P,t)\pi(\lambda)\langle\lambda,\nu\rangle d\tilde\sigma_0,~\text{on the graph of}~\tau^0(\lambda-\kappa_P),\end{aligned}\right.
\end{align*}
where $d\sigma_0$ and $d\tilde\sigma_0$ are the standard induced Lebesgue measure on corresponding facets, respectively.
\end{itemize}
\end{lem}


Also we introduce a family of convex polytopes that describes the K-stability of $X$. The polytope $\Delta_\mathscr Z(K_X^{-1})$, although independent with the choice of divisor $\mathfrak d$, itself hardly represents full information of ${\rm H}^0(X,K_X^{-1})^{(B)}_\lambda$ for a fixed $\lambda$. In \cite[Section 4.4]{Ilten-Suss-Duke}, Ilten-S\"u\ss \,introduced a family of polytopes for Fano $T$-varieties of complexity 1 that fully encodes the information of K-stability. In the following we define its counterpart for a $\mathbb Q$-Fano $G$-variety of complexity 1.
\begin{lem}\label{polytope-anti-can}
The function $A(\mathfrak d,\cdot)$, and the polytopes
\begin{align}\label{Delta-O-K-def}
\Delta_x^O(K_X^{-1}):=&\{(\lambda,t)\in(\Gamma_\mathbb R+\kappa_P)\times\mathbb R|-A_{x}(\mathfrak d,\lambda-\kappa_P)\leq t\leq A(\mathfrak d,\lambda-\kappa_P) -A_{x}(\mathfrak d,\lambda-\kappa_P)\}\notag\\&+(0,a_{x}-1),~x\in C.
\end{align}
are independent of the choice of an anti-canonical divisor $\mathfrak d$ in \eqref{anti-can-div} in the oen-parameter case or \eqref{anti-can-div-quasi-homo} in the quasihomogeneous case.
\end{lem}

\begin{proof}
By \eqref{anti-can-div} and \eqref{anti-can-div-quasi-homo} we see that in both cases
$$(a_x-1)-A_x(\mathfrak d,\lambda)=-\min_{x_D=x}\frac{1+\ell_D(\lambda)}{h_D},~\lambda\in\Gamma_\mathbb R,$$
and
$$A(\mathfrak d,\lambda)=2+\sum_{x\in C}\min_{x_D=x}\frac{1-h_D+\ell_D(\lambda)}{h_D},~\lambda\in\Gamma_\mathbb R,$$
since $\sum_{x\in C}a_x=2$. Hence we get the Lemma.
\end{proof}
\begin{rem}
In the following we denote
$$A(K_X^{-1},\lambda-\kappa_P):=A(\mathfrak d,\lambda-\kappa_P),~\lambda\in\Delta_\mathscr Z(K_X^{-1}).$$
In fact, one can even choose any $s_0\in{\rm H}^0(X,K_X^{-1})^{(B)}_{\lambda_0}$ with any $\lambda_0$ (not necessarily equals to $\kappa_P$) and prove
$$A(K_X^{-1},\lambda-\kappa_P)=A({\rm div}(s_0),\lambda-\lambda_0),~\lambda\in\Delta_\mathscr Z(K_X^{-1}).$$
This shows that the function $A(K_X^{-1},\lambda-\kappa_P)$ is in fact totally determined by $K_X^{-1}$. The same also holds for $\Delta_x^O(K_X^{-1})$.
\end{rem}

Now we prove the main result in this section, which holds for both cases $h_0\not=0$ and $h_0=0$:
\begin{theo}\label{Fut-one-para-thm}
Let $X$ be a $\mathbb Q$-Fano $G$-variety of complexity 1 with ${\rm k}(X)^B\cong{\rm k}(\mathbb P^1)$. Let $(\mathcal X,\mathcal L)$ be a test configuration of $(X,K_X^{-1})$ that is associated to some integral $v_0=\ell_0+h_0q_{x_0}\in\mathscr Q_{x_0,+}$ for some $x_0\in C$ (if $h_0=0$ then $v_0\in\mathscr Q\subset\mathscr Q_{x,+}$ for all $x\in C$) and $m=-1$. Then
\begin{align}\label{Fut-one-para-eq}
{\rm Fut}(\mathcal X,\mathcal L)=\langle\kappa_P-\mathbf{b}(\Delta_{x_0}^O(K_X^{-1})),v_0\rangle,
\end{align}
where
$$\mathbf{b}(\Delta_{x_0}^O(K_X^{-1}))=\frac1V\int_{\Delta_{x_0}^O(K_X^{-1})}(\lambda,t)\pi(\lambda) d\lambda\wedge dt,$$
and
\begin{align*}V:=&{\int_{\Delta_\mathscr Z(K_X^{-1})}A(K_X^{-1},\lambda-\kappa_P)\pi(\lambda)d\lambda}.
\end{align*}
\end{theo}

\begin{rem}
Clearly it holds
\begin{align*}
V=&\int_{\Delta_\mathscr Z(K_X^{-1})}\left(\int_{-A_{x}(\mathfrak d,\lambda-\kappa_P)}^{A(\mathfrak d,\lambda-\kappa_P)-A_{x}(\mathfrak d,\lambda-\kappa_P)} 1dt\right)\pi(\lambda)d\lambda\notag\\
=&\int_{\Delta_{x}^O(K_X^{-1})}\pi(\lambda) d\lambda\wedge dt,~\forall x\in C.
\end{align*}
\end{rem}

\begin{proof}[Proof of Theorem \ref{Fut-one-para-thm}]
Fix any divisor $\mathfrak d$ defined by \eqref{anti-can-div} (in one-parameter case) or \eqref{anti-can-div-quasi-homo} (in quasihomogeneous case), and take $\lambda_0=\kappa_P$ in Lemmas \ref{h0(X,Lk)} and \ref{wk}. We will simplify the terms there with the help of Lemmas \ref{bdry-measure} and \ref{nabla-pi}.

We first simplify the expression of $\dim{\rm H}^0(X,L^k)$ in \eqref{h0(X,Lk)-eq}.
By Lemma \ref{bdry-measure} (1),
\begin{align*}
\int_{\partial\Delta_\mathscr Z(K_X^{-1})}A(\mathfrak d,\lambda-\kappa_P)\pi(\lambda)d\sigma=&\int_{\partial\Delta_\mathscr Z(K_X^{-1})}A(\mathfrak d,\lambda-\kappa_P)\pi(\lambda)\langle\lambda-\kappa_P,\nu\rangle d\sigma_0.
\end{align*}
Note that
$$n=\dim X=r+\deg\pi+1.$$
By integration by parts,
\begin{align}\label{h0(X,Lk)-bdry-term}
&\int_{\partial\Delta_\mathscr Z(K_X^{-1})}A(\mathfrak d,\lambda-\kappa_P)\pi(\lambda)\langle\nu,\lambda\rangle d\sigma_0\notag\\
=&\int_{\Delta_\mathscr Z(K_X^{-1})}(n-1)A(\mathfrak d,\lambda-\kappa_P)\pi(\lambda)d\lambda+\int_{\Delta_\mathscr Z(K_X^{-1})}\langle\nabla(A(\mathfrak d,\lambda-\kappa_P)),\lambda\rangle\pi(\lambda)d\lambda\notag\\
=&n\int_{\Delta_\mathscr Z(K_X^{-1})}A(\mathfrak d,\lambda-\kappa_P)\pi(\lambda)d\lambda-\sum_{x\in C}\sum_{a=1}^N\int_{\Omega_a}\frac{m_{D_a(x)}}{h_{D_a(x)}}\pi(\lambda)d\lambda\notag\\
&+\int_{\Delta_\mathscr Z(K_X^{-1})}\langle\nabla(A(\mathfrak d,\lambda-\kappa_P)),\kappa_P\rangle\pi(\lambda)d\lambda.
\end{align}
By Theorems \ref{anti-can-div-thm} (for one-parameter case) and \ref{anti-can-div-thm-quasi-homo} (for quasihomogeneous case),
\begin{align}\label{mD-one-para}
m_{D}=1-h_D+h_Da_x,~\forall D\in\mathscr B(X).
\end{align}
Plugging \eqref{mD-one-para} into the second term of \eqref{h0(X,Lk)-bdry-term},
\begin{align*}
&\int_{\Delta_\mathscr Z(K_X^{-1})}\langle\nabla(A(\mathfrak d,\lambda-\kappa_P)),\lambda\rangle\pi(\lambda)d\lambda\\
=&\int_{\Delta_\mathscr Z(K_X^{-1})}A(\mathfrak d,\lambda-\kappa_P)\pi(\lambda)d\lambda-\sum_{x\in C}\sum_{a=1}^N\int_{\Omega_a}(\frac1{h_{D_a(x)}}+1-a_x)\pi(\lambda)d\lambda\\
&+\int_{\Delta_\mathscr Z(K_X^{-1})}\langle\nabla(A(\mathfrak d,\lambda-\kappa_P)),\kappa_P\rangle\pi(\lambda)d\lambda.
\end{align*}
Note that since $\mathfrak a=\sum_{x\in C}a_x\cdot s$ is an anti-canonical divisor on $C=\mathbb P^1$, it has degree
\begin{align}\label{deg(a)}
\deg(\mathfrak a)=\sum_{x\in C}a_x=2.
\end{align}
We get
\begin{align}\label{h0(X,Lk)-bdry-1st-term}
&\int_{\Delta_\mathscr Z(K_X^{-1})}\langle\nabla(A(\mathfrak d,\lambda-\kappa_P)),\lambda\rangle\pi(\lambda)d\lambda\notag\\
=&\int_{\Delta_\mathscr Z(K_X^{-1})}A(\mathfrak d,\lambda-\kappa_P)\pi(\lambda)d\lambda-\sum_{x\in C}\sum_{a=1}^N\int_{\Omega_a}(\frac1{h_{D_a(x)}}+1)\pi(\lambda)d\lambda\notag\\
&-2\int_{\Delta_\mathscr Z(K_X^{-1})}\pi(\lambda)d\lambda+\int_{\Delta_\mathscr Z(K_X^{-1})}\langle\nabla(A(\mathfrak d,\lambda-\kappa_P)),\kappa_P\rangle\pi(\lambda)d\lambda.
\end{align}
On the other hand,
\begin{align}\label{h0(X,Lk)-bdry-ext-term}
&\int_{\partial\Delta_\mathscr Z(K_X^{-1})}A(\mathfrak d,\lambda-\kappa_P)\pi(\lambda)\langle\nu,\kappa_P\rangle d\sigma_0\notag\\
=&\int_{\Delta_\mathscr Z(K_X^{-1})}A(\mathfrak d,\lambda-\kappa_P)\langle\nabla\pi(\lambda),\kappa_P\rangle d\lambda+\int_{\Delta_\mathscr Z(K_X^{-1})}\langle\nabla(A(\mathfrak d,\lambda-\kappa_P)),\kappa_P\rangle\pi(\lambda)d\lambda.
\end{align}
Plugging \eqref{h0(X,Lk)-bdry-term}-\eqref{h0(X,Lk)-bdry-ext-term} into \eqref{h0(X,Lk)-eq}, we get
\begin{align}\label{h0(X,Lk)-eq-one-para}
\dim{\rm H}^0(X,L^k)=&k^n\int_{\Delta_\mathscr Z(K_X^{-1})}A(\mathfrak d,\lambda-\kappa_P)\pi(\lambda)d\lambda\notag\\
&+\frac12k^{n-1}n\int_{\Delta_\mathscr Z(K_X^{-1})}A(\mathfrak d,\lambda-\kappa_P)\pi(\lambda)d\lambda+O(k^{n-1})\notag\\
=&Vk^n+\frac12Vnk^{n-1}+O(k^{n-1}),~k\to+\infty.
\end{align}

Then we deal with the total weight $w_k$ in Lemma \ref{wk}. We only show the reduction of \eqref{wk-h0not=0} since the remaining case is much simpler. Recall \eqref{tilde-DZ(d)} Denote by $\tilde \nu$ the unit outer normal vector of $\partial\Delta_\mathscr Z(\mathcal L)$. Then
\begin{align}\label{tilde-nu-one-para}
\tilde\nu=\left\{\begin{aligned}&(0,-1),~&\text{on}~\Delta_\mathscr Z(K_X^{-1})\times\{0\},\\&(\nu,0),~&\text{on the walls}.\end{aligned}\right.
\end{align}
Denote by $\tilde W$ the union of all walls of $\partial\Delta_\mathscr Z(\mathcal L)$. By Lemma \ref{bdry-measure} (2) we have
\begin{align}\label{wk-one-para-eq-1}
&\int_{\partial\Delta_\mathscr Z(\mathcal L)}A(\mathfrak D,\lambda-\kappa_P,t)\pi(\lambda)d\tilde\sigma\notag\\
=&\int_{\Delta_\mathscr Z(K_X^{-1})} A(\mathfrak d,\lambda-\kappa_P)\pi(\lambda)d\lambda+\int_{\partial\Delta_\mathscr Z(\mathcal L)}A(\mathfrak D,\lambda-\kappa_P,t)\pi(\lambda)\langle \tilde\nu,(\lambda,t)\rangle d\tilde\sigma_0\notag\\
&-\int_{\tilde W}A(\mathfrak D,\lambda-\kappa_P,t)\pi(\lambda)\nu(\kappa_P)d\sigma_0\wedge dt.
\end{align}
For the second term, taking integration by parts,
\begin{align}\label{wk-one-para-eq-2}
&\int_{\partial\Delta_\mathscr Z(\mathcal L)}A(\mathfrak D,\lambda-\kappa_P,t)\pi(\lambda)\langle \tilde\nu,(\lambda,t)\rangle d\tilde\sigma_0\notag\\
=&\int_{\Delta_\mathscr Z(\mathcal L)}(nA(\mathfrak D,\lambda-\kappa_P,t)+\langle\nabla(A(\mathfrak D,\lambda-\kappa_P,t)),(\lambda,t)\rangle )\pi(\lambda) d\lambda\wedge dt.
\end{align}
Recall the division \eqref{linear-domains}. We have
\begin{align*}
\langle\nabla(A(\mathfrak D,\lambda-\kappa_P,t)),(\lambda,t)\rangle=&A(\mathfrak D,\lambda-\kappa_P,t)-\sum_{x\not=x_0}\sum_{a=1}^N\frac{m_{D_a(x)}-\ell_{D_a(x)}(\kappa_P)}{h_{D_a(x)}}\chi_{\tilde\Omega_a'\cup\tilde\Omega^o_a}\\
&-\sum_{a=1}^N\frac{m_{D_a(x_0)}-\ell_{D_a(x_0)}(\kappa_P)}{h_{D_a(x_0)}}\chi_{\tilde\Omega_a'}-\frac{m_0-\ell(\kappa_P)}{h_0}\chi_{\tilde\Delta_\mathscr Z^o(\mathfrak d)},
\end{align*}
where $\chi_S=\chi_S(\lambda,t)$ denotes the characteristics function of the set $S$.

As before, using \eqref{mD-one-para}, \eqref{deg(a)} and taking integration, we have
\begin{align}\label{wk-one-para-eq-3}
&\int_{\Delta_\mathscr Z(\mathcal L)}\langle\nabla(A(\mathfrak D,\lambda-\kappa_P,t)),(\lambda,t)\rangle \pi(\lambda)d\lambda\wedge dt\notag\\
=&\int_{\Delta_\mathscr Z(\mathcal L)}A(\mathfrak D,\lambda-\kappa_P,t)\pi(\lambda)d\lambda\wedge dt-\sum_{x\not=x_0}\sum_{a=1}^N\int_{\tilde\Omega_a'\cup\tilde\Omega^o_a}(\frac1{h_{D_a(x)}}-1)\pi(\lambda)d\lambda\wedge dt\notag\\
&-\sum_{a=1}^N\int_{\tilde\Omega_a'}(\frac1{h_{D_a(x_0)}}-1)\pi(\lambda)d\lambda\wedge dt-\int_{\Delta_\mathscr Z^o(\mathcal L)}\frac{m_0-\ell(\kappa_P)}{h_0}\pi(\lambda)d\lambda\wedge dt\notag\\
&-2\int_{\Delta_\mathscr Z(\mathcal L)}\pi(\lambda)d\lambda\wedge dt+a_{x_0}\int_{\tilde\Delta_\mathscr Z^o(\mathfrak d)}\pi(\lambda)d\lambda\wedge dt\notag\\
&+\sum_{x\not=x_0}\sum_{a=1}^N\int_{\tilde\Omega_a'\cup\tilde\Omega^o_a}\frac{\ell_{D_a(x)}(\kappa_P)}{h_{D_a(x)}}\pi(\lambda)d\lambda\wedge dt+\sum_{a=1}^N\int_{\tilde\Omega_a'}\frac{\ell_{D_a(x_0)}(\kappa_P)}{h_{D_a(x_0)}}\pi(\lambda)d\lambda\wedge dt.
\end{align}
Plugging \eqref{wk-one-para-eq-2}-\eqref{wk-one-para-eq-3} into \eqref{wk-one-para-eq-1}, we have
\begin{align}\label{wk-one-para-eq-4}
&\int_{\partial\Delta_\mathscr Z(\mathcal L)}A(\mathfrak D,\lambda-\kappa_P,t)\pi(\lambda)d\tilde\sigma\notag\\
=&\int_{\Delta_\mathscr Z(K_X^{-1})} A(\mathfrak d,\lambda-\kappa_P)\pi(\lambda)d\lambda-\int_{\tilde W}A(\mathfrak D,\lambda-\kappa_P,t)\pi(\lambda)\nu(\kappa_P)d\sigma_0\wedge dt\notag\\
&+(n+1)\int_{\Delta_\mathscr Z(\mathcal L)}A(\mathfrak D,\lambda-\kappa_P,t)\pi(\lambda)d\lambda\wedge dt-\sum_{x\not=x_0}\sum_{a=1}^N\int_{\tilde\Omega_a'\cup\tilde\Omega^o_a}(\frac1{h_{D_a(x)}}-1)\pi(\lambda)d\lambda\wedge dt\notag\\
&-\sum_{a=1}^N\int_{\tilde\Omega_a'}(\frac1{h_{D_a(x_0)}}-1)\pi(\lambda)d\lambda\wedge dt-\int_{\Delta_\mathscr Z^o(\mathcal L)}\frac{m_0-\ell(\kappa_P)}{h_0}\pi(\lambda)d\lambda\wedge dt\notag\\
&-2\int_{\Delta_\mathscr Z(\mathcal L)}\pi(\lambda)d\lambda\wedge dt+a_{x_0}\int_{\tilde\Delta_\mathscr Z^o(\mathfrak d)}\pi(\lambda)d\lambda\wedge dt\notag\\
&+\sum_{x\not=x_0}\sum_{a=1}^N\int_{\tilde\Omega_a'\cup\tilde\Omega^o_a}\frac{\ell_{D_a(x)}(\kappa_P)}{h_{D_a(x)}}\pi(\lambda)d\lambda\wedge dt+\sum_{a=1}^N\int_{\tilde\Omega_a'}\frac{\ell_{D_a(x_0)}(\kappa_P)}{h_{D_a(x_0)}}\pi(\lambda)d\lambda\wedge dt.
\end{align}

On the other hand, by Lemma \ref{nabla-pi},
\begin{align*}
\int_{\Delta_\mathscr Z(\mathcal L)}A(\mathfrak D,\lambda-\kappa_P,t)\langle\nabla\pi(\lambda),\rho\rangle d\lambda\wedge dt
=&\frac12\int_{\Delta_\mathscr Z(\mathcal L)}A(\mathfrak D,\lambda-\kappa_P,t)\langle\nabla\pi(\lambda),\kappa_P\rangle d\lambda\wedge dt\notag\\
=&\frac12\int_{\partial\Delta_\mathscr Z(\mathcal L)}A(\mathfrak D,\lambda-\kappa_P,t)\langle\tilde\nu,\kappa_P\rangle \pi(\lambda)d\tilde\sigma_0\notag\\
&-\frac12\int_{\Delta_\mathscr Z(\mathcal L)}\langle\nabla(A(\mathfrak D,\lambda-\kappa_P,t)),\kappa_P\rangle\pi(\lambda) d\lambda\wedge dt.
\end{align*}
Here in the brackets we write $\kappa_P$ in short of $(\kappa_P,0)$. By \eqref{tilde-nu-one-para} and the fact that $\tilde A\equiv0$ on the graph of $\tau^0$, we further get
\begin{align}\label{wk-one-para-eq-5}
\int_{\Delta_\mathscr Z(\mathcal L)}A(\mathfrak D,\lambda-\kappa_P,t)\langle\nabla\pi(\lambda),\rho\rangle d\lambda\wedge dt\notag
=&\frac12\int_{\tilde W}A(\mathfrak D,\lambda-\kappa_P,t)\nu(\kappa_P) \pi(\lambda)d\sigma_0\wedge dt\notag\\
&-\frac12\int_{\Delta_\mathscr Z(\mathcal L)}\langle\nabla(A(\mathfrak D,\lambda-\kappa_P,t)),\kappa_P\rangle\pi(\lambda) d\lambda\wedge dt.
\end{align}
Clearly,
\begin{align}\label{wk-one-para-eq-6}
&\int_{\Delta_\mathscr Z(\mathcal L)}\langle\nabla(A(\mathfrak D,\lambda-\kappa_P,t)),\kappa_P\rangle\pi(\lambda) d\lambda\wedge dt\notag\\
=&\sum_{x\not=x_0}\sum_{a=1}^N\int_{\tilde\Omega_a'\cup\tilde\Omega^o_a}\frac{\ell_{D_a(x)}(\kappa_P)}{h_{D_a(x)}}\pi(\lambda)d\lambda\wedge dt+\sum_{a=1}^N\int_{\tilde\Omega_a'}\frac{\ell_{D_a(x_0)}(\kappa_P)}{h_{D_a(x_0)}}\pi(\lambda)d\lambda\wedge dt\notag\\
&+\int_{\Delta_\mathscr Z^o(\mathcal L)}\frac{\ell_0(\kappa_P)}{h_0}\pi(\lambda)d\lambda\wedge dt.
\end{align}

Plugging \eqref{wk-one-para-eq-4}-\eqref{wk-one-para-eq-6} into \eqref{wk-h0not=0}, the coefficient of the $k^n$-term is
\begin{align*}
&\frac12\int_{\partial\Delta_\mathscr Z(\mathcal L)}A(\mathfrak D,\lambda-\kappa_P,t)\pi(\lambda)d\tilde\sigma+\frac12\sum_{\tilde a=1}^{\tilde N}\sum_{x\in C}\int_{\tilde\Omega_{\tilde a}}(\frac1{h_{\tilde D_{\tilde a}(x)}}-1)\pi(\lambda)d\lambda\wedge dt\notag\\
&-\int_{\Delta_\mathscr Z(K_X^{-1})}A(\mathfrak d,\lambda-\kappa_P)\pi(\lambda)d\lambda+\int_{\tilde\Delta_\mathscr Z(\mathfrak d)}\pi(\lambda)d\lambda\wedge dt\notag\\
&+\int_{\Delta_\mathscr Z(\mathcal L)}A(\mathfrak D,\lambda-\kappa_P,t)\langle\nabla\pi(\lambda),\rho\rangle d\lambda\wedge dt\notag\\
=&\frac12n\int_{\Delta_\mathscr Z(\mathcal L)}A(\mathfrak D,\lambda-\kappa_P,t)\pi(\lambda) d\lambda\wedge dt-\frac12\int_{\Delta_\mathscr Z(K_X^{-1})}A(\mathfrak d,\lambda-\kappa_P)\pi(\lambda)d\lambda\notag\\
&+\frac12\int_{\Delta_\mathscr Z(\mathcal L)}A(\mathfrak D,\lambda-\kappa_P,t)\pi(\lambda) d\lambda\wedge dt+\frac12\int_{\Delta_\mathscr Z^o(\mathcal L)}(\frac{1-m_0}{h_0}-1+a_{x_0})\pi(\lambda) d\lambda\wedge dt.
\end{align*}
Combining with \eqref{h0(X,Lk)-eq-one-para} and \eqref{Fut-def},
\begin{align*}
{\rm Fut}(\mathcal X,\mathcal L)
=&\frac1{V}\int_{\Delta_\mathscr Z(K_X^{-1})}A(\mathfrak d,\lambda-\kappa_P)\pi(\lambda)d\lambda
-\frac1{V}\int_{\Delta_\mathscr Z(\mathcal L)}A(\mathfrak D,\lambda-\kappa_P,t)\pi(\lambda) d\lambda\wedge dt\notag\\&-\frac1{V}\int_{\Delta_\mathscr Z^o(\mathcal L)}(\frac{1-m_0}{h_0}-1+a_{x_0})\pi(\lambda) d\lambda\wedge dt.
\end{align*}
Using \eqref{tilde-DZ(d)} and \eqref{tilde-A}, we can rewrite the integrations on the right-hand side as integrations over $\Delta_\mathscr Z(K_X^{-1})$, and get
\begin{align}\label{fut-inv-integral}
{\rm Fut}(\mathcal X,\mathcal L)
=&-\frac1{V}\int_{\Delta_\mathscr Z(K_X^{-1})}\ell_0(\lambda-\kappa_P)A(\mathfrak d,\lambda-\kappa_P)\pi(\lambda)d\lambda-h_0(a_{x_0}-1)\notag\\
&-\frac{h_0}{2V}\int_{\Delta_\mathscr Z(K_X^{-1})}(A(\mathfrak d,\lambda-\kappa_P)-2A_{x_0}(\mathfrak d,\lambda-\kappa_P))A(\mathfrak d,\lambda-\kappa_P)\pi(\lambda)d\lambda.
\end{align}
For the last term in the above equation, it holds
\begin{align*}
&\frac{h_0}{2V}\int_{\Delta_\mathscr Z(K_X^{-1})}(A(\mathfrak d,\lambda-\kappa_P)-2A_{x_0}(\mathfrak d,\lambda-\kappa_P))A(\mathfrak d,\lambda-\kappa_P)\pi(\lambda)d\lambda\notag\\
=&\frac{h_0}{2V}\int_{\Delta_\mathscr Z(K_X^{-1})}\left((A(\mathfrak d,\lambda-\kappa_P)-A_{x_0}(\mathfrak d,\lambda-\kappa_P))^2-A_{x_0}^2(\mathfrak d,\lambda-\kappa_P)\right)\pi(\lambda)d\lambda\notag\\
=&\frac{h_0}{V}\int_{\Delta_\mathscr Z(K_X^{-1})}\left(\int_{-A_{x_0}(\mathfrak d,\lambda-\kappa_P)} ^{A(\mathfrak d,\lambda-\kappa_P)-A_{x_0}(\mathfrak d,\lambda-\kappa_P)} tdt\right)\pi(\lambda)d\lambda\notag\\
=&\frac{h_0}V\int_{\Delta_{x_0}^O(K_X^{-1})}t\pi(\lambda) d\lambda\wedge dt-h_0(a_{x_0}-1).
\end{align*}
Plugging this into \eqref{fut-inv-integral} one directly gets \eqref{Fut-one-para-eq}.
\end{proof}

Using Theorem \ref{Fut-one-para-thm}, Remark \ref{base-change-rmk} and \eqref{base-change-of-NA}, one directly gets
\begin{cor}\label{MNA-coro}
Let $(\mathcal X,\mathcal L)$ be the $G$-equivariant normal test configuration of $(X,L)$ associated to $(v_0,m)\in\mathscr V\times\mathbb Z_{<0}$ defined in Proposition \ref{tc-to-v0}. Then
\begin{align*}
{\rm M}^{\rm NA}(\mathcal X,\mathcal L)=\langle\kappa_P-\mathbf{b}(\Delta_{x_0}^O(K_X^{-1})),v_0\rangle.
\end{align*}
\end{cor}

\begin{proof}
This is a consequence of
$${\rm M}^{\rm NA}(\mathcal X,\mathcal L)=-\frac1{m}{\rm M}^{\rm NA}(\mathcal X^{(-m)},\mathcal L^{(-m)})=-\frac1{m}{\rm Fut}(\mathcal X^{(-m)},\mathcal L^{(-m)}).$$
\end{proof}

\subsection{The non-Archimedean J-functional}
Let $v_0=\ell_0+h_0q_{x_0}\in\mathscr Q_{x_0,+}$ for some $x_0\in C$, and $\mathfrak d=\mathfrak d_\mathfrak a$ be an anti-canonical divisor given by \eqref{anti-can-div} or \eqref{anti-can-div-quasi-homo}. 

We have
\begin{prop}\label{J-NA-prop}
Let $X$ be a $\mathbb Q$-Fano $G$-variety of complexity 1 with ${\rm k}(X)^B\cong{\rm k}(\mathbb P^1)$. Let $(\mathcal X,\mathcal L)$ be a test configuration of $(X,K_X^{-1})$ that is associated to some integral $v_0=\ell_0+h_0q_{x_0}\in\mathscr Q_{x_0,+}$ for some $x_0\in C$ (if $h_0=0$ then $v_0\in\mathscr Q\subset\mathscr Q_{x,+}$ for all $x\in C$) and $m=-1$. Then
\begin{align}\label{JNA-eq}
{\rm J^{NA}}(\mathcal X,\mathcal L)=\frac1V\int_{\Delta_{x_0}^O(K_X^{-1})}(\max_{(\lambda,t)\in\Delta_{x_0}^O(K_X^{-1})}\langle(\lambda,t),v_0\rangle-\langle(\lambda,t),v_0\rangle)\pi(\lambda)d\lambda\wedge dt.
\end{align}
\end{prop}

\begin{proof}
Fix any $\mathfrak d$ given by \eqref{anti-can-div} in the one-parameter case or \eqref{anti-can-div-quasi-homo} in the quasihomogeneous case, which is the divisor of a $B$-seiinvariant section of weight $\kappa_P$. By \eqref{tilde-DZ(d)}, \eqref{tilde-A} and direct computation, we have
\begin{align}\label{E-NA}
\frac1{V}\mathcal L^{\cdot{n+1}}=&\frac1V\int_{\Delta_\mathscr Z(\mathcal L)}A(\mathfrak D,\lambda-\kappa_P,t)\pi(\lambda)d\lambda\wedge dt\notag\\
=&\frac1V\int_{\Delta_\mathscr Z(K_X^{-1})}(m_0+\ell_0(\lambda-\kappa_P))A(\mathfrak d,\lambda-\kappa_P)\pi(\lambda)d\lambda\notag\\
&+\frac1V\int_{\Delta_\mathscr Z(K_X^{-1})}h_0(\frac12A(\mathfrak d,\lambda-\kappa_P)-A_{x_0}(\mathfrak d,\lambda-\kappa_P))A(\mathfrak d,\lambda-\kappa_P)\pi(\lambda)d\lambda\notag\\
=&\frac1V\int_{\Delta_{x_0}^O(K_X^{-1})}\langle(\lambda,t),v_0\rangle\pi(\lambda)d\lambda\wedge dt+m_0-h_0(a_{x_0}-1)-\ell_0(\kappa_P),
\end{align}
where in the last line we used \eqref{Delta-O-K-def}, the definition of $\Delta_{x_0}^O(K_X^{-1})$. 

On the other hand, note that the upper bound of the support of the DH-measure of $(\mathcal X,\mathcal L)$ is $$\Lambda_{\max}(\mathcal X,\mathcal L)=\max_{\lambda\in\Delta_\mathscr Z(K_X^{-1})}\tau^0(\lambda-\kappa_P),$$
where the function $\tau^0(\cdot)$ is defined by \eqref{tau-0}. Again by Lemma \ref{polytope-anti-can} we get
\begin{align*}
\Lambda_{\max}(\mathcal X,\mathcal L)=&\max_{\lambda\in\Delta_\mathscr Z(K_X^{-1})}\{\tau^0(\lambda-\kappa_P)+h_0(a_{x_0}-1)\}-h_0(a_{x_0}-1)\\
=&\max_{(\lambda,t)\in\Delta_{x_0}^O(K_X^{-1})}\{m_0+\ell_0(\lambda-\kappa_P)+h_0t\}-h_0(a_{x_0}-1)\\
=&\max_{(\lambda,t)\in\Delta_{x_0}^O(K_X^{-1})}\langle(\lambda,t),v_0\rangle+m_0-h_0(a_{x_0}-1)-\ell_0(\kappa_P).
\end{align*}
Combining with \eqref{E-NA} we get the Proposition.

\end{proof}

\subsection{K-stability criterion}
Recall Definition \ref{K-stab-def-Fut}. Theorem \ref{K-st-thm} is then a direct consequence of Theorem \ref{Fut-one-para-thm}.

\begin{proof}[Proof of Theorem \ref{K-st-thm}]
The proof of (1)$\Leftrightarrow$(2) is a combination of Proposition \ref{Aut-cone-prop} and Theorem \ref{Fut-one-para-thm}.

It remains to show (1)$\Leftrightarrow$(3) when $X$ is quasihomogeneous and is not a $G\times{\rm k}^\times$-spherical variety. In this case, $\mathscr A_x=\mathscr A$ for any $x\in C$ by Corollary \ref{non-spherical-cor}. Assume that (1) fails. By Theorem \ref{Fut-one-para-thm} there is a non-zero $v$ in some $\mathscr V_x$ so that the test configuration induced by $v$ is non-product and has non-positive Futaki invariant. Hence (3) fails, a contradiction. Thus we get the direction (3)$\Rightarrow$(1).

It remains to show (1)$\Rightarrow$(3). 
Both $G$ and the torus ${\mathbf T}$ (see Section 2.2.1 for definition) are reductive subgroups in ${\rm Aut}(X)$, and they commute with each other.\footnote{More precisely, the image of $G$ in ${\rm Aut}(X)$ of the homomorphism $G\to{\rm Aut}(X)$.} Thus ${\mathbf G}$ is the homomorphism image of $G\times{\mathbf T}$ in ${\rm Aut}(X)$, which is reductive and ${\mathbf T}$ is contained in the centre of ${\mathbf G}$. Also note that ${\rm Lie}({\mathbf T})\cong\mathscr A\hookrightarrow\bar{\mathscr A}$, the linear part of the valuation cone $\bar{\mathscr V}$ of $\mathcal X$. Any $G$-equivariant test configuration of $(X,L)$ is automatically ${\mathbf G}$-equivariant. It suffices to show \eqref{uni-sta-def} holds by taking $\sigma\in{\mathbf T}$ for any $G$-equivariant special test configuration.

We adopt the arguments used in spherical cases (see for example \cite{Do,Del-2020-09}). For any test configuration $(\mathcal X,\mathcal L)$ associated to $v_0$, we can twist $(\mathcal X,\mathcal L)$ by an $\ell'\in\mathscr A\cong{\rm Lie}({\mathbf T})$. 
Fix any inner product on $\mathscr Q_{x,+}\cong\mathscr Q\times\mathbb Q_+$ which is invariant under the little Weyl group $W_X$ of $X$ (see \cite{Knop94} and \cite[Section]{Timashev-book} for details on $W_X$). Choose $\ell'$ so that $v_0+\ell'$ is orthogonal to $\mathscr A$ with respect to this inner product. We say such an $(\mathcal X',\mathcal L')$ is \emph{normalized}. Note that $\mathscr A$ is an integral, $W_X$-invariant subspace of $\mathscr Q$ (cf. \cite[Theorem 7.4]{Knop94}). The element $\ell'$ is always rational. Assume that $q(v_0+\ell')$ is primitive. By Lemma \ref{twist-lem}, the base change $({\mathcal X'}^{(q)},{\mathcal L'}^{(q)})$ of $(\mathcal X',\mathcal L')$ is a test configuration associated to $q(v_0+\ell')$. In particular, $({\mathcal X'}^{(q)},{\mathcal L'}^{(q)})$ has reduced central fibre.
By \eqref{base-change-of-NA} and \eqref{JNA-eq},
\begin{align}\label{J(X',L')}
{\rm J^{NA}}(\mathcal X',\mathcal L')=&\frac1q{\rm J^{NA}}({\mathcal X'}^{(q)},{\mathcal L'}^{(q)})&\notag\\=&\frac1V\int_{\Delta_{x_0}^O(K_X^{-1})}(\max_{(\lambda,t)\in\Delta_{x_0}^O(K_X^{-1})}\langle(\lambda,t),v_0+\ell'\rangle-\langle(\lambda,t),v_0+\ell'\rangle)\pi(\lambda)d\lambda\wedge dt.
\end{align}
Also, by \eqref{Fut-one-para-eq},
$${\rm M}^{\rm NA}(\mathcal X,\mathcal L)={\rm Fut}(\mathcal X,\mathcal L)=\langle\kappa_P-\mathbf{b}(\Delta_{x_0}^O(K_X^{-1})),v_0\rangle.$$
Since Condition \eqref{K-ss-eq} implies
$$\langle\kappa_P-\mathbf{b}(\Delta_{x_0}^O(K_X^{-1})),\ell'\rangle=0.$$
We get
\begin{align}\label{seq-bar}
\langle\kappa_P-\mathbf{b}(\Delta_{x_0}^O(K_X^{-1})),v_0+\ell'\rangle={\rm M}^{\rm NA}(\mathcal X,\mathcal L).
\end{align}

Now we prove that when (1) is true, there is a constant $\epsilon_0>0$ so that
\begin{align}\label{K-us-eq}
{\rm M}^{\rm NA}(\mathcal X,\mathcal L)\geq\epsilon_0\inf_{\ell'\in{\rm Lie}({\mathbf T})}{\rm J^{NA}}(\mathcal X_\ell,\mathcal L_\ell)
\end{align}
for any $(\mathcal X,\mathcal L)$ associated to some integral $v_0\in\mathscr V$. Clearly \eqref{K-us-eq} implies (3).

Let us use a standard argument by contradiction introduced by \cite{Do}. Suppose that \eqref{K-us-eq} is not true. Then there is a sequence of test configurations $\{(\mathcal X_k,\mathcal L_k)\}_{k=1}^{+\infty}$ so that
\begin{align}\label{seq-tc-1}
\left\{\begin{aligned}
&\inf_{\ell'\in{\rm Lie}({\mathbf T})}{\rm J^{NA}}(\mathcal X'_{k\,\ell'},\mathcal L'_{k\,\ell'})>0,~\forall k\in\mathbb N_+,\\
&\lim_{k\to+\infty}\frac{{\rm M}^{\rm NA}(\mathcal X_k,\mathcal L_k)}{\inf_{\ell'\in{\rm Lie}({\mathbf T})}{\rm J^{NA}}(\mathcal X'_{k\,\ell'},\mathcal L'_{k\,\ell'})}=0.\end{aligned}\right.
\end{align}
Assume that each $(\mathcal X_k,\mathcal L_k)$ is associated to an integral $v_k\in\mathscr Q_{x_k,+}$. As above, there is a sequence $\{\ell'_k\in\mathscr A\}_{k=1}^{+\infty}$ such that each $v'_k:=v_k+\ell'_k\in\mathscr A^\perp$. Then the twisted sequence $\{(\mathcal X'_k,\mathcal L'_k):=(\mathcal X'_{k\,\ell'_k},\mathcal L'_{k\,\ell'_k})\}_{k=1}^{+\infty}$ satisfies
\begin{align}\label{seq-tc-2}
c_k:={\rm J^{NA}}(\mathcal X'_k,\mathcal L'_k)\geq\inf_{\ell'\in{\rm Lie}({\mathbf T})}{\rm J^{NA}}(\mathcal X'_{k\,\ell'},\mathcal L'_{k\,\ell'})>0,~\forall k\in\mathbb N_+,
\end{align}
and
\begin{align}\label{seq-tc-4}
0\leq\lim_{k\to+\infty}\frac{{\rm M}^{\rm NA}(\mathcal X_k,\mathcal L_k)}{{\rm J^{NA}}(\mathcal X'_{k\,\ell'_k},\mathcal L'_{k\,\ell'_k})}\leq\lim_{k\to+\infty}\frac{{\rm M}^{\rm NA}(\mathcal X_k,\mathcal L_k)}{\inf_{\ell'\in{\rm Lie}({\mathbf T})}{\rm J^{NA}}(\mathcal X'_{k\,\ell'},\mathcal L'_{k\,\ell'})}=0.
\end{align}

Note that there are only finitely many different $\Delta_x^O(K_X^{-1})$'s. Thus there is a subsequence $\{k_j\}_{j=1}^{+\infty}$ satisfying
$$\Delta_{x_{k_j}}^O(K_X^{-1})\cong\Delta_{x_{k_1}}^O(K_X^{-1}),~j\in\mathbb N_+,$$
and we can further identify $\{v_{k_j}\}_{j=1}^{+\infty}$ (hence also $\{v'_{k_j}\}_{j=1}^{+\infty}$) with a subsequence (still denoted by the same letter) in $\mathscr Q_{x_1,+}\cong\mathscr Q\times\mathbb Q_+$ so that \eqref{seq-tc-1}-\eqref{seq-tc-4} still hold.

By \eqref{J(X',L')} and \eqref{seq-tc-2},
\begin{align}\label{seq-tc-3}
&\frac1V\int_{\Delta_{x_1}^O(K_X^{-1})}(\max_{(\lambda,t)\in\Delta_{x_1}^O(K_X^{-1})}\langle(\lambda,t),\frac{v'_k}{c_k}\rangle-\langle(\lambda,t),\frac{v'_k}{c_k}\rangle)\pi(\lambda)d\lambda\wedge dt\notag\\
=&\frac1{c_k}{\rm J^{NA}}(\mathcal X'_k,\mathcal L'_k)=1,~\forall k\in\mathbb N_+.
\end{align}
We see that $\{\frac{v'_k}{c_k}\}_{k=1}^{+\infty}$ is a bounded sequence. Thus, up to passing to a subsequence, we can assume $\{\frac{v'_k}{c_k}\}_{k=1}^{+\infty}$ converges to some $v'_*\in\mathscr Q_{x_1,+}\cap\mathscr A^\perp$.

On the other hand,
By \eqref{seq-bar} and \eqref{seq-tc-4},
\begin{align*}
\lim_{k\to+\infty}\langle\kappa_P-\mathbf{b}(\Delta_{x_1}^O(K_X^{-1})),\frac{v'_k}{c_k}\rangle=0.
\end{align*}
Taking limit one gets
\begin{align*}
\langle\kappa_P-\mathbf{b}(\Delta_{x_{k_1}}^O(K_X^{-1})),v'_*\rangle=0.
\end{align*}
Together with \eqref{K-ps-eq} this implies that $v'_*\in\mathscr A$. Thus $v'_*=0$. A contradiction to \eqref{seq-tc-3}.
\end{proof}

\begin{rem}
We remark that \eqref{J(X',L')} can also be derived from \eqref{grading-twist} and \cite[Eq. (2.3)]{Hisamoto} (based on \cite[Section 7.2]{Boucksom-Hisamoto-Jonsson}) without using Lemma \ref{twist-lem}.
\end{rem}

There is a variant of Theorem \ref{K-st-thm} for horospherical $G$-varieties of complexity 1, which is much more simplified. The following is a generalization of the K-polystability criterion \cite[Theorem 4.12]{Ilten-Suss-Duke} of Fano $T$-varieties of complexity 1 to horospherical $G$-varieties of complexity 1.
\begin{theo}\label{horo-st-thm}
Let $X$ be a $\mathbb Q$-Fano, horospherical $G$-varieties of complexity 1 with ${\rm k}(X)^B\cong{\rm k}(\mathbb P^1)$. Then $X$ is $G$-equivariantly $K$-semistable if and only if
\begin{align}\label{K-ss-horo-eq}
\kappa_P-\mathbf{b}(\Delta_{x}^O(K_X^{-1}))\in\{O\}\times\mathbb R_{\geq0}\subset\Gamma_\mathbb R\times\mathbb R_{\geq0},~\forall x\in C.
\end{align}
Moreover, $X$ is $G$-equivariantly $K$-polystable if and only if in addition it holds
\begin{align}\label{K-ps-horo-eq}
\kappa_P-\mathbf{b}(\Delta_{x}^O(K_X^{-1}))\in\{O\}\times\mathbb R_{>0}\subset\Gamma_\mathbb R\times\mathbb R_{>0},~\forall x\in C,
\end{align}
when $X$ is not a $G\times{\rm k}^\times$-spherical variety, or in addition
\begin{align}\label{K-ps-horo-Gtimesk}
\mathbf{b}(\Delta_{x}^O(K_X^{-1}))=\kappa_P~\text{if and only if}~\mathscr A_x\setminus\mathscr Q\not=\emptyset,
\end{align}
otherwise.
\end{theo}

\begin{proof}
When $X$ is horospherical, $\mathscr V\cap\mathscr Q=\mathscr Q$ (cf. \cite{Vinberg-86,Popov-1986}) and $\mathscr V_x\cong\mathscr Q\times\mathbb Q_+$ for any $x\in C$ (cf. \cite{Langlois-Terpereau-2016}). Also by Propositions \ref{Aut-cone-prop}, it holds $\mathscr Q\subset\mathscr A_x$ for any $x\in C$. Combining with Corollary \ref{central-spherical}, $\mathscr A_x\setminus\mathscr Q\not=\emptyset$ if and only if $X$ is a $G\times{\rm k}^\times$-spherical variety. The conditions \eqref{K-ss-horo-eq} (and \eqref{K-ps-horo-eq}, resp.) is equivalent to \eqref{K-ss-eq} (and \eqref{K-ps-eq}, resp.) in the horospherical case. The last condition \eqref{K-ps-horo-Gtimesk} is equivalent to that the Futaki invariant vanishes precisely on product test configurations.
\end{proof}

\section{Examples}
In this section we apply Theorem \ref{K-st-thm} to concrete examples.

\subsection{The Mukai-umemura threefold}
The Mukai-Umemura threefold was constructed in \cite{Mukai-Umemura-1983} as a smooth Fano ${\rm SL}_2$-variety. The existence of K\"ahler-Einstein metrics on this threefold (and equivalently, uniform K-stability) has been proved in \cite[Section 5.4]{Don-2008} by computing the $\alpha$-invariant. In the following we view the Mukai-Umemura threefold as an equivariant completion of ${\rm SL}_2/H$, where $H$ is the binary icosahedral group, and present concretely how the theorems established in the previous sections apply on this well-known example. We also remark that the equivariant K-polystability of smooth, Fano quasihomogeneous ${\rm SL}_2$-varieties of complexity 1 has been studied in \cite{Rogers-2022} via a different approach.

Let us recall the construction in \cite{Mukai-Umemura-1983}: Denote by $s^{k}$ the irreducible $(G:=){\rm SL}_2$-representation which consists of homogeneous symmetric polynomials of two variables of degree $k\in\mathbb N$. We identify $\mathbb P^1$ with the Riemann sphere $S^2$ so that $G$ acts on it, and ${\rm SU}_2\subset{\rm SL}_2$ acts on $S^2$ via the epimorphism ${\rm SU}_2\to{\rm SO}_3$. Also we identify a polynomial in $s^{12}$ with its roots located on $S^2$. There is an element $f_v\in s^{12}$ so that its roots are the 12 vertices of an icosahedron inscribed in $S^2$, the stabilizer subgroup of $f_v$ in $G$ is exactly $H$. Then the closure $X$ of the $G$-orbit of $[f_v]\in\mathbb P(s^{12})$ is the Mukai-Umemura threefold, which is a smooth Fano threefold (cf. \cite{Mukai-Umemura-1983, Don-2008}).

As shown in \cite{Mukai-Umemura-1983}, there are three $G$-orbits in $X$. First, we have
\begin{itemize}
\item The unique open orbit $\mathcal O:=G\cdot[f_v]\cong G/H$.
\end{itemize}
The complement of $\mathcal O$ is a $G$-divisor $D_\infty$ which is the closure of
\begin{itemize}
\item The orbit $\mathcal O_1:=G\cdot[f_1]$, where $f_1\in s^{12}$ is a polynomial which has a root of multiplicity 11 and a simple root.
\end{itemize}
The divisor $D_\infty$ contains $\mathcal O_1$ as its unique open $G$-orbit and the complement $\Delta_\infty$ of $\mathcal O_1$ in $D_\infty$ is
\begin{itemize}
\item The unique closed $G$-orbit in $X$, which consists of $G\cdot[f_2]$, where $f_2\in s^{12}$ is a polynomial which has a root of multiplicity 12.
\end{itemize}

The coloured data of $G/H$ are given in \cite[Section 5.6]{Timashev-1997}. Let us briefly recall the construction in \cite{Timashev-1997} here. Recall the action of $G$ on $S^2$ and an icosahedron $\mathbf I$ inscribed in this $S^2$. Then $H$ is the pre-image of the symmetry group of $\mathbf I$ in $G$. The vertices, edge midpoints, and face centers of $\mathbf I$ give those points on $S^2$ so that $H$ acts with non-trivial subgroup of stabilizers. In fact, they form three $H$-orbits, respectively. The corresponding points are denoted by $x_v$, $x_e$ and $x_f$, respectively. Also they define three polynomials $f_v\in s^{12}$, $f_e\in s^{30}$ and $f_f\in s^{20}$.

Denote by $B$ the subgroup of upper-triangular matrixes in $G$ and fix it as the positive Borel subgroup. Also let $\omega$ be the corresponding fundamental weight of $G$. It is showed in \cite[Section 5.6]{Timashev-1997} that the subregular semiinvariants are $f_v$, $f_e$ and $f_f$ with corresponding $B\times H$-biweights $(12\omega; 1)$, $(30\omega; 1)$ and $(20\omega; 1)$. The regular semiinvariants fill the two-dimensional subspace ${\rm k}[G]^{(B\times H)}_{(60\omega;1)}$ spanned by $f_v^5$, $f_e^2$ and $f_f^3$, where $$f_v^5+f_e^2+f_f^3=0.$$
Functions in ${\rm k}[G]^{(B\times H)}_{(60\omega;1)}$ give all the colours in $X$, and there is no colour of central type (cf. \cite[Section 2.2]{Timashev-1997}). Also $\Gamma=\mathbb Z\cdot2\omega$ with generator
$$e_{2\omega}:=\frac{f_vf_f}{f_e}.$$
In particular, ${\rm rk}(X)=1$.

The valuation cone $\mathscr V$ is given as follows: an element $hq_x+a\frac12\omega^*\in\mathscr Q_{x,+}$\footnote{We put $\frac12$ here since the dual lattice $\Gamma^*=\frac12\mathbb Z\omega^*$} is contained in $\mathscr V_x$ if and only if
\begin{align*}
(a,h)~\text{satisfies}~\left\{
\begin{aligned}
&-a\geq h\geq0,~\text{if}~x\not=x_v,x_f,\\
&h,-a\geq0,~\text{if}~x=x_v,x_f.
\end{aligned}
\right.
\end{align*}

Colours are computed in \cite[Section 5.6]{Timashev-1997}, and the coloured data of $X$ by \cite[Section 7.8]{Rogers-2022}. We list the data of $B$-stable divisors in Table-1.1.
\begin{table}[h]
Table-1.1: $B$-stable divisors in $X$\\
\begin{tabular}{|c|c|c|}
\hline
Colour                                                 & $x_D$                                            & $v_D$                                      \\ \hline
$X_x,~x\not=x_e,x_v,x_f$                               & $x$                                              & $q_x$                                      \\ \hline
$X_{x_e}$                                             & $x_e$                                            & $-\frac12\omega^*+2q_{x_e}$                    \\ \hline
$X_{x_v}$                                             & $x_v$                                            & $\frac12\omega^*+5q_{x_v}$                   \\ \hline
$X_{x_f}$                                             & $x_f$                                            & $\frac12\omega^*+3q_{x_f}$                    \\ \hline
$G$-stable divisor                                                 & $x_D$                                            & $v_D$                                      \\ \hline
$D_\infty$                                             & $x_v$                                            & $q_{x_v}$   \\ \hline
\end{tabular}
\end{table}
\\
Clearly the associated parabolic subgroup of $X$ is $B$ and $\kappa_P=2\omega$.

The coloured fan of $X$ consists of one maximal hypercone $\mathscr C$ of type \uppercase\expandafter{\romannumeral 2}. More precisely, it is a hypercone of type A$_1$ (in the sense of \cite[Section 9]{Luna-Vust}) with data $$\mathscr W=\{q_{x_v}\},~\mathscr R=\{X_x|x\not=x_v\}.$$
This hypercone is shown below in Fig-2, where its intersection with $\mathscr Q_{x_+}$ is marked as the dark areas, and the valuation cones as hatched. Colors are represented by circles and $D_\infty$ by a dark dot.

\begin{figure}[h]
\begin{tikzpicture}[scale=0.8]
\fill[color=gray!20] (0,0)--(-2.2,0)--(-2.2,5)--(0,5)--(0,0);
\draw [semithick] (-2.2,0)--(1.2,0);
\draw [semithick] (0,-0.2)--(0,5.2);
\draw (1,0) node {$\cdot$};
\draw (-1,0) node {$\cdot$};
\draw (0.2,-0.2) node {\scriptsize{$O$}};
\draw (-1,-0.2) node {\scriptsize{$-1$}};
\draw (1,-0.2) node {\scriptsize{$1$}};
\draw (-0.3,1) node {\scriptsize{${X_x}$}};
\draw (0.2,1) node {\scriptsize{$1$}};
\draw (0,1) node {$\circ$};
\draw [dashed] (0,0)--(-2.2,2.2);
\draw [dashed] (-2,0)--(-2,2);
\draw [dashed] (-1.75,0)--(-1.75,1.75);
\draw [dashed] (-1.5,0)--(-1.5,1.5);
\draw [dashed] (-1.25,0)--(-1.25,1.25);
\draw [dashed] (-1,0)--(-1,1);
\draw [dashed] (-0.75,0)--(-0.75,0.75);
\draw [dashed] (-0.5,0)--(-0.5,0.5);
\draw [dashed] (-0.25,0)--(-0.25,0.25);
\draw (0,-0.5) node {\scriptsize{$x\not=x_v,x_e,x_f$}};
\draw (1,1) node {$\mathscr Q_{x,+}$};
\draw (-2.4,0) node {$\mathscr Q$};
\end{tikzpicture}
\begin{tikzpicture}[scale=0.8]
\fill[color=gray!20] (0,0)--(-2.2,0)--(-2.2,4.4)--(0,0);
\draw [semithick] (-2.2,0)--(1.2,0);
\draw [semithick] (0,-0.2)--(0,5.2);
\draw [semithick] (0,0)--(-2.2,4.4);
\draw (1,0) node {$\cdot$};
\draw (-1,0) node {$\cdot$};
\draw (0.2,-0.2) node {\scriptsize{$O$}};
\draw (-1,-0.2) node {\scriptsize{$-1$}};
\draw (1,-0.2) node {\scriptsize{$1$}};
\draw (0.2,1) node {\scriptsize{$1$}};
\draw (0.2,2) node {\scriptsize{$2$}};
\draw (-1,2) node {$\circ$};
\draw (-0.6,2) node {\scriptsize{${X_{x_e}}$}};
\draw (0,1) node {$\cdot$};
\draw (0,2) node {$\cdot$};
\draw [dashed] (0,0)--(-2.2,2.2);
\draw [dashed] (-2,0)--(-2,2);
\draw [dashed] (-1.75,0)--(-1.75,1.75);
\draw [dashed] (-1.5,0)--(-1.5,1.5);
\draw [dashed] (-1.25,0)--(-1.25,1.25);
\draw [dashed] (-1,0)--(-1,1);
\draw [dashed] (-0.75,0)--(-0.75,0.75);
\draw [dashed] (-0.5,0)--(-0.5,0.5);
\draw [dashed] (-0.25,0)--(-0.25,0.25);
\draw (0,-0.5) node {\scriptsize{$x=x_e$}};
\draw (1,1) node {$\mathscr Q_{x_e,+}$};
\draw (-2.4,0) node {$\mathscr Q$};
\end{tikzpicture}
\begin{tikzpicture}[scale=0.8]
\fill[color=gray!20] (0,0)--(-2.2,0)--(-2.2,5)--(0,5)--(0,0);
\draw [semithick] (-2.2,0)--(1.2,0);
\draw [semithick] (0,-0.2)--(0,5.2);
\draw (1,0) node {$\cdot$};
\draw (-1,0) node {$\cdot$};
\draw (0.2,-0.2) node {\scriptsize{$O$}};
\draw (-1,-0.2) node {\scriptsize{$-1$}};
\draw (1,-0.2) node {\scriptsize{$1$}};
\draw (0.2,1) node {\scriptsize{$1$}};
\draw (0.2,2) node {\scriptsize{$2$}};
\draw (0.2,3) node {\scriptsize{$3$}};
\draw (0.2,4) node {\scriptsize{$4$}};
\draw (0.2,5) node {\scriptsize{$5$}};
\draw (1,5) node {$\circ$};
\draw (1.4,5) node {\scriptsize{${X_{x_v}}$}};
\draw (0,1) node {$\bullet$};
\draw (-0.4,1) node {\scriptsize{${D_\infty}$}};
\draw (0,1) node {$\cdot$};
\draw (0,2) node {$\cdot$};
\draw (0,3) node {$\cdot$};
\draw (0,4) node {$\cdot$};
\draw (0,5) node {$\cdot$};
\draw [dashed] (-2,0)--(-2,5);
\draw [dashed] (-1.75,0)--(-1.75,5);
\draw [dashed] (-1.5,0)--(-1.5,5);
\draw [dashed] (-1.25,0)--(-1.25,5);
\draw [dashed] (-1,0)--(-1,5);
\draw [dashed] (-0.75,0)--(-0.75,5);
\draw [dashed] (-0.5,0)--(-0.5,5);
\draw [dashed] (-0.25,0)--(-0.25,5);
\draw (0,-0.5) node {\scriptsize{$x=x_v$}};
\draw (1,1) node {$\mathscr Q_{x_v,+}$};
\draw (-2.4,0) node {$\mathscr Q$};
\end{tikzpicture}
\begin{tikzpicture}[scale=0.8]
\fill[color=gray!20] (0,0)--(-2.2,0)--(-2.2,5)--(1.2,5)--(1.2,3.6)--(0,0);
\draw [semithick] (-2.2,0)--(1.2,0);
\draw [semithick] (0,-0.2)--(0,5.2);
\draw [semithick] (0,0)--(1.2,3.6);
\draw (1,0) node {$\cdot$};
\draw (-1,0) node {$\cdot$};
\draw (0.2,-0.2) node {\scriptsize{$O$}};
\draw (-1,-0.2) node {\scriptsize{$-1$}};
\draw (1,-0.2) node {\scriptsize{$1$}};
\draw (0.2,1) node {\scriptsize{$1$}};
\draw (0.2,2) node {\scriptsize{$2$}};
\draw (0.2,3) node {\scriptsize{$3$}};
\draw (1,3) node {$\circ$};
\draw (1.4,3) node {\scriptsize{${X_{x_f}}$}};
\draw (0,1) node {$\cdot$};
\draw (0,2) node {$\cdot$};
\draw (0,3) node {$\cdot$};
\draw [dashed] (-2,0)--(-2,5);
\draw [dashed] (-1.75,0)--(-1.75,5);
\draw [dashed] (-1.5,0)--(-1.5,5);
\draw [dashed] (-1.25,0)--(-1.25,5);
\draw [dashed] (-1,0)--(-1,5);
\draw [dashed] (-0.75,0)--(-0.75,5);
\draw [dashed] (-0.5,0)--(-0.5,5);
\draw [dashed] (-0.25,0)--(-0.25,5);
\draw (0,-0.5) node {\scriptsize{$x=x_f$}};
\draw (1,1) node {$\mathscr Q_{x_f,+}$};
\draw (-2.4,0) node {$\mathscr Q$};
\end{tikzpicture}
\\
\begin{center}
\begin{tikzpicture}[scale=1.2]
\fill[color=gray!10] (1,0)--(3.1,2.1)--(2.1-1.2,2.1)--(-1.2,0)--(1,0);
\fill[color=gray!25] (1,0)--(1.6+0.2,1.8+0.6)--(1.6+0.2,3)--(-1.2,3)--(-1.2,0)--(1,0);
\fill[color=gray!50] (1,0)--(-0.8,2.4)--(-1.8,2.4)--(-1.2,0)--(1,0);
\fill[color=gray!75] (1,0)--(0.4,1.2)--(-1.8,1.2)--(-1.2,0)--(1,0);

\draw [semithick] (-1.2,0)--(1.2+0.6,0)--(0.6+0.6,1.2)--(-1.8,1.2)--(-1.2,0);
\draw [semithick] (-1.2,0)--(1.2+0.6,0)--(-0.6+1.2+0.6,2.4)--(-0.6-1.2,2.4)--(-1.2,0);
\draw [semithick] (-1.2,0)--(1.2+0.6,0)--(1.2+0.6,3)--(-1.2,3)--(-1.2,0);
\draw [semithick] (-1.2,0)--(1.2+0.6,0)--(2.1+1.2+0.6,2.1)--(2.1-1.2,2.1)--(-1.2,0);

\draw [thick] (-1.2,0)--(1,0);
\draw [thick] (1,0)--(0.4,1.2);
\draw (0.8,0.4) node {$\bullet$};
\draw [thick] (1,0)--(-0.8,2.4);
\draw (0.25,1) node {$\circ$};
\draw [thick] (1,0)--(1.6+0.2,1.8+0.6);
\draw (1.6,1.8) node {$\circ$};
\draw [thick] (1,0)--(3.1,2.1);
\draw (1.4,0.4) node {$\circ$};
\draw (0,-0.2) node {\scriptsize{$\mathscr Q$}};
\draw (1,-0.2) node {\scriptsize{$O$}};
\draw (-1.8,0.8) node {\scriptsize{$\mathscr C_{x_v}$}};
\draw (-1.8,1.6) node {\scriptsize{$\mathscr C_{x_e}$}};
\draw (-1.4,2.8) node {\scriptsize{$\mathscr C_{x_f}$}};
\draw (3.6,0.6) node {\scriptsize{$\mathscr C_{x},~x\not=x_v,x_e,x_f$}};
\end{tikzpicture}
\end{center}
Fig-2: Colored hypercone of the Mukai-Umemura threefold.\\ {\footnotesize (Marked as dark areas. In $\mathscr Q$, the generator of $\Gamma^*$ is $\frac12\omega^*$)}
\end{figure}

Now we compute anti-canonical divisors. By Theorem \ref{anti-can-div-quasi-homo}, for each anti-canonical divisor $\mathfrak a=\sum_{x\in \mathbb P^1}a_x\cdot x$ on $\mathbb P^1$, there is a anti-canonical divisor
\begin{align}\label{anti-can-div-mukai}
\mathfrak d_\mathfrak a=\sum_{x\not= x_v,x_e,x_f}a_xX_{x}+(-4+5a_{x_v})X_{x_v}+(-1+2a_{x_e})X_{x_e}+(-2+3a_{x_f})X_{x_f}+a_{x_v}D_\infty
\end{align}
of $X$, which is the divisor of a rational section of $K_X^{-1}$ of weight $2\omega$.

However, there is no effective divisor in \eqref{anti-can-div-mukai}. In fact, it is showed in \cite[Section 5.1]{Don-2008} that
$${\rm H}^0(X,K_X^{-1})\cong s^0\oplus s^{12},$$
which implies
$${\rm H}^0(X,K_X^{-1})^{(B)}_{2\omega}=0.$$
Also, \cite[Section 5.1]{Don-2008} gives an effective divisor
$$\mathfrak d_0=D_\infty$$
on $X$, which is the divisor of a $G$-invariant section in the component $s^0\subset K_X^{-1}$. The divisor $\mathfrak d_0$ can also be expressed as
$$\mathfrak d_0=\mathfrak d_\mathfrak a-{\rm div}(e_{2\omega})$$
with $$\mathfrak a=[x_v]+[x_f].$$
Also by choosing
$$\mathfrak a'=3[x_e]-[x_f],$$
we get another effective anti-canonical divisor
$$\mathfrak d_{\mathfrak a'}+5{\rm div}(e_{2\omega})=X_{x_v},$$
which is the divisor of a section in ${\rm H}^0(X,K_X^{-1})^{(B)}_{12\omega}\subset s^{12}$.

For convenience, below we work with the divisor $\mathfrak d_0$. Since it is $G$-invariant, we have $\Delta_{\mathscr Z}(\mathfrak d_0)=\Delta_{\mathscr Z}(K_X^{-1})$ and $A_x(\mathfrak d_0,\lambda)=A_x(K_X^{-1},\lambda-\kappa_P)$ for all $x$. Set $\lambda=k\omega$. We have
$$\Delta_{\mathscr Z}(K_X^{-1})=\{k\omega|k\in\mathbb Q,~0\leq k\leq 12\},$$
and the data are listed in Table-1.2. The ploytopes $\Delta^O_x(K_X^{-1})$ are drawn in Fig-3. Also,
$$A(\mathfrak d_0,\lambda)=\min\{\frac1{60}k,1-\frac1{12}k\}.$$
\begin{table}[h]
Table-1.2: Data of $K_X^{-1}$\\
\begin{tabular}{|c|c|c|}
\hline
$x$                                & $A_x(\mathfrak d,\lambda)$           & $\mathbf{b}(\Delta_x^O(K_X^{-1}))-\kappa_P=(k,t)$   \\ \hline
$x\not=x_e,x_v,x_f$                & $0$                                  & $(\frac{69}{11},-\frac{31}{33})$ \\ \hline
$x_e$                              & $-\frac14k$                          & $(\frac{69}{11},\frac{149}{132})$   \\ \hline
$x_v$                              & $\min\{1,\frac1{10}k\}$              & $(\frac{69}{11},-\frac34)$   \\ \hline
$x_f$                              & $\frac16k$                           & $(\frac{69}{11},-\frac{29}{22})$   \\ \hline
\end{tabular}
\end{table}

\begin{figure}[h]
\begin{center}
\begin{tikzpicture}[scale=0.5]
\fill[color=gray!20] (2,0)--(6.4,2.2)--(12,2.2)--(12,-1.5)--(2,-1.5)--(2,0);
\draw [dashed] (-0.5,0)--(12.2,0);
\draw [dashed] (0,2.2)--(0,-1.5);
\draw [dashed] (2,0)--(2,-1.5);
\draw [dashed] (2,0)--(6.4,2.2);
\draw (2,0) node {\scriptsize{$\cdot$}};
\draw (12,0) node {\scriptsize{$\cdot$}};
\draw (2,0.3) node {\scriptsize{$\kappa_P$}};
\draw (-0.3,0.3) node {\scriptsize{$O$}};
\draw (2,-0.3) node {\scriptsize{$2\omega$}};
\draw (12,-0.3) node {\scriptsize{$12\omega$}};

\draw [thick] (0,-1)--(12,-1)--(10,0.1667-1)--(0,-1);
\draw (8.273,-0.939) node {\scriptsize{$\times$}};

\draw (5,-2) node {\scriptsize{$x\not=x_v,x_e,x_f$}};
\end{tikzpicture}
\begin{tikzpicture}[scale=0.5]
\fill[color=gray!20] (2,0)--(6.4,2.2)--(12,2.2)--(12,-1.5)--(2,-1.5)--(2,0);
\draw [dashed] (-0.5,0)--(12.2,0);
\draw [dashed] (0,2.2)--(0,-1.5);
\draw [dashed] (2,0)--(2,-1.5);
\draw [dashed] (2,0)--(6.4,2.2);
\draw (2,0) node {\scriptsize{$\cdot$}};
\draw (12,0) node {\scriptsize{$\cdot$}};
\draw (2,0.3) node {\scriptsize{$\kappa_P$}};
\draw (-0.3,0.3) node {\scriptsize{$O$}};
\draw (2,-0.3) node {\scriptsize{$2\omega$}};
\draw (12,-0.3) node {\scriptsize{$12\omega$}};

\draw [thick] (0,-1)--(10,1.667)--(12,2)--(0,-1);
\draw (8.273,1.129) node {\scriptsize{$\times$}};

\draw (5,-2) node {\scriptsize{$x=x_e$}};
\end{tikzpicture}

\begin{tikzpicture}[scale=0.5]
\fill[color=gray!20] (2,0)--(12,0)--(12,-2)--(2,-2)--(2,0);
\draw [dashed] (-0.5,0)--(12.2,0);
\draw [dashed] (0,1)--(0,-2);
\draw [dashed] (2,0)--(2,-2);
\draw (2,0) node {\scriptsize{$\cdot$}};
\draw (12,0) node {\scriptsize{$\cdot$}};
\draw (2,0.3) node {\scriptsize{$\kappa_P$}};
\draw (-0.3,0.3) node {\scriptsize{$O$}};
\draw (2,-0.3) node {\scriptsize{$2\omega$}};
\draw (12,-0.3) node {\scriptsize{$12\omega$}};

\draw [thick] (0,0)--(10,-1)--(12,-1)--(0,0);
\draw (8.273,-0.75) node {\scriptsize{$\times$}};

\draw (5,-2.5) node {\scriptsize{$x=x_v$}};
\end{tikzpicture}
\begin{tikzpicture}[scale=0.5]
\fill[color=gray!20] (2,0)--(12,0)--(12,-2)--(2,-2)--(2,0);
\draw [dashed] (-0.5,0)--(12.2,0);
\draw [dashed] (0,1)--(0,-2);
\draw [dashed] (2,0)--(2,-2);
\draw (2,0) node {\scriptsize{$\cdot$}};
\draw (12,0) node {\scriptsize{$\cdot$}};
\draw (2,0.3) node {\scriptsize{$\kappa_P$}};
\draw (-0.3,0.3) node {\scriptsize{$O$}};
\draw (2,-0.3) node {\scriptsize{$2\omega$}};
\draw (12,-0.3) node {\scriptsize{$12\omega$}};

\draw [thick] (0,0)--(12,-2)--(10,-1.5)--(0,0);
\draw (8.273,-1.318) node {\scriptsize{$\times$}};

\draw (5,-2.5) node {\scriptsize{$x=x_f$}};
\end{tikzpicture}
\end{center}
Fig-4: The polytopes $\Delta^O_x(K_X^{-1})$.\\ {\footnotesize (In the figure the dark area is $\kappa_P+(-\mathscr V)^\vee$, and the barycenter is marked by $\times$)}
\end{figure}

It is direct to check that \eqref{K-ss-eq} holds and
$$(\kappa_P-\mathbf{b}(\Delta_{x}^O(K_X^{-1})))^\perp\cap \mathscr V_x=\{O\}$$
for all $x$. This implies that $X$ is not ${\rm SL}_2\times{\rm k}^\times$-spherical. Otherwise the $\{\rm Id\}\times{\rm k}^\times$-action would induce a non-trivial, $G$-equivariant product test configuration with zero Futaki invariant that is associated to some non-central $v_0\in\mathscr V$. A contradiction. Thus by Theorem \ref{K-st-thm}, $X$ is ${\mathbf G}$-uniform K-stable.

\subsection{A completion of space of ordered triangles in $\mathbb P^2$}
In this section we apply our results to study the equivariant uniform K-stability of a $\mathbb Q$-Fano quasihomogeneous ${\rm SL}_3$-variety given by \cite[Example 16.24]{Timashev-book}. In the following we use the notations and conventions in \cite{Timashev-book}. Let $G={\rm SL}_3$ and $T\subset G$ the torus of diagonal matrixes.  Then $G/T$ is quasihomogeneous, which is referred as the space of ordered triangles in $\mathbb P^2$. Choose $B$ the group of upper triangular matrixes. Combinatorial data, in particular, the colours $\mathscr D^B$, the lattice $\Gamma$ and the embedding $e$ of $\Gamma$ into $K^{(B)}$, and the valuation cone $\mathscr V$ of $G/H$ are established in \cite[Section 5.14]{Timashev-1997} (see also \cite{Timashev-book}). We recall some of them for our later use.

As usual, for $g=(g_{ij})\in G$, we choose $\epsilon_i(g)=g_{ii}$, $\omega_1 =\epsilon_1$, $\omega_2 =\epsilon_1+\epsilon_2$ are the fundamental weights, and $\alpha_1 =\epsilon_1-\epsilon_2$, $\alpha_2 =\epsilon_2-\epsilon_3$ are the simple roots. Denote by $\omega^\vee_i$ the fundamental coweights. 
Then for the weight lattice $\Gamma$ we have $$\Gamma=\mathbb Z\alpha_1\oplus\mathbb Z\alpha_2.$$

The data of colours are listed as in Table-2.1 (according to the notation of \cite[Example 16.24]{Timashev-book}).
\begin{table}[h]
Table-2.1: Colours in $G/H$\\
\begin{tabular}{|c|c|c|}
\hline
Colour                                                 & $x_D$                                            & $v_D$                                      \\ \hline
$D_x,~x\in\mathbb P^1\setminus\{x_1,x_2,x_3,\infty\}$  & $x$                                              & $q_x$                                      \\ \hline
$D_\infty$                                             & $\infty$                           & $q_\infty-(\omega_1^\vee+\omega_2^\vee)$   \\ \hline
$D_{i},~i=1,2,3$                                     & $x_i$                                            & $q_{x_i}+\omega_2^\vee$                    \\ \hline
$\tilde D_{i},~i=1,2,3$                              & $x_i$                                            & $q_{x_i}+\omega_1^\vee$                    \\ \hline
\end{tabular}
\end{table}
\\
Clearly, the associated parabolic subgroup
$$P(G/H)={\rm Stab}_G(D_2)\cap{\rm Stab}_G(\tilde D_2)=B.$$
Thus $\kappa_P=2(\alpha_1+\alpha_2)$.

The valuation cone of $G/H$ is given as following: an element $a_1\omega_1^\vee+a_2\omega_2^\vee+hq_x$ is contained in $\mathscr V_x$ if and only if
\begin{align*}
(a_1,a_2,h)~\text{satisfies}~\left\{
\begin{aligned}
&a_1,a_2\leq0\leq h,~\text{if}~x=x_1, x_2,x_3,\\
&a_1,a_2\leq-2h\leq0,~\text{if}~x=\infty,\\
&a_1,a_2\leq-h\leq0,~\text{otherwise}.
\end{aligned}
\right.
\end{align*}

We have the following completion of $G/H$,
$$X = \{(p_1, p_2, p_3, l_1, l_2, l_3)|p_j\in \mathbb P^2,~l_i \in \mathbb P^{2^*},~p_j \in l_i~\text{whenever}~i\not=j\}.$$
Here $G$ acts on $\mathbb P^2$ via the tautological representation on ${\rm k}^3$ and on the dual space $\mathbb P^{2^*}$ via the corresponding dual representation. Let us recall the list of $G$-orbits in $X$ given in \cite[Section 16.5]{Timashev-book},
\begin{itemize}
\item The open $G$-orbit of non-degenerate triangles.
\item The divisors $W_i$, $i=1,2,3$ which consists of degenerate triangles with $p_j=p_k$ and $l_j=l_k$, where $\{i,j,k\}=\{1,2,3\}$.
\item The divisor $\tilde W$ which consists of degenerate triangles with $p_1$, $p_2$, $p_3$ collinear and $l_1=l_2=l_3$.
\item The divisor $W$ which consists of degenerate triangles with $p_1=p_2=p_3$ and $l_1$, $l_2$, $l_3$ pass through this point.
\item Codimension 2 orbits $\tilde Y_i$ which consists of degenerate triangles with $p_j=p_k$ and $l_1=l_2=l_3$, where $\{i,j,k\}=\{1,2,3\}$.
\item Codimension 2 orbits $Y_i$ which consists of degenerate triangles with $l_j=l_k$ and $p_1=p_2=p_3$, where $\{i,j,k\}=\{1,2,3\}$. It has codimension 3.
\end{itemize}
and
\begin{itemize}
\item A minimal $G$-germ (the closed orbit) $Y$ which consists of degenerate triangles with $p_1=p_2=p_3$ and $l_1=l_2=l_3$.
\end{itemize}
Thus $X = G\mathring X$, where $\mathring X$ is the minimal $B$-chart of $Y$ determined by the colored hypercone $\mathscr C=\mathscr C(\mathscr W,\mathscr R)$ with
\begin{align*}
\mathscr W=\{W,\tilde W, W_1,W_2,W_3\},~
\mathscr R=\{D_x|x\not=x_1,x_2,x_3\},
\end{align*}
where the $G$-valuations in $\mathscr W$ are given in Table-2.2.
\begin{table}[h]
Table-2.2: $G$-valuations in $\mathscr W$\\
\begin{tabular}{|c|c|c|}
\hline
Valuation                 & $x_D$                                            & $v_D$                                      \\ \hline
$W$                       & Central                                          & $-\omega_2^\vee$                           \\ \hline
$\tilde W$                & Central                                          & $-\omega_1^\vee$                           \\ \hline
$W_i,~i=1,2,3$            & $x_i$                                            & $q_{x_i}$                                  \\ \hline
\end{tabular}
\end{table}

There is a $G$-equivariant small resolution of $X$ (cf. \cite[Section 1]{Collino-Fulton} and \cite[Remark 11]{Timashev-2000}). By ``small" we mean the exceptional locus has codimension at least 2 (it is showed in \cite[Section 1]{Collino-Fulton} that in case of $X$ the exceptional locus has dimension 4). A small resolution is in particular crepant. This implies that $X$ has klt singularities (cf. \cite{Kollar-lecture}).

By Theorem \ref{anti-can-div-thm-quasi-homo}, there is a $B$-stable anti-canonical $\mathbb Q$-divisor
$$\mathfrak d=W+\tilde W+\frac23\sum_{i=1}^3(D_i+\tilde D_i+W_i),$$
where the corresponding $\mathfrak a=\frac23([x_1]+[x_2]+[x_3])$. Choose the function $f=f_0e_{-\alpha_1-\alpha_2}\in K^{(B)}_{-\alpha_1-\alpha_2}$, where $f_0\in K^B$ with ${\rm div}(f_0)=\frac23([x_1]+[x_2]+[x_3])-2[\infty]$ on $\mathbb P^1$, it is direct to check by Theorem \ref{ampleness-criterion} that $3\mathfrak d$ is an ample divisor on $X$, which is the divisor of a section in ${\rm H}^0(X,K_X^{-3})^{(B)}_{3\kappa_P}$.
It is showed that $\Gamma=\mathbb Z\alpha_1\oplus\mathbb Z\alpha_2$. Choose a coordinate $\lambda=\lambda_1\alpha_1+\lambda_2\alpha_2$ for $\lambda\in\Gamma_\mathbb Q$. We have
$$\Delta_\mathscr Z(\mathfrak d)=\{\lambda|2+2\lambda_1-\lambda_2\geq0,~2+2\lambda_2-\lambda_1\geq0,~2-\lambda_1-\lambda_2\geq0,~1-\lambda_1\geq0,~1-\lambda_2\geq0\}.$$

\begin{figure}[h]
\begin{center}
\begin{tikzpicture}[scale=1.2]
\draw [dashed] (0,0) -- (1.2,0);
\draw [semithick] (1,1)--(-0.5,1)--(-1,0)--(-2,-2)--(1,-0.5)--(1,1);
\draw [dashed] (0,0)--(0,1.2);
\draw [dashed] (0,0)--(-2.2,-2.2);

\draw (0.2,-0.2) node {\scriptsize{$O$}};
\draw (1.4,1) node {\scriptsize{$(1,1)$}};
\draw (-1,1) node {\scriptsize{$(-\frac12,1)$}};
\draw (1.4,-0.5) node {\scriptsize{$(1,-\frac12)$}};
\draw (-1.6,-2.2) node {\scriptsize{$(-2,-2)$}};
\draw (-1.6,0.2) node {\scriptsize{$A(\mathfrak d,\lambda)=2+2\lambda_1-\lambda_2$}};
\draw (0.2,-1.6) node {\scriptsize{$A(\mathfrak d,\lambda)=2+2\lambda_2-\lambda_1$}};
\draw (1.2,0.6) node {\scriptsize{$A(\mathfrak d,\lambda)=2-\lambda_1-\lambda_2$}};
\end{tikzpicture}
\end{center}
Fig-4: The polytope $\Delta_\mathscr Z(\mathfrak d)$ and domains of linearity of $A(\mathfrak d,\lambda)$.
\end{figure}

The functions $\{A_x(\mathfrak d,\lambda):\Delta_\mathscr Z(\mathfrak d)|x\in C\}$, and barycenter of $\Delta_x^O(K_X^{-1})$ are given in Table-2.3. The concave piecewise linear function $A(\mathfrak d,\lambda)$ is marked in Fig-4.
\begin{table}[h]
Table-2.3: Data of $K_X^{-1}$\\
\begin{tabular}{|c|c|c|}
\hline
$x\in\mathbb P^1$                                 & $A_x(\mathfrak d,\lambda)$           & $\mathbf{b}(\Delta_x^O(K_X^{-1}))-\kappa_P=(\lambda_1,\lambda_2,t)$   \\ \hline
$x\in\mathbb P^1\setminus\{x_1,x_2,x_3,\infty\}$  & $0$                                  & $(\frac{16141}{76706},\frac{16141}{76706},-\frac{12279}{27395})$ \\ \hline
$\infty\in\mathbb P^1$                            & $-\lambda_1-\lambda_2$               & $(\frac{16141}{76706},\frac{16141}{76706},-\frac{5248}{191 765})$   \\ \hline
$x_i,~i=1,2,3$            & $\min\{\frac23+\lambda_1,\frac23+\lambda_2,\frac23\}$        & $(\frac{16141}{76706},\frac{16141}{76706},-\frac{166658}{575 295})$   \\ \hline
\end{tabular}
\end{table}\\
It is direct to check that \eqref{K-ss-eq} holds and
$$(\kappa_P-\mathbf{b}(\Delta_{x}^O(K_X^{-1})))^\perp\cap \mathscr V_x=\{O\},~\forall x\in \mathbb P^1,$$
which in particular implies that $X$ can not be $G\times{\rm k}^\times$-spherical again. 
By Theorem \ref{K-st-thm}, we have
\begin{prop}
The $\mathbb Q$-Fano variety $X$ is ${\mathbf G}$-uniformly K-stable.
\end{prop}

\section{Appendix}

\subsection{Lemmas on colours in spherical homogeneous spaces}
The following lemma is a combination of several known results (cf. \cite[Section 30.10]{Timashev-book} and \cite[Remark 11.8]{Gandini-2018}), we include it here for readers' convenience.
\begin{lem}\label{spherical-quot}
Let $G/H$ be a spherical homogeneous space and $\sigma\in {\rm Aut}_G(G/H)$ a $G$-equivariant isomorphism. Suppose that two different colours $D$ and $D'$ satisfies $D=\sigma(D')$. Then both $D$ and $D'$ are of type-a, and there is a simple root $\alpha\in\Pi_G$ so that $\mathscr D^B(G/H;\alpha)=\{D,D'\}$. Moreover, $v_D=v_{D'}$ in the hyperspace of $G/H$.
\end{lem}

\begin{proof}
It is obvious that if $D\in\mathscr D^B(G/H;\alpha)$, then so is $D'$. Thus $\#\mathscr D^B(G/H;\alpha)\geq2$. This is possible only if both $D$ and $D'$ are of type-a (cf. \cite[Section 30.10]{Timashev-book}), and $\mathscr D^B(G/H;\alpha)=\{D,D'\}$. For the last point it suffices to show that for any $e_\lambda\in{\rm k}(G/H)^{(B)}_\lambda$,
$${\rm ord}_D(e_\lambda)={\rm ord}_{D'}(e_\lambda).$$
Since $\sigma$ commutes with the $G$-action, $\sigma\cdot e_\lambda\in{\rm k}(G/H)^{(B)}_\lambda$. Hence $\sigma\cdot e_\lambda=c e_\lambda$ for some constant $c\not=0$, and
$${\rm ord}_D(e_\lambda)={\rm ord}_{\sigma(D')}(e_\lambda)={\rm ord}_{D'}(\sigma^{-1}\cdot e_\lambda)={\rm ord}_{D'}(e_\lambda),$$
which concludes the Lemma.
\end{proof}

\begin{lem}\label{bdry-coe-spherical}
Let $(X,L)$ be a polarized $G$-spherical variety and $\mathfrak d=\sum_{D\in\mathscr B(X)}m_DD$ a divisor of $L$, which is the divisor of some $B$-semiinvariant rational section $s$ of $L$ with $B$-weight $\lambda_0$. 
Suppose that $D\in\mathscr D^B(X;\alpha)$ is a colour of type-a' or b that corresponds to a simple root $\alpha\in\Pi_G$. Then $\alpha^\vee(\lambda_0)=m_D$, and for any $\lambda\in\Gamma_\mathbb Q$ so that $v_D(\lambda)+m_D=0$, it holds
$$\langle\alpha^\vee,\lambda+\lambda_0\rangle=0.$$
\end{lem}

\begin{proof}
Let $G/H$ be the spherical homogeneous space that is embedded in $X$. Then for each $D\in\mathscr D^B$, there is a line bundle $L_D$ on $G/H$ so that $D={\rm div}(s_D)$ for some $s_D\in{\rm H}^0(G/H,L_D)^{(B)}_{\lambda_D}$, where $\lambda_D$ is the $B$-weight of $s_D$. Since $(\mathfrak d-\sum_{D\in\mathscr D^B}m_DD)|_{G/H}=0$, we have
$$f:=\frac{s|_{G/H}}{\prod_{D\in\mathscr D^B}s_D^{m_D}}\in\mathscr O(G/H)^\times.$$
By \cite[Proposition 1.3]{KKV}, $G$-acts on $f$ through a character $\lambda_G\in\mathfrak X(G)$. Consequently,
$$\lambda_0=\sum_{D\in\mathscr D^B}m_D\lambda_D+\lambda_G.$$
Let $\alpha$ be a simple root in $\Pi_G$. It holds
\begin{align*}
\langle\alpha^\vee,\lambda_0\rangle=\sum_{D\in\mathscr D^B}m_D\langle\alpha^\vee,\lambda_D\rangle
=\sum_{D\in\mathscr D^B(X;\alpha)}m_D\langle\alpha^\vee,\lambda_D\rangle.
\end{align*}

When $\alpha$ is of type-a', $\mathscr D^B(X;\alpha)$ contains precisely one colour (denoted by $D$), and by \cite[Section 2.2, Theorem 2.2]{Fo} (see also \cite[Proposition 2.2]{Ga-Ho-datum} or \cite[Lemma 30.24]{Timashev-book}),
$$\langle\alpha^\vee,\lambda_0\rangle=m_D\langle\alpha^\vee,\lambda_D\rangle=2m_D.$$
Combining with the fact that $v_D=\frac12\alpha^\vee|_\Gamma$ for the type-a' colour $D$, we get the Lemma.

When $\alpha$ is of type-b, again $\mathscr D^B(X;\alpha)$ contains only one colour (denoted by $D$), similarly as above,
$$\langle\alpha^\vee,\lambda_0\rangle=m_D\langle\alpha^\vee,\lambda_D\rangle=m_D.$$
Combining with the fact that $v_D=\alpha^\vee|_\Gamma$ for the type-b colour $D$, we get the Lemma.
\end{proof}

\subsubsection{A combinatorial property of $B$-stable ample divisors}

Let $(X,L)$ be a polarized $G$-variety of complexity 1, and $$\mathfrak d=\sum_{D\in\mathscr B(X)}m_DD$$ a divisor of $L$, which is the divisor of some $B$-semiinvariant rational section $s$ of $L$ with $B$-weight $\lambda_0$. We have
\begin{lem}\label{face-vanishing}
Suppose that
\begin{itemize}
\item[(1)] $X$ is a one-parameter $G$-variety and $D\in\mathscr D^B(\alpha)$ a colour of type-a' or b,
\end{itemize}
 or
\begin{itemize}
\item[(2)] $X$ is a quasi-homogeneous $G$-variety and $D$ a central colour that descends to a type-a' or b colour in $\mathscr D^B(\alpha)$ on $Z'$ in the proof of Theorem \ref{anti-can-div-thm-quasi-homo}.
\end{itemize}
Then
$$\langle\alpha^\vee,\lambda+\lambda_0\rangle=0,$$
whenever $\lambda\in\Gamma_\mathbb Q$ satisfies $v_D(\lambda)+m_D=0$.
\end{lem}
\begin{proof}
We start with case (1) when $X$ is a one-parameter $G$-variety. In this case a general $G$-orbit in $X$ is isomorphic to some spherical homogeneous space $O$. For simplicity we denote by $O$ any fixed such an orbit. Recall that $\Gamma\cong\Gamma(O)$. Let $D\in\mathscr D^B$ be a colour of $X$ of type-a' or b. Then from the construction in Section 3.1.1 and Lemma \ref{spherical-quot}, $D\cap O=\hat D$ is a single colour of $O$ which has the same type with $D$, and satisfies $v_D=v_{\hat D}$. By restricting $L$ and $s$ on $O$, we get $s|_O$ a $B$-semiinvariant section of $L|_O$ with the same weight $\lambda_0$ and $\mathfrak d|_O$ its divisor. The statement then follows from Lemma \ref{bdry-coe-spherical}.

Now we turn to case (2). Assume that $X$ is quasihomogeneous which contains a homogeneous space $G/H$. Then for any colour $D$, $D\cap(G/H)$ is the divisor of a $B$-semiinvariant section $s_D\in{\rm H}^0(G/H,L_D)^{(B)}_{\lambda_D}$ for some $\lambda_D\in \mathfrak X(B)$. Recall the set $C^o$ constructed in Section 3.2 above. We claim that $$\lambda_{X_z}\equiv\lambda^o~\text{(modulo a $G$-character)}$$ for some $\lambda^o\in\mathfrak X(B)$ for any $X_z$ with $z\in C^o$. Note for any $z_1,z_2\in\mathring C$, the difference of the colours $X_{z_1}-X_{z_2}$ coincides with the divisor of a $B$-invariant function $f_0:=\frac{z-z_1}{z-z_2}\in{\rm k}(\mathbb P^1)\cong{\rm k}(X)^B$. This implies
$$f':=\frac {s_{X_{z_1}}}{f_0s_{X_{z_2}}}\in\mathscr O(G/H)^\times.$$
By \cite[Proposition 1.3]{KKV}, $G$ acts on $f'$ through some character $\mu$, whence
$$\lambda_{X_{z_1}}=\lambda_{X_{z_2}}+\mu,$$
and we get the claim.

By restricting $s$ on $G/H$, we get
$$f=\frac{s|_{G/H}}{\prod_{D\in\mathscr D^B}s_D^{m_D}}\in\mathscr O(G/H)^\times.$$
Again by \cite[Proposition 1.3]{KKV}, we can decompose
$$\lambda_0=m\lambda^o+\lambda_0'+\lambda_0''+\lambda_G,$$
where $m=\sum_{z\in C^o}m_{X_z}\in\mathbb N_+$, $\lambda_0'=\sum_{D\in\mathscr D^B\setminus\mathscr B(X)^{P'}}m_D\lambda_D$ (every $D$ appears in this sum descends to a colour of  $Z'$), $\lambda_0''=\sum_{D\in(\mathscr D^B)^{P'},x_D\not\in C^o}m_D\lambda_D$ (every $D$ in this sum descends to a $L_{P'}$-stable divisor in $Z'$), and $\lambda_G\in\mathfrak X(G)$. In particular, $\lambda^o$ and $\lambda_0''$ are $L_{P'}$-characters, since $P'$ stabilizes $X_z$ for $z\in C^o$ and any $D\in(\mathscr D^B)^{P'}$. Thus,
$$\langle\alpha^\vee,\lambda_0\rangle=\langle\alpha^\vee,\lambda_0'\rangle,~\forall\alpha\in\Pi_{L_{P'}}.$$
On the other hand, for a colour $D$ that descends to a colour $D'$ of type-a' or b in $Z'$, $v_D=\frac12\alpha^\vee|_{\Gamma}$ or $\alpha^\vee|_\Gamma$ for $\alpha\in\Pi_{L_{P'}}$. The Lemma then follows from the previous case.

\end{proof}

\subsection{Lemmas on one-parameter $G$-varieties of type  \uppercase\expandafter{\romannumeral1}}

\begin{lem}\label{orbit-C-lem}
Let $X$ be a one-parameter $G$-variety of type \uppercase\expandafter{\romannumeral1} and $\mathcal O$ any $G$-orbit in it. Then the coloured cone $(\mathscr C,\mathscr R)$ of  $\mathcal O$ can not be totally contained in $\mathscr Q$.
\end{lem}
\begin{proof}
Note that any $G$-orbit in general position is spherical. By \cite[Proposition 1]{Arzhantsev-1997} any $G$-orbit is spherical. Also $X_\mathcal O:=\overline{\mathcal O}$ is normal by \cite[Theorem 16.25]{Timashev-book}, whence a spherical variety. In particular $X_\mathcal O$ contains only finitely many $G$-orbits by Akhiezer's theorem \cite{Akhiezer-1985}. Thus there are only finitely many coloured cones in the fan $\mathfrak F_X$ of $X$ whose relative interior intersects $\mathscr V$ and contains $(\mathscr C,\mathscr R)$ as a face.

Denote by $C$ the smooth projective curve so that ${\rm k}(X)^G={\rm k}(C)$. If $(\mathscr C,\mathscr R)\subset\mathscr Q$, then for almost every $x\in C$, $({\rm Cone}(q_x,\mathscr C),\mathscr R)$ is a coloured cone in $\mathfrak F_X$. Also ${\rm RelInt}(\mathscr C)\cap\mathscr V\not=\emptyset$ since $(\mathscr C,\mathscr R)$ is the coloured cone of a $G$-subvariety, which implies ${\rm RelInt}({\rm Cone}(q_x,\mathscr C))\cap\mathscr V\not=\emptyset$ for those $x\in C$ above. Thus $X_\mathcal O$ contains infinitely many $G$-orbits. A contradiction.
\end{proof}

The following Lemma is a special case of a general result \cite[Proposition 12.12]{Timashev-book}. We include an elementary proof below in our case only for readers' convenience.

\begin{lem}\label{X-C-map-lem}
Let $X$ be a one-parameter $G$-variety of type \uppercase\expandafter{\romannumeral1} and $C$ the smooth projective curve so that ${\rm k}(X)^G={\rm k}(C)$. Then the rational quotient ${\rm pr}_B:X\dashrightarrow C$ is a morphism ${\rm pr}_B:X\to C$ separating general $G$-orbits.
\end{lem}
\begin{proof}
Suppose that $\mathring X\subset X$ is any $B$-chart given by the coloured data $(\mathscr W,\mathscr R)$. Then there is an open subset $X^o\subset\mathring X$ where ${\rm pr}_B$ is defined. By definition, for any $x\in C$, ${\rm pr}_B^{-1}(x)\cap X^o$ is the union of all $D\cap X^o$, where $D\in\mathscr W\sqcup\mathscr R$ so that $x_D=x$ and $h_D>0$. We hope to show under our assumptions, ${\rm pr}_B$ in fact can be extended to the whole $X$.

By Lemma \ref{orbit-C-lem}, for any $G$-orbit $\mathcal O\subset X$, there is a unique $x\in C$ so that the coloured cone of $X_\mathcal O\subset\mathscr Q_{x,+}$. Thus there is a $G$-invariant map
$${\rm Pr}: X\to C$$
that maps any point in $\mathcal O$ to $x\in C$. This is a map globally defined on $X$, and ${\rm Pr}|_{X^o}={\rm pr}_B|_{X^o}$. It suffices to show that ${\rm Pr}$ is regular on the whole $X$.

Suppose that $(\mathscr C,\mathscr R)$ is a hypercone of type \uppercase\expandafter{\romannumeral1} so that each $(\mathscr C_x,\mathscr R_x)$, $x\in C$ is a coloured cone in $\mathscr F_X$. As in \cite{Petersen-Suss-2011, Langlois} we denote the locus of $(\mathscr C,\mathscr R)$ by ${\rm Loc}(\mathscr C,\mathscr R):=\{x\in C|(\mathscr C_x,\mathscr R_x)\not\subset\mathscr Q\}$. Denote by $U(\mathscr C,\mathscr R)$ the corresponding $B$-chart defined by $(\mathscr C,\mathscr R)$. Note that $h_D\geq0$ for any $D\in\mathscr B(X)$. For any affine subset $\mathring C\subset C$, ${\rm k}[\mathring C]\subset{\rm k}(C)={\rm k}(X)^G$ is contained in ${\rm k}[U(\mathscr C,\mathscr R)]$ if and only if ${\rm Loc}(\mathscr C,\mathscr R)\subset \mathring C$. On the other hand, since every $f\in{\rm k}[\mathring C]$ is $G$-invariant, $f$ is regular on the $G$-span of $U(\mathscr C,\mathscr R)$ if and only if $f\in{\rm k}[U(\mathscr C,\mathscr R)]$. Then ${\rm Pr}$ is regular on the $G$-span of every $U(\mathscr C,\mathscr R)$ so that ${\rm Loc}(\mathscr C,\mathscr R)\subset \mathring C$, and is a globally defined morphism.
\end{proof}

\subsection{Lemmas on counting integral points}
\begin{lem}\label{sum-neg}
Suppose there are $m$ numbers $\alpha_1,...,\alpha_m\in\mathbb R$ so that
\begin{align}\label{sum-cond-si}
\sum_{i=1}^m\alpha_i\geq0>\sum_{i=1}^m[\alpha_i].
\end{align}
Then
\begin{align}\label{sum-cond-estimate}
m>\sum_{i=1}^m\alpha_i\geq0>\sum_{i=1}^m[\alpha_i]>-m.
\end{align}
\end{lem}

\begin{proof}
Note that $$\alpha_i=[\alpha_i]+\{\alpha_i\},~i=1,...,m.$$
From the right-hand side inequality of \eqref{sum-cond-si} we have
$$m>\sum_{i=1}^m\{\alpha_i\}>\sum_{i=1}^m\alpha_i\geq0.$$
At the same time, the left-hand side inequality of \eqref{sum-cond-si} yields
$$\sum_{i=1}^m[\alpha_i]\geq-\sum_{i=1}^m\{\alpha_i\}>-m.$$
Hence we get the Lemma.
\end{proof}

From the above Lemma we get the following estimates
\begin{lem}\label{bad-point-lem}
Let $\mathfrak d_0$ be an ample divisor given by \eqref{B-stable-divisor}. Set
$$\mathfrak T_k:=\{\lambda\in\Delta_\mathscr Z(\mathfrak d_0)\cap\frac 1k\Gamma|\sum_{x\in C}[kA_x(\mathfrak d_0,\lambda)]<0\},~k\in\mathbb N_+.$$
Then there is a $k_0\in\mathbb N_+$ depends only on $\mathfrak d_0$, and constants $c,c'>0$ so that the cardinal number $\#\mathfrak T_k\leq ck^{r-1}$ for any $k\in\mathbb N_{\geq k_0}$, and
$$-c'<\sum_{x\in C}[kA_x(\mathfrak d_0,\lambda)]<kA(\mathfrak d_0,\lambda)<c', ~\forall k\in\mathbb N_{\geq k_0}~\text{and}~\lambda\in\mathfrak T_k.$$
\end{lem}
\begin{proof}
Note that there are only finitely many points $x_1,...,x_m\in C$ so that $A_{x_i}(\mathfrak d_0,\lambda)\not\equiv0$ for $i=1,...,m$. Take $\alpha_i=kA_{x_i}(\mathfrak d_0,\lambda)$. By Lemma \ref{sum-neg},
\begin{align}\label{neg-set}
\mathfrak T_k\subset\{\lambda\in\Delta_\mathscr Z(\mathfrak d_0)\cap\frac 1k\Gamma|A(\mathfrak d_0,\lambda)<\frac mk\}.
\end{align}
On the other hand, the piecewise linear function $A(\mathfrak d_0,\lambda)>0$ on $\Delta_\mathscr Z(\mathfrak d_0)$. Thus there is a $k_0\in\mathbb N_+$ that depends only on $\mathfrak d_0$ (more precisely, the integer $m$, the function $A(\mathfrak d_0,\lambda)$, and the shape of $\Delta_\mathscr Z(\mathfrak d_0)$ which are completely determined by $\mathfrak d_0$) such that the right-hand side of \eqref{neg-set} is contained in a strip near the boundary of $\Delta_\mathscr Z(\mathfrak d_0)$ with facets parallel to that of the boundary. It is also direct to check that for $k\geq k_0$, this strip can be chosen so that its width $\leq\frac {c_0m}k$, where $c_0>0$ is again a constant that depends only on the function $A(\mathfrak d_0,\lambda)$ and the shape of $\Delta_\mathscr Z(\mathfrak d_0)$. In particular, $c_0$ is independent of $k\in\mathbb N_+$. Thus $\#\mathfrak T_k\leq ck^{r-1}$ for some uniform $c>0$. The last point follows from \eqref{sum-cond-estimate} by taking $c'=m$.
\end{proof}

The following Lemma can be derived from a general result of Pukhlikov-Khovanskij \cite{Pukhlikov-Khovanskii}. We include it here for reads' convenience.
\begin{lem}\label{riemann-sum}
Let $\mathfrak M\cong \mathbb Z^r$ be a lattice and $\Delta\subset \mathfrak M_\mathbb R$ be a solid, integral convex polytope in it. Let $\pi:\Delta\to\mathbb R$ be a monomial of degree $d$ on $\mathfrak M_\mathbb R$ and $f:\Delta\to\mathbb R$ a concave, piecewise linear function whose domains of linearity $\{\Omega_a\}_{a=1}^{N_f}$ consist of integral polytopes in $\Delta$. Suppose that on each domain of linearity $\Omega_a$,
$$f(\lambda)=\frac1{p_a}(q_a(\lambda)+r_a),$$
where $(p_a,-q_a)\in\mathbb Z\oplus\mathbb Z$ is a primitive vector, $r_a\in\mathbb Q$, and $f$ takes integral value at every vertex f its domains of linearity. Define
$$S_k(f;\pi):=\sum_{\lambda\in k\Delta\cap\mathfrak M}[kf(\frac\lambda k)]\pi(\lambda),~k\in\mathbb N_+.$$
Then
\begin{align*}
S_k(f;\pi)=&k^{r+d+1}\int_{\Delta}f(\lambda)\pi(\lambda)d\lambda+\frac12k^{r+d}\int_{\partial\Delta}f(\lambda)\pi(\lambda)d\sigma\notag\\
&-\frac12k^{r+d}\sum_{a=1}^{N_f}\int_{\Omega_a}(1-\frac1{|p_a|})\pi(\lambda)d\lambda+O(k^{r+d-1}),~k\to+\infty.
\end{align*}
where $d\sigma$ is the induced lattice measure on $\partial\Delta$.
\end{lem}

\begin{proof}
Since $f$ is concave, $\min_\Delta f$ is attained at some vertex of $\Delta$. Hence $\min_\Delta f\in\mathbb Z$. Consider a convex polytope
$$\Delta_m:=\{(t,\lambda)|\lambda\in\Delta,~t\in\mathbb R,~\min_\Delta f\leq t\leq f(\lambda)\}.$$
Then $\Delta_m$ is an $(r+1)$-dimensional convex integral polytope in $\mathfrak M_\mathbb R\oplus\mathbb R$. Clearly,
\begin{align*}
[kf(\frac\lambda k)]=&[k(f(\frac\lambda k)-\min_\Delta f)]+k\min_\Delta f\notag\\
=&\#\{(\{\frac1k\lambda\}\times\frac1k\mathbb Z)\cap\Delta_m\}+k\min_\Delta f-1.
\end{align*}
Thus
\begin{align}\label{1st-term1-end-lemma-app}
\sum_{\lambda\in k\Delta\cap\mathfrak M}[kf(\frac\lambda k)]\pi(\lambda)=&\sum_{(\lambda,t)\in\Delta_m\cap\frac1k(\mathfrak M\oplus\mathbb Z)}\pi(k\lambda)+(k\min_\Delta f-1)\sum_{\lambda\in\Delta\cap\frac1k\mathfrak M}\pi(k\lambda)\notag\\
=&k^d\sum_{(\lambda,t)\in\Delta_m\cap\frac1k(\mathfrak M\oplus\mathbb Z)}\pi(\lambda)+k^d(k\min_\Delta f-1)\sum_{\lambda\in\Delta\cap\frac1k\mathfrak M}\pi(\lambda).
\end{align}
By \cite{Pukhlikov-Khovanskii}, the first term
\begin{align}\label{1st-term1-lemma-app}
\sum_{(\lambda,t)\in\Delta_m\cap\frac1k(\mathfrak M\oplus\mathbb Z)}\pi(\lambda)=&k^{r+1}\int_{\Delta_m}\pi(\lambda)dt\wedge d\lambda+\frac12k^r\int_{\partial\Delta_m}\pi(\lambda)d\bar\sigma+O(k^{r-1}),~k\to+\infty,
\end{align}
where $d\bar\sigma$ is the induced lattice measure on $\partial\Delta_m$. We have
\begin{align}\label{main-term-lemma-app}
\int_{\Delta_m}\pi(\lambda)dt\wedge d\lambda=\int_{\Delta}(f(\lambda)-\min_\Delta f)\pi(\lambda)d\lambda.
\end{align}
The boundary $\partial\Delta_m$ consists of three parts: On $F_1:=(\mathbb R\times\partial\Delta)\cap\Delta_m$,
\begin{align}\label{bdry-term1-lemma-app}
\int_{F_1}\pi(\lambda)d\bar\sigma=\int_{\partial\Delta}(f(\lambda)-\min_\Delta f)\pi(\lambda)d\sigma,
\end{align}
where $d\sigma$ is the induced lattice measure on $\partial\Delta$. On $F_2=\{\min_\Delta f\}\times \Delta$,
\begin{align}\label{bdry-term2-lemma-app}
\int_{F_2}\pi(\lambda)d\bar\sigma=\int_{\Delta}\pi(\lambda)d\bar\sigma.
\end{align}
On $F_3=\{\text{graph of}~f\}$, since $f$ is rational, on each domain of linearity $\Omega$ where $$f(\lambda)=\frac 1p(q(\lambda)+r')$$
for primitive vector $(p,-q)\in\mathbb Z\oplus\mathbb M^*$ and $r'\in\mathbb Q$,
\begin{align}\label{bdry-term3-lemma-app}
\int_{\text{graph $f$ on}~\Omega}\pi(\lambda)d\bar\sigma=\int_{\Omega}\pi(\lambda)\frac1{|(p,-q)|}d\sigma_0
=\frac1{|p|}\int_{\Omega}\pi(\lambda)d\lambda.
\end{align}
Here $d\sigma_0=\sqrt{1+\frac{|q|^2}{p^2}}d\lambda$ is the standard induced Lebesgue measure. Plugging \eqref{main-term-lemma-app}-\eqref{bdry-term3-lemma-app} into \eqref{1st-term1-lemma-app}, we get
\begin{align*}
\sum_{(\lambda,t)\in\Delta_m\cap\frac1k(\mathfrak M\oplus\mathbb Z)}\pi(\lambda)=&k^{r+1}\int_{\Delta}(f(\lambda)-\min_\Delta f)\pi(\lambda)d\lambda+\frac12k^r\int_{\partial\Delta}(f(\lambda)-\min_\Delta f)\pi(\lambda)d\sigma\notag\\
&+\frac12k^r\sum_{a=1}^{N_f}\int_{\Omega_a}(1-\frac1{|p_a|})\pi(\lambda)d\lambda+O(k^{r-1}),~k\to+\infty.
\end{align*}
Similarly,
\begin{align*}
\sum_{\lambda\in\Delta\cap\frac1k\mathfrak M}\pi(\lambda)=&k^{r}\int_{\Delta}\pi(\lambda)d\lambda+\frac12k^{r-1}\int_{\partial\Delta}\pi(\lambda)d\sigma+O(k^{r-2}),~k\to+\infty.
\end{align*}
Plugging the above two relations into \eqref{1st-term1-end-lemma-app} we get the Lemma.
\end{proof}

\subsection{An alternative proof to Theorem \ref{Fut-one-para-thm}}
In the following we give an alternative proof of via an intersection formula of the Futaki invariant (cf. \cite{Li-Xu-Annl,Boucksom-Hisamoto-Jonsson}). Given a normal test configuration $(\mathcal X,\mathcal L)$ of $(X,L)$, the Futaki invariant can also be interpreted as intersection numbers
\begin{align}\label{Fut-def-int}
{\rm Fut}(\mathcal X,\mathcal L)=\frac{1}VK_{\mathcal X/\mathbb P^1}\cdot\mathcal L^{\cdot n}+\frac{\bar S}{V(n+1)}\mathcal L^{\cdot(n+1)},
\end{align}
where $V=L^n$ is the volume of $(X,L)$, $\bar S$
the mean value of the scalar curvature of $(X,L)$, and
\begin{align*}
K_{\mathcal X/\mathbb P^1}=K_{\mathcal X}-{\rm pr}^*K_{\mathbb P^1}
\end{align*}
a Weil divisor on $\mathcal X$.

Assume that $(X,L)$ is a polarized spherical variety with $L$ has a divisor $\mathfrak d_0$ given by \eqref{B-stable-divisor}, and
$$\mathfrak d_K=\sum_{D\in\mathscr B(X)}\bar m_DD$$
a $B$-stable anti-canonical divisor of $X$ corresponds to weight $\kappa_P$. Let $(\mathcal X,\mathcal L)$ be a normal test configuration associated to $v_0\in\mathcal Q_{x_0,+}$ and $m=-1$. Without loss of generality we may also assume that $\mathcal L$ has a divisor $\mathfrak D$ given by \eqref{div-mathcal-L} with $r_0=1$ and $m_\infty=0$ (otherwise one needs to divide \eqref{Fut-def-int} by $r_0^{n+1}$). Then by Theorems \ref{anti-can-div-thm} (in the one-parameter case) and \ref{anti-can-div-thm-quasi-homo} (in the quasihomogeneous case), it is direct to see
\begin{align*}
\mathfrak d_{\mathcal X/\mathbb P^1}=-\sum_{D\in\mathscr B(X)}\bar m_D\overline D-h_0(a_{x_0}-1)\mathcal X_0.
\end{align*}
In particular, if $X$ is $\mathbb Q$-Fano and $K=K_X^{-1}$, we can take $\mathfrak d_K=\mathfrak d_0$ (in particular, $\bar m_D=m_D$ for all $D\in\mathscr B(X)$) and
$$\mathfrak d_{\mathcal X/\mathbb P^1}=-(\mathfrak D+(h_0(a_{x_0}-1)-m_0){\rm pr}^*([0]))$$
is also a $\mathbb Q$-Cartier divisor that corresponds to the weight $-\kappa_P$.

The line bundle
\begin{align*}
\mathcal L_\epsilon:=\mathcal L+\epsilon K_{\mathcal X/\mathbb P^1}
\end{align*}
is a $\mathbb Q$-line bundle on $\mathcal X$, and is ample when $0<\epsilon\ll1$. The divisor
\begin{align*}
\mathfrak D_\epsilon:=\mathfrak D+\epsilon \mathfrak d_{\mathcal X/\mathbb P^1}=\sum_{D\in\mathscr B(X)}(m_D-\epsilon m_D)\overline D+(m_0-\epsilon(a_{x_0}-1))\mathcal X_0
\end{align*}
is a divisor of $\mathcal L_\epsilon$ corresponds to weight $(1-\epsilon)\kappa_P$, and the associated function is
\begin{align*}
\tilde A_{x}^\epsilon(\lambda,t):=\left\{\begin{aligned}&(1-\epsilon)\min\{A_{x_0}(\mathfrak d,\frac\lambda{1-\epsilon}),\frac{-t+m_0-\epsilon h_0(a_{x_0}-1)+\ell_0(\lambda)}{h_0(1-\epsilon)}\},~\text{when}~x=x_0,\\
&(1-\epsilon)A_{x}(\mathfrak d,\frac{\lambda}{1-\epsilon}),~\text{when}~x\not=x_0,\end{aligned}\right.
\end{align*}
for $\lambda\in\Gamma_\mathbb R$ and $\tilde A^\epsilon(\lambda,t)=\sum_{x\in C}\tilde A_{x}^\epsilon(\lambda,t)$. The polytope
\begin{align}\label{Delta-eps}
\tilde\Delta_\mathscr Z^\epsilon(\mathfrak D_\epsilon)=&\{(\lambda,t)|\tilde A^\epsilon(\lambda,t)\geq0\}\notag\\
=&\{(\lambda,t)|\lambda\in(1-\epsilon)\Delta_\mathscr Z(\mathfrak d_0),~0\leq t\leq\tau^0_\epsilon(\lambda)\},
\end{align}
where
\begin{align*}
\tau^0_\epsilon(\lambda)=m_0-\epsilon h_0(a_{x_0}-1)+\ell_0(\lambda)+h_0(1-\epsilon)(A(\frac\lambda{1-\epsilon})-A_{x_0}(\frac\lambda{1-\epsilon})).
\end{align*}
Set
\begin{align*}
\tilde\tau^0_\epsilon(\lambda)=m_0-\epsilon h_0(a_{x_0}-1)+\ell_0(\lambda)-h_0(1-\epsilon)A_{x_0}(\frac\lambda{1-\epsilon}).
\end{align*}
Then for $\lambda\in\Gamma_\mathbb R$,
\begin{align}\label{A-x-eps-piece}
\tilde A^\epsilon(\lambda,t):=\left\{\begin{aligned}
&(1-\epsilon)A(\mathfrak d_0,\frac{\lambda}{1-\epsilon}),~\text{when}~0\leq t\leq\tilde\tau_0^\epsilon(\lambda),\\
&(1-\epsilon)\sum_{x\not=x_0}A_{x}(\mathfrak d_0,\frac{\lambda}{1-\epsilon})+\frac{m_0-\epsilon h_0(a_{x_0}-1)-t+\ell_0(\lambda)}{h_0},~\text{when}~\tilde\tau^0_\epsilon(\lambda)\leq t\leq\tau^0_\epsilon(\lambda).\end{aligned}\right.
\end{align}

Note that when $\epsilon\in\mathbb Q$,
$$\mathcal L_\epsilon^{\cdot(n+1)}=\mathcal L^{\cdot(n+1)}+(n+1)\epsilon K_{\mathcal X/\mathbb P^1}\cdot \mathcal L^{\cdot n}+O(\epsilon^2),~\epsilon\to0^+.$$
Apply the intersection formula \cite[Theorem 8]{Timashev-2000} to $\mathcal L_\epsilon$, we get
\begin{align}\label{mathcal-L-n+1}
\mathcal L^{\cdot(n+1)}=&(n+1)!\int_{\tilde\Delta_\mathscr Z(\mathfrak D)}\tilde A(\lambda,\tau)\pi(\lambda+\kappa_P) d\lambda\wedge d\tau\notag\\
=&(n+1)!\int_{\Delta_\mathscr Z(\mathfrak d_0)}\int_{0}^{\tau^0(\lambda)}\tilde A(\lambda,\tau)\pi(\lambda+\kappa_P) d\lambda\wedge d\tau,
\end{align}
and
\begin{align}\label{KX/P-term}
K_{\mathcal X/\mathbb P^1}\cdot \mathcal L^{\cdot n}&=\frac1{n+1}\left.\frac d{d\epsilon}\right|_{\epsilon=0}\mathcal L_\epsilon^{\cdot(n+1)}\notag\\
&=n!\left.\frac d{d\epsilon}\right|_{\epsilon=0}\int_{\tilde\Delta_\mathscr Z^\epsilon(\mathfrak D_\epsilon)}\tilde A^\epsilon(\lambda,\tau)\pi(\lambda+(1-\epsilon)\kappa_P) d\lambda\wedge d\tau.
\end{align}
By \eqref{Delta-eps} and \eqref{A-x-eps-piece},
\begin{align*}
&\int_{\tilde\Delta_\mathscr Z^\epsilon(\mathfrak D_\epsilon)}\tilde A^\epsilon(\lambda,\tau)\pi(\lambda+(1-\epsilon)\kappa_P) d\lambda\wedge d\tau\notag\\
=&\int_{(1-\epsilon)\Delta_\mathscr Z(\mathfrak d_0)}\int_{0}^{\tilde\tau^0_\epsilon(\lambda)}(1-\epsilon)A(\mathfrak d_0,\frac{\lambda}{1-\epsilon})\pi(\lambda+(1-\epsilon)\kappa_P) d\lambda\wedge dt\notag\\
&+\int_{(1-\epsilon)\Delta_\mathscr Z(\mathfrak d_0)}\int_{\tilde\tau^0_\epsilon(\lambda)}^{\tau^0_\epsilon(\lambda)}(1-\epsilon)\sum_{x\not=x_0}A_x(\mathfrak d_0,\frac{\lambda}{1-\epsilon})\pi(\lambda+(1-\epsilon)\kappa_P) d\lambda\wedge dt\\
&+\int_{(1-\epsilon)\Delta_\mathscr Z(\mathfrak d_0)}\int_{\tilde\tau^0_\epsilon(\lambda)}^{\tau^0_\epsilon(\lambda)}\frac{m_0-\epsilon h_0(a_{x_0}-1)-t+\ell_0(\lambda)}{h_0}\pi(\lambda+(1-\epsilon)\kappa_P) d\lambda\wedge dt.
\end{align*}
Using change of variables $\lambda\to(1-\epsilon)\lambda$, $t\to(1-\epsilon)t$, and the relation \eqref{n=c+r+1}, we can rewrite the first two terms as
\begin{align*}
&\int_{(1-\epsilon)\Delta_\mathscr Z(\mathfrak d_0)}\int_{0}^{\tilde\tau^0_\epsilon(\lambda)}(1-\epsilon)A(\mathfrak d_0,\frac{\lambda}{1-\epsilon})\pi(\lambda+(1-\epsilon)\kappa_P) d\lambda\wedge dt\notag\\
&+\int_{(1-\epsilon)\Delta_\mathscr Z(\mathfrak d_0)}\int_{\tilde\tau^0_\epsilon(\lambda)}^{\tau^0_\epsilon(\lambda)}(1-\epsilon)\sum_{x\not=x_0}A_x(\mathfrak d_0,\frac{\lambda}{1-\epsilon})\pi(\lambda+(1-\epsilon)\kappa_P) d\lambda\wedge dt\notag\\
=&(1-\epsilon)^{n+1}\int_{\Delta_\mathscr Z(\mathfrak d_0)}\int_{0}^{\frac1{1-\epsilon}\tilde\tau^0_\epsilon((1-\epsilon)\lambda)}A(\mathfrak d_0,\lambda)\pi(\lambda+\kappa_P) d\lambda\wedge dt\notag\\
&+(1-\epsilon)^{n+1}\int_{\Delta_\mathscr Z(\mathfrak d_0)}\int_{\frac1{1-\epsilon}\tilde\tau^0_\epsilon((1-\epsilon)\lambda)}^{\frac1{1-\epsilon}\tau^0_\epsilon((1-\epsilon)\lambda)}\sum_{x\not=x_0}A_x(\mathfrak d_0,\lambda)\pi(\lambda+\kappa_P) d\lambda\wedge dt,
\end{align*}
and the last term can be rewritten as
\begin{align*}
&\int_{(1-\epsilon)\Delta_\mathscr Z(\mathfrak d_0)}\int_{\tilde\tau^0_\epsilon((1-\epsilon)\lambda)}^{\tau^0_\epsilon((1-\epsilon)\lambda)}\frac{m_0-\epsilon h_0(a_{x_0}-1)-t+\ell_0(\lambda)}{h_0}\pi(\lambda+(1-\epsilon)\kappa_P) d\lambda\wedge dt\\
=&(1-\epsilon)^{n+1}\int_{\Delta_\mathscr Z(\mathfrak d_0)}\int_{\frac1{1-\epsilon}\tilde\tau^0_\epsilon((1-\epsilon)\lambda)}^{\frac1{1-\epsilon}\tau^0_\epsilon((1-\epsilon)\lambda)}\frac{\frac1{1-\epsilon}({m_0-\epsilon h_0(a_{x_0}-1)})-t+\ell_0(\lambda)}{h_0}\pi(\lambda+\kappa_P) d\lambda\wedge dt
\end{align*}
Taking variation and using Lemma \ref{bdry-measure}, we get
\begin{align*}
&\left.\frac d{d\epsilon}\right|_{\epsilon=0}\int_{\tilde\Delta_\mathscr Z^\epsilon(\mathfrak D_\epsilon)}\tilde A^\epsilon(\lambda,\tau)\pi(\lambda+(1-\epsilon)\kappa_P) d\lambda\wedge d\tau\notag\\
=&-(n+1)\int_{\tilde\Delta_\mathscr Z(\mathfrak D)}\tilde A(\mathfrak D,\lambda,\tau)\pi(\lambda+\kappa_P) d\lambda\wedge d\tau
+\int_{\Delta_\mathscr Z(\mathfrak d_0)}\int_{\tilde\tau^0(\lambda)}^{\tau^0(\lambda)}(\frac{m_0}{h_0}-a_{x_0}+1)\pi(\lambda+\kappa_P)d\lambda\wedge dt.
\end{align*}
Combining with \eqref{Fut-def-int}, \eqref{mathcal-L-n+1}-\eqref{KX/P-term}, we get
\begin{align*}
{\rm Fut}(\mathcal X,\mathcal L)=&\frac1V\int_{\Delta_\mathscr Z(\mathfrak d_0)}\int_{\tilde\tau^0(\lambda)}^{\tau^0(\lambda)}\frac{m_0-h_0(a_{x_0}-1)}{h_0}\pi(\lambda+\kappa_P)d\lambda\wedge dt\\
&-\frac1V\int_{\Delta_\mathscr Z(\mathfrak d_0)}\int_{0}^{\tau^0(\lambda)}\tilde A(\mathfrak D,\lambda,t)\pi(\lambda+\kappa_P)d\lambda\wedge dt\notag\\
=&\frac1V\int_{\Delta_{x_0}^O(K_X^{-1})}\langle(\kappa_P-\lambda,-t),v_0\rangle\pi(\lambda)d\lambda\wedge dt,
\end{align*}
where in the last line we used \eqref{E-NA} and the fact that
\begin{align*}
V=\frac{(K_X^{-1})^{\cdot n}}{n!}=\int_{\Delta_\mathscr X(\mathfrak d_0)}A(\mathfrak d_0,\lambda)\pi(\lambda+\kappa_P)d\lambda=\int_{\Delta_{x_0}^O(K_X^{-1})}\pi(\lambda)d\lambda\wedge dt.
\end{align*}

\subsection{One the polytope $\Delta_{x_0}^O(K_X^{-1})$}
The polytope $\Delta_{x_0}^O(K_X^{-1})$ has another geometric meaning
\begin{prop}\label{polytope-centre-lem}
Suppose that $(\mathcal X,\mathcal L)$ is a $G$-equivariant special test configuration of $(X,K_X^{-1})$ associated to $(v_0,-1)$. Assume that $v_0=\ell_0+h_0q_{x_0}$ with $h_0\not=0$. Then $(\mathcal X_0,\mathcal L_0)$ is a polarized $G\times{\rm k}^\times$-spherical variety, and the associated moment polytope is $\Delta_{x_0}^O(K_X^{-1})-(0,a_{x_0}-1)$.
\end{prop}
\begin{proof}
Recall that for any $k\in\mathbb N$,
$${\rm H}^0(\mathcal X_0,\mathcal L_0^k)^{(B\times{\rm k}^\times)}_{(\lambda,\tau)}\cong(\mathscr F_{(\mathcal X,\mathcal L)}^\tau R_k/\mathscr F_{(\mathcal X,\mathcal L)}^{>\tau}R_k)^{(B)}_\lambda\subset{\rm Gr}(\mathscr F_{(\mathcal X,\mathcal L)}).$$
We decompose ${\rm H}^0(\mathcal X_0,\mathcal L_0^k)$ into irreducible $G\times{\rm k}^\times$-modules using Proposition \ref{F-tau-Rk}.
From Proposition \ref{F-tau-Rk} (2) and \eqref{pt-disconti-t-GrR} we see that ${\rm H}^0(\mathcal X_0,\mathcal L_0^k)^{(B\times{\rm k}^\times)}_{(\lambda,\tau)}\not=0$ if and only if
$$\deg(\delta_k(\lambda,\tau))\geq0,~\text{where $\delta_k(\lambda,\tau)$ is defined by \eqref{delta(d,lambda)}},$$
and
$$\deg(\delta_k(\lambda,\tau))-\deg(\delta_k(\lambda,\tau+1))>0.$$

The above conditions require that $(\lambda,\tau)$ satisfies
\begin{align*}
\left\{\begin{aligned}&\tau\leq k\tau^0(\frac\lambda k-\kappa_P),~\text{where $\tau^0$ is defined by \eqref{A-sum-of-D}},   \\&A_{x_0}(\mathfrak d,\frac\lambda k-\kappa_P)\geq\frac{-\frac\tau k+m_0+\ell_0(\frac\lambda k-\kappa_P)}{h_0},  \end{aligned}\right.
\end{align*}
or equivalently $(\lambda,\tau)\in k{\Delta}_\mathscr Z^o(\mathcal L)$ (see \eqref{polytope-D-Zo(mathcal-L)} for definition). Conversely, by Lemma \ref{bad-point-lem}, for any $k\in\mathbb N_+$, all points at which the first condition may fail lie in a strip of uniform (to $k$) width along $k\partial{\Delta}_\mathscr Z^o(\mathcal L)$. Thus the closure of
\begin{align*}
\bigcup_{k=0}^{+\infty}\frac1k\{(\lambda,\tau)\in(\Gamma+k\kappa_P)\times\mathbb Z|{\rm H}^0(\mathcal X_0,\mathcal L_0^k)^{(B\times{\rm k}^\times)}_{(\lambda,\tau)}\not=0\}
\end{align*}
in $(\Gamma_\mathbb R+\kappa_P)\times\mathbb R$ is $\Delta_\mathscr Z^o(\mathcal L)$.

However, to get the moment polytope we need some normalization: Take a (unimodular and integral) transformation of $\sigma$ on $\Gamma\oplus\mathbb Z$,
$$\sigma:(\lambda,\tau)\to(\lambda,\tau-m_0-\ell_0(\lambda)),$$
which induce a transformation
$$(\lambda+\kappa_P,\tau)\to(\lambda+\kappa_P,\tau-m_0-\ell(\lambda))$$
on $(\Gamma+\kappa_P)\times\mathbb Z$. The polytope ${\Delta}_\mathscr Z^o(\mathcal L)\subset(\Gamma+\kappa_0)_\mathbb R\times\mathbb R$ is then transformed to
$$\{(\lambda,t)|-h_0A_{x_0}(\mathfrak d,\lambda-\kappa_P)\leq t\leq h_0A(\mathfrak d,\lambda-\kappa_P)-h_0A_{x_0}(\mathfrak d,\lambda-\kappa_P)\}\subset(\Gamma+\kappa_0)_\mathbb R\times\mathbb R.$$
On the other hand, from \eqref{pt-disconti-t-GrR} we see the second condition also requires
\begin{align*}
{-\tau +k m_0+\ell_0(\lambda-k\lambda_0)}\in{h_0}\mathbb Z.
\end{align*}
Hence the group generated by ${\rm k}^\times$-weights on $\cup_{k=0}^{+\infty}{\rm H}^0(\mathcal X_0,\mathcal L_0^k)$ is $h_0\mathbb Z$. By rescaling it to the standard $\mathbb Z$ we get the moment polytope of $(\mathcal X_0,\mathcal L_0)$ is $\Delta_{x_0}^O(K_X^{-1})-(0,a_{x_0}-1)$.
\end{proof}
We remark that Proposition \ref{polytope-centre-lem} for $T$-varieties of complexity 1 has been proved in \cite[Corollary 4.6]{Ilten-Suss-Duke}. Also the above method applies to $G$-equivariant special test configurations of a general polarized $(X,L)$ (not necessarily $L=K_X^{-1}$).

\vskip20pt

\subsection*{Acknowledgement} We sincerely thank Prof. D.A. Timash\"ev for kindly introduce us his book \cite{Timashev-book} and many helpful discussions. We also thank the referees for careful reading and valuable comments, which improve this paper a lot.


\end{document}